\documentclass[a4paper]{amsart}

\usepackage[utf8]{inputenc}
\usepackage[T1]{fontenc}
\usepackage{lmodern, enumerate}
\usepackage{amssymb,amsxtra,amscd,amsmath}
\usepackage{amsthm}
\usepackage{amsfonts}
\usepackage{comment}
\usepackage{tikz-cd}
\usepackage{bbm}
\usepackage[all]{xy}
\usepackage{xcolor}
\usepackage{nicefrac,mathtools}
\usepackage[inline]{enumitem}
\usepackage{microtype}
\usepackage[
  pdftitle={Approximation properties of Fell bundles over inverse semigroups},
  pdfauthor={Alcides Buss, Diego Martínez},
  pdfsubject={Mathematics},
  colorlinks=true, linkcolor=blue, linkbordercolor=blue, citecolor=red, citebordercolor=red, urlcolor=blue, linktocpage=true
]{hyperref}
\usepackage[capitalise, nameinlink, noabbrev, nosort]{cleveref} 
\usepackage[lite]{amsrefs}

\newcommand*{\arxiv}[1]{\href{http://www.arxiv.org/abs/#1}{arXiv: #1}}

\setlist[enumerate]{font=\normalfont}
\crefname{enumi}{}{} 
\crefname{enumi}{}{} 

\newcounter{thmintrocnt}
\newtheorem{thmintro}[thmintrocnt]{Theorem}

\numberwithin{equation}{section}
\theoremstyle{plain}
\newtheorem{theorem}[equation]{Theorem}
\newtheorem{lemma}[equation]{Lemma}
\newtheorem{proposition}[equation]{Proposition}
\newtheorem{corollary}[equation]{Corollary}
\theoremstyle{definition}
\newtheorem{definition}[equation]{Definition}
\theoremstyle{remark}
\newtheorem{remark}[equation]{Remark}

\newtheorem{example}[equation]{Example}

\DeclareMathOperator{\cspn}{\overline{span}}
\DeclareMathOperator{\spn}{{span}}
\DeclareMathOperator{\Ind}{Ind}

\DeclareMathOperator{\supp}{\mathrm{supp}}
\DeclareMathOperator{\id}{\mathrm{id}}

\DeclareMathOperator{\Bis}{\mathrm{Bis}}

\newcommand*{\Star}{\(^*\)\nobreakdash-}

\newcommand*{\N}{\mathbb N}
\newcommand*{\C}{\mathbb C}

\renewcommand*{\L}{\mathcal L}
\renewcommand*{\H}{\mathcal H}
\newcommand*{\Hilb}{\mathcal H} 
\newcommand*{\Qc}{\mathcal Q_c}
\newcommand*{\F}{\mathcal F} 
\newcommand*{\cont}{C}
\newcommand*{\contz}{\cont_0}
\newcommand*{\contc}{\cont_c}
\newcommand*{\M}{\mathcal M}

\newcommand*{\NN}{\mathcal N}

\newcommand*{\gammac}{\Gamma_c}

\newcommand*{\Id}{\textup{id}}

\newcommand*{\E}{\mathcal E}
\newcommand*{\congto}{\xrightarrow\sim}

\newcommand*{\braket}[2]{\langle#1\!\mid\!#2\rangle}

\newcommand*{\sbe}{\subseteq} 

\newcommand*{\cstar}{\texorpdfstring{$C^*$\nobreakdash-\hspace{0pt}}{*-}}
\newcommand*{\into}{\hookrightarrow}
\newcommand*{\onto}{\twoheadrightarrow}
\newcommand*{\red}{r}
\renewcommand*{\max}{\mathrm{max}}
\newcommand*{\alg}{\mathrm{alg}}

\newcommand*{\dual}[1]{\widehat{#1}}

\newcommand*{\I}{\mathcal{I}} 
\newcommand*{\J}{\mathcal{J}} 

\newcommand{\csp}{\overline{\operatorname{span}}}

\newcommand{\A}{\mathcal A}
\newcommand{\B}{\mathcal B}

\newcommand*{\s}{s} 
\newcommand*{\rg}{r}

\newcommand*{\algalg}[1]{\mathbb{C}_{\text{alg}}\left(#1\right)}

\newcommand*{\redalg}[1]{C_{\text{red}}^*\left(#1\right)} 
\newcommand*{\fullalg}[1]{C_{\text{max}}^*\left(#1\right)} 
\newcommand*{\essalg}[1]{C_{\text{ess}}^*\left(#1\right)} 

\newcommand*{\multiplieralg}[1]{\mathcal{M}\left(#1\right)} 

\newcommand*{\tensor}{\otimes}
\newcommand*{\algtensor}{\odot}
\newcommand*{\maxtensor}{\tensor_{\text{max}}}
\newcommand*{\mintensor}{\tensor_{\text{min}}}

\newcommand*{\Prob}[1]{\text{Prob}\left(#1\right)}

\newcommand{\idealin}{\mathrel{\triangleleft}} 

\begin{document}

\title[Approximation properties of Fell bundles over inverse semigroups]{Approximation properties of Fell bundles over inverse semigroups and non-Hausdorff groupoids}

\author[Alcides Buss]{Alcides Buss $^{1}$}
\address{Departamento de Matem\'atica, Universidade Federal de Santa Catarina, 88.040-900 Florian\'opolis-SC, Brazil}
\email{alcides.buss@ufsc.br}

\author[Diego Mart\'{i}nez]{Diego Mart\'{i}nez $^{2}$}
\address{Mathematisches Institut, WWU M\"{u}nster, Einsteinstr. 62, 48149 M\"{u}nster, Germany.}
\email{diego.martinez@uni-muenster.de}

\begin{abstract}
In this paper we study the nuclearity and weak containment property of reduced cross-sectional \cstar{}algebras of Fell bundles over inverse semigroups. In order to develop the theory, we first prove an analogue of Fell's absorption principle in the context of Fell bundles over inverse semigroups. In parallel, the approximation property of Exel can be reformulated in this context, and Fell's absorption principle can be used to prove that the approximation property, as defined here, implies that the full and reduced cross-sectional \cstar{}algebras are isomorphic via the left regular representation, i.e., the Fell bundle has the weak containment property. 

We then use this machinery to prove that a Fell bundle with the approximation property and nuclear unit fiber has a nuclear cross-sectional \cstar{}algebra. This result gives nuclearity of a large class of \cstar{}algebras as, remarkably, all the machinery in this paper works for \'{e}tale non-Hausdorff groupoids just as well.
\end{abstract}



\thanks{{$^{1}$} Partially supported by CNPq, CAPES - Brazil \& Humboldt Fundation and Germany’s Excellence Strategy - University of Münster - Germany.}

\thanks{{$^{2}$} Funded by the Deutsche Forschungsgemeinschaft (DFG, German Research Foundation) under Germany’s Excellence Strategy – EXC 2044 – 390685587, Mathematics Münster – Dynamics – Geometry – Structure; the Deutsche Forschungsgemeinschaft (DFG, German Research Foundation) – Project-ID 427320536 – SFB 1442, and ERC Advanced Grant 834267 - AMAREC}

\maketitle

\section{Introduction}
There is an abundant number of examples of \cstar{}algebras that have been studied in the almost 80 years of \cstar{}literature, along with numerous different properties. Of these properties, one of the most ubiquitous ones is \emph{nuclearity}. In general, a \cstar{}algebra $A$ is \emph{nuclear} if there is exactly one \cstar{}norm on $A \algtensor B$, the algebraic tensor product of $A$ and $B$, for any \cstar{}algebra $B$ (see \cite{Brown-Ozawa:Approximations} for many other equivalent characterizations). Since it was first introduced, nuclearity has played a crucial role in the field, for it provides a natural finite-dimensional approximation of $A$ in terms of finite-dimensional \cstar{}algebras (see, e.g.,~\cite{Brown-Ozawa:Approximations}*{Section~2.3}). Thus, as the class of \cstar{}algebras is vast, criteria that guarantees when a \cstar{}algebra is nuclear is desirable, as it allows to tackle many different problems by various means. The present paper has the goal of giving sufficient criteria for the nuclearity of a large class of \cstar{}algebras, namely those arising as \emph{cross-sectional} \cstar{}algebras from a \emph{Fell bundle} over an \emph{inverse semigroup}.

\emph{Inverse semigroups} were introduced independently by Wagner~\cites{Wagner1952,Wagner1953} and Preston~\cites{Preston1954-1,Preston1954-2,Preston1954-3} in the 50's. If a (discrete) group can be understood as the set of \emph{global} symmetries of a space, then an \emph{inverse semigroup} can be understood as an algebraic structure that captures the \emph{local} symmetries of a space~\cite{lledo-martinez-2021}. Following this idea, one can also define \emph{actions} of inverse semigroups on locally compact and Hausdorff spaces, which generalize group \emph{partial} actions~\cites{KellendonkLawson:PartialActions,lledo-martinez-2021}. In this setting, such an action gives rise to \emph{full} and \emph{reduced} crossed product \cstar{}algebras. One may even \emph{twist} the action, in the sense of~\cite{Exel:TwistedActions}, and construct \emph{full} and \emph{reduced twisted} crossed products. This class of \cstar{}algebras is, already, quite large, as it contains every \emph{classifiable} \cstar{}algebra~\cite{Li:Classifiable}, among many others. In this paper we shall go one step further, and study cross-sectional \cstar{}algebras arising from \emph{Fell bundles} over inverse semigroups.

A \emph{Fell bundle over an inverse semigroup} $S$ is a bundle $\A = (A_s)_{s\in S}$ of Banach spaces $A_s$ endowed with associative multiplication maps $A_s \times A_t \to A_{st}$, involution maps $A_s\to A_{s^*}$ and inclusion maps $A_s \into A_t$ for $s \leq t$ in $S$ satisfying certain natural axioms that are reminiscent of \cstar{}algebraic structures (see \cref{def:pre:fell-bundle} below). These objects were first introduced by Sieben in the unpublished paper~\cite{SiebenFellbundles}, and first formally introduced in~\cite{Exel:noncomm.cartan} by Exel. In fact, Exel introduced them in order to deal with \emph{noncommutative Cartan subalgebras} (see~\cite{Exel:noncomm.cartan} for precise definitions). Indeed, the Fell bundle data generalizes both twisted actions on spaces and actions on noncommutative \cstar{}algebras, and hence it is the right setting where one may study noncommutative Cartan subalgebras. We shall usually assume that the inverse semigroup $S$ has a unit $1 \in S$, as this amounts to no loss in generality. In this case $A_1$ is a distinguished \cstar{}algebra, and the fibres $A_s$ are Hilbert bimodules over $A_1$ with respect to the Fell bundle operations. In this setting, the Fell bundle $\A$ can be viewed as a (generalized) action of $S$ on $A_1$ by Hilbert bimodules; this is the point of view taken in \cite{Buss-Meyer:Actions_groupoids} and in the sequel.

As is well known~\cites{Lawson2012:duality,chung-martinez-szakacs-2022}, inverse semigroups are strongly connected to \emph{\'{e}tale groupoids}. In particular, every action of an inverse semigroup gives rise to an \'{e}tale groupoid that may be non-Hausdorff (see \cref{subsec:ap:ap-def} below). Likewise, Fell bundles over inverse semigroups are strongly related to Fell bundles over \'{e}tale groupoids, as defined by Kumjian in~\cite{Kumjian:Fell_bundles}. The connection between these structures is, in fact, so tight that Fell bundles over inverse semigroups may be interpreted as \emph{(twisted) \'etale groupoids with noncommutative unit spaces}, as is done in~\cite{BussExel:Fell.Bundle.and.Twisted.Groupoids}. This relationship is further detailed in \cref{subsec:ap:ap-def}, particularly in \cref{prop:fell-isg-vs-groupoids,thm:ap:groupoids-inv-smgps}.

Ever since it was known that amenable groups have nuclear reduced \cstar{}algebras, \emph{amenability} of different types of structures has received considerable attention from the community (see~\cites{Buss-Echterhoff-Willett:amenability,ExelNg:ApproximationProperty,Kranz:AP-groupooids-Fell-bundles,martinez-2022,Ozawa-Suzuki:Amenable_examples} and references therein for examples of \emph{amenable} actions). Amenability of Fell bundles has, in particular, also received attention, starting from the work of Exel~\cite{Exel:Amenability} and Exel-Ng~\cite{ExelNg:ApproximationProperty} where the notion of \emph{approximation property} for a Fell bundle $\A = (A_g)_{g \in G}$ over a group $G$ was first introduced. This should be viewed as amenability of the (generalized) action of $G$ on $A_1$ implemented by $\A$; indeed, whenever $\A$ comes from an ordinary action $\alpha$ of $G$ on $A_1$ the approximation property of $\A$ turns out to be equivalent to the amenability of $\alpha$ as defined by Anantharaman-Delaroche in~\cite{Anantharaman-Delaroche:Systemes}. Although Exel's approximation property and Anantharaman-Delaroche's amenability have been around for quite some time (around 35 years), only recently have they been proved to be equivalent for Fell bundles associated with ordinary group actions; see the recent papers \cites{Buss-Echterhoff-Willett:amenability,Abadie-Buss-Ferraro:Amenability} for further discussions on this.

Exel's approach in~\cite{Exel:Amenability}, in fact, shows that the approximation property of $\A$, together with the nuclearity of the unit fiber $A_1$, imply that $\fullalg{\A} \cong \redalg{\A}$ is nuclear. Note that $A_1$ needs to be nuclear if $\redalg{\A}$ is to be nuclear as well, since there always is a canonical conditional expectation $P \colon \redalg{\A} \to A_1$. Moreover, it was already known for some time, and proved in \cite{Ara-Exel-Katsura:Dynamical_systems}*{Theorem~6.4}, that if $\redalg{\A}$ is nuclear then $\fullalg{\A} \cong \redalg{\A}$, that is, the weak containment property for $\A$ holds. Only recently, see \cites{Buss-Echterhoff-Willett:amenability, Abadie-Buss-Ferraro:Amenability}, it was proved that the nuclearity of $\redalg{\A}$ (or equivalently of $\fullalg{\A}$) is equivalent to nuclearity of $A_1$ plus the approximation property of $\A$. It is, however, not yet known whether the weak containment property for $\A$ implies its approximation property.\footnote{\, This is known to be false in the more general setting of locally compact groups, as there are examples of locally compact groups $G$ and actions on noncommutative \cstar{}algebras for which the weak containment holds but the underlying Fell bundle does not have the approximation property (i.e., the action is not amenable), see \cites{Buss-Echterhoff-Willett:amenability}.} Lastly, Kranz recently~\cite{Kranz:AP-groupooids-Fell-bundles} generalized the notion of approximation property so as to cover Fell bundles over second countable locally compact Hausdorff \'{e}tale groupoids, and derived some of the expected results.

Our main goal in this paper is to extend the notion of \emph{approximation property} from Fell bundles over groups to Fell bundles over inverse semigroups. This is not trivial due to the subtlety implicitly contained in these objects. In fact, the most direct translation from groups or groupoids to inverse semigroups does not capture the correct properties. Nevertheless, our notion of approximation property (see \cref{def:fell-isg}) does extend that of groups and groupoids (see \cref{prop:fell-isg-vs-groupoids}). It is worth noting that our techniques actually allow us to generalize some of the main theorems of~\cite{Kranz:AP-groupooids-Fell-bundles}, as we do not require the groupoids to be Hausdorff or second-countable, nor do we need the Fell bundles to be saturated or separable (see \cref{cor:groupoids-final} below). In fact, we generalize the main result of~\cites{Exel:Amenability,Kranz:AP-groupooids-Fell-bundles} in the following way.
\begin{thmintro}[cf.\ \cref{thm:ap-nuclear}] \label{thmintro:ap-nuclear}
Let $\A = (A_s)_{s \in S}$ be a Fell bundle over an inverse semigroup with unit $1 \in S$. Suppose that $\A$ has the approximation property, as in \cref{def:fell-isg}, and that $A_1$ is nuclear. Then $\fullalg{\A} \cong \redalg{\A}$ is nuclear.
\end{thmintro}

In order to prove \cref{thmintro:ap-nuclear} we, in fact, need to prove a version of Fell's absorption principle (see~\cite{Exel:Partial_dynamical}*{Proposition~18.4}) for Fell bundles over inverse semigroups (see \cref{thmintro:fell-principle} below). In the case of groups, Fell's absorption principle states that a representation $\pi$ of $\A$, once tensored with the (left) regular representation $\lambda$ of the underlying group, yields a new representation $\pi \otimes \lambda$ whose extension to $\fullalg{\A}$ factors through $\redalg{\A}$. Again, for inverse semigroups this theorem is much more subtle, since the simple-minded tensor product $\pi \otimes \lambda$ does not necessarily yield a representation of $\A$, due to the relations one has to impose on those representations coming from the inclusions $A_s \into A_t$ for $s \leq t$. We are going to construct instead a different representation $\pi^\Lambda$ that will play the role of $\pi\otimes\lambda$. It will act on a Hilbert space $\ell^2_\pi(S,\H)$ that is a Hausdorff completion of a certain space of finitely supported sections $S\to \H$. Our version of Fell's absorption principle then reads as follows.

\begin{thmintro}[cf.\ \cref{thm:fell-principle,cor:fell-principle:red}] \label{thmintro:fell-principle}
Let $\A = (A_s)_{s \in S}$ be a Fell bundle over an inverse semigroup $S$, and let $\pi = (\pi_s)_{s \in S}$ be a representation of $\A$. Then $\pi^\Lambda$ is unitarily equivalent to the representation $\Ind\pi_1''=\Lambda\otimes_{\pi_1''}\id\colon \A\to \L(\ell^2(\A'')\otimes_{\pi_1''}\H)$ induced from the regular representation $\Lambda\colon \A\to \L(\ell^2(\A''))$. In particular $\pi^\Lambda$ extends to a representation of $\redalg{\A}$. Moreover, if $\pi_1''$ is faithful as a representation of $A_1''$, then $\pi^\Lambda$ yields a faithful representation 
$$\pi^\Lambda\colon \redalg{\A}\into \L(\ell^2_\pi(S,\H)).$$
\end{thmintro}

\cref{thmintro:fell-principle} above will allow us to describe induced representations from the regular representation, and hence is the main technical tool in order to prove \cref{thmintro:ap-nuclear}. We now briefly mention that in \cref{cor:ess-nuclearity} we also prove that the approximation property of $\A$ implies that the \emph{essential \cstar{}algebra} $\essalg{\A}$ is nuclear (see~\cite{Kwaniewski2019EssentialCP}*{Theorem~4.11} or \cref{thm:ess-construction} below). In fact, in an upcoming paper, the authors will use similar methods to those in this text in order to introduce an \emph{essential approximation property}, proving the same results we now show but in the context of $\essalg{\A}$ instead.

We end the introduction with the structure of the paper. In \cref{sec:preliminaries} we introduce and recall the necessary background needed throughout the rest of the paper. Later, in \cref{sec:regular-Fell-principle} we recall the construction of the \emph{reduced} and \emph{full} cross-sectional \cstar{}algebras and present and prove our version of Fell's absorption principle (see \cref{thmintro:fell-principle} above). In \cref{sec:ap} we introduce the \emph{approximation property} for Fell bundles over inverse semigroups. Moreover, we also relate it to its companion property for Fell bundles over \'{e}tale groupoids. \cref{sec:tensor} then deals with maximal and minimal tensor products of a Fell bundle $\A$ with a fixed \cstar{}algebra $B$, and we prove the expected compatibility properties. These are technical tools used in the final \cref{sec:nuclearity}, where we prove \cref{thmintro:ap-nuclear}.

\textbf{Conventions:} throughout the paper $S$ will denote a (discrete) inverse semigroup with unit $1 \in S$ (i.e., an inverse \emph{monoid}), and its set of idempotents will be denoted by $E$. Likewise, $\A = (A_s)_{s \in S}$ will be a Fell bundle over $S$. Moreover, $\L(\Hilb{})$ denotes the set of linear bounded operators on the Hilbert space $\Hilb{}$. Lastly, given a \cstar{}algebra $A$, its double dual, or von Neumann enveloping algebra, will be denoted by $A''$.

\textbf{Acknowledgements:} we would like to thank Julian Kranz for his useful comments on a previous version of this manuscript. We would also like to thank the anonymous referee for their comments.

\section{Preliminaries}\label{sec:preliminaries}
We recall here all the basic theory necessary for this paper, including inverse semigroups and groupoids (see \cref{subsec:pre:isg-grpds}); Fell bundles (see \cref{subsec:pre:fell-bundles}); and cross-sectional \cstar{}algebras and their representations (see \cref{subsec:pre:cross-sectional-alg,subsec:pre:fell-reps}).

\subsection{Inverse semigroups and groupoids} \label{subsec:pre:isg-grpds}
Recall that a \emph{semigroup} is a set $S$ equipped with a binary associative operation $S\times S\to S$, where $(s,t) \mapsto st$. As expected, the main difference between \emph{semigroups} and \emph{groups} is the lack of the inverse $s^{-1}$ of some elements, which means that the maps given by $\lambda_s \colon S \to S$, where $t \mapsto st$, may be non-injective. In particular, the map $\lambda_s$ may not define a \emph{bounded} operator on $\ell^2(S)$. The following notion, introduced independently by Wagner~\cites{Wagner1952,Wagner1953} and Preston~\cites{Preston1954-1,Preston1954-2,Preston1954-3}, will fix this caveat, albeit only \emph{locally} (see \cref{rem:isg-local-actions}). 
\begin{definition} \label{def:isg}
A semigroup $S$ is called an \emph{inverse semigroup} if for all $s \in S$ there is a unique $s^* \in S$ such that $ss^*s = s$ and $s^*ss^* = s^*$.
\end{definition}

Henceforth $S$ will denote a discrete inverse semigroup. We now briefly recall the inner structure of a general inverse semigroup, and set some notation. We refer the reader to \cites{Lawson:InverseSemigroups} for a more comprenhensive introduction to inverse semigroups. The set of \emph{idempotents of $S$}, denoted by $E$, is the set $E \coloneqq \{ e \in S \mid e^2 = e\}$. It is well known (see~\cite{Lawson:InverseSemigroups}*{Theorem~3}) that $E \sbe S$ forms a \emph{commutative} sub-semigroup, and hence we may define a partial order on $E$ by stating that $e \leq f$ if, and only if, $ef = e$. This endows $E$ with a meet semi-lattice structure. In fact, the partial order $\leq$ can be extended to the whole semigroup $S$ by stating that $s \leq t$ if, and only if, $te = s$ for some idempotent $e \in E$. It is also well known (and follows from the uniqueness of $s^*$ in \cref{def:isg}, see~\cite{Lawson:InverseSemigroups}) that $E$ is precisely the set $\{s^*s \mid s \in S\}$, and hence every idempotent $e \in E$ is self-adjoint. We say that an element $z \in S$ is \emph{a zero} if $zt = tz = z$ for every $t \in S$. If there is a zero element then it is clear it must be unique, and in such case we will denote it by $0 \in E \sbe S$. Likewise, a \emph{unit}, or an \emph{identity of $S$}, is an element $u \in S$ such that $ut = tu = t$ for every $t \in S$. Again, if there is a unit in $S$ then it must be unique, and we shall henceforth denote it by $1 \in E \sbe S$. In fact, a unit may be always adjoined to $S$ and for convenience we shall assume $S$ has a unit throughout.

\begin{remark} \label{rem:isg-local-actions}
We briefly highlight that, even though the map $\lambda_s \colon S \to S$ defined above is still not injective even if $S$ is an inverse semigroup, it will be injective when restricted to $\lambda_s \colon s^*s \cdot S \to ss^* \cdot S$. This observation leads to a theory of \emph{actions} of inverse semigroups (see, e.g.,~\cite{AraLledoMartinez-2020} and references therein), and will permeate throughout the present paper, although not explicitly. In fact, the canonical action given by $\lambda_s \colon s^*s \cdot S \to ss^* \cdot S$ is known as the \emph{Wagner-Preston representation}. Likewise, inverse semigroups are then realized by partial isometries of Hilbert spaces with commuting domains and ranges.
\end{remark}

We may now shift our point of view, and discuss \emph{groupoids}. However, observe that the main objects of study of the text (see \cref{def:pre:fell-bundle,def:fell-isg}) revolve around inverse semigroups, and so we will keep the discussions about groupoids brief. We refer the reader to~\cite{SimsNotes2020} and references therein for a more in-depth discussion, as well as natural examples of groupoids arising in the literature.
\begin{definition} \label{def:grpd}
A \emph{groupoid} $G$ is a small category with inverses, that is, it is a set $G$ together with a distinguished set $G^{(2)} \sbe G \times G$ and maps $\cdot \colon G^{(2)} \to G$ and $^{-1} \colon G \to G$ such that
\begin{enumerate}
\item $(g^{-1})^{-1} = g$;
\item if $(g, h), (h, k) \in G^{(2)}$ then $(g, hk), (gh, k) \in G^{(2)}$ and $(gh) k = g (hk)$;
\item $(g, g^{-1}) \in G^{(2)}$, and whenever $(g, h) \in G^{(2)}$ (resp. $(h, g) \in G^{(2)}$) we have  $g^{-1}g h = h$ (resp. $h = h gg^{-1}$),
\end{enumerate}
for all $g, h, k \in G$.
\end{definition}

We usually think of elements in a groupoid as \emph{arrows between units}. Observe that every groupoid has a distinguished subset $G^0 \sbe G$, whose elements are called \emph{units}, which are precisely those elements that, as arrows, form loops. In fact, the rest of the elements of the groupoid can be thought of as arrows $g \colon g^{-1}g \to gg^{-1}$. This way we may define maps $r, s \colon G \to G^{0}$ by $s(g) \coloneqq g^{-1}g$ and $r(g) \coloneqq gg^{-1}$, which we call the \emph{range} and \emph{source} map respectively. Henceforth, every groupoid here considered will be endowed with a topology for which the multiplication and inverse maps are continuous. Hence, we say $G$ is a \emph{locally compact groupoid} if it is locally compact as a topological space. Moreover, the notion of \emph{\'etale} groupoid will be discussed in \cref{subsec:ap:ap-def}, and thus we will not discuss it here any further.

It is by now well known that \emph{\'etale groupoids} and \emph{inverse semigroups} are dual to each other (see~\cite{Lawson2012:duality} for precise statements). On one hand, given an inverse semigroup $S$ one may construct the \emph{universal} groupoid associated to $S$ (see~\cite{Paterson1999}*{Chapter~4}). As is expected from the name, this groupoid encapsulates the representation theory of $S$, and is \emph{universal} in some suitable sense~\cite{Paterson1999}*{Theorems~4.4.1 and~4.4.2}. However, we will not use this construction in the present paper, and hence we will refrain from discussing it further.

On the other hand, given a locally compact \'{e}tale groupoid $G$, one may look at the set of open \emph{bisections} $\Bis(G)$, which carries a natural inverse semigroup structure (see \cref{subsec:ap:ap-def} below, particularly \cref{def:wide-isg}). We may then recover the groupoid $G$ as the \emph{groupoid of germs $G^0 \rtimes \Bis(G)$}, where $\Bis(G)$ acts on $G^0$ canonically (see \cref{prop:wide}). All in all, elements in an inverse semigroup may be seen as appropriately chosen open sets of a groupoid, and arrows in a groupoid may be seen as germs of actions of inverse semigroups on spaces.

\subsection{Fell bundles} \label{subsec:pre:fell-bundles}
This section introduces the main object of study of the paper, and fixes notation that will be of use throughout the paper. 
\begin{definition} \label{def:pre:fell-bundle}
Let $S$ be an inverse semigroup with unit $1 \in S$, and let $A$ be a \cstar{}algebra. A \emph{Fell bundle over $S$ with unit fibre $A$} consists of 
\begin{itemize}
  \item Hilbert $A$-$A$-bimodules $A_s$ for every $s \in S$;
  \item bimodule embeddings $\mu_{s, t} \colon A_s \otimes_A A_t \hookrightarrow A_{st}$ for all $s, t \in S$;
\end{itemize}
such that
\begin{enumerate}[label=(\roman*)]
  \item \label{item:def:pre:fell-bundle:unit} $A_1 = A$ (as a $A$-$A$-bimodule);
  \item \label{item:def:pre:fell-bundle:mu} $\mu_{s, 1} \colon A_s \otimes_A A \rightarrow A$ and $\mu_{1, s} \colon A \otimes_A A_s \rightarrow A_s$ are the canonical isomorphisms;
  \item \label{item:def:pre:fell-bundle:assoc} for all $s, t, r \in S$ the diagram
    \begin{center}
      \begin{tikzcd}[scale=50em]
        A_s \otimes_A A_t \otimes_A A_r \arrow{d}{\id_{A_s} \otimes \mu_{t, r}} \arrow{r}{\mu_{s, t} \otimes \id_{A_r}} &[2.2em] A_{st} \otimes_A A_r \arrow{d}{\mu_{st, r}} \\
        A_s \otimes_A A_{tr} \arrow{r}{\mu_{s, tr}} & A_{str}
      \end{tikzcd}
    \end{center}
   commutes, and the common composition is denoted by $$\mu_{s,t,r}\colon A_s \otimes_A A_t \otimes_A A_r\to A_{srt};$$
 \item \label{item:def:pre:fell-bundle:iso} and $\mu_{s,s^*,s} \colon A_s \otimes_A A_{s^*} \otimes_A A_s \to A_s$ is an isomorphism for every $s \in S$.
\end{enumerate}
\end{definition}

The following remarks are in order.
\begin{remark} \label{rem:fell-history}
Historically, the notion of Fell bundles has first been defined by Sieben in an unpublished manuscript \cite{SiebenFellbundles}. This definition was then published and used in the paper \cite{Exel:noncomm.cartan} by Exel. We use the above notion from~\cite{Buss-Meyer:Actions_groupoids} by the first-named author and Meyer, where it is shown that both definitions are equivalent. In~\cite{Buss-Meyer:Actions_groupoids}, in fact, a Fell bundle over $S$ with unit fibre $A$ is viewed as a (generalized) \emph{action of $S$ on $A$ by Hilbert bimodules}.
\end{remark}

\begin{remark} \label{rem:saturated-vs-semi}
The notion in \cref{def:pre:fell-bundle} is actually what is known as a \emph{non-saturated} Fell bundle in the literature, since the maps $\mu_{s,t} \colon A_s \otimes_A A_t \rightarrow A_{st}$ are assumed to be only embeddings of Hilbert $A$-$A$-bimodules.  \emph{Saturated} Fell bundles, on the other hand, assume that $\mu_{s,t}$ is, actually, an isomorphism. It is well known (see~\cite{BussExel:InverseSemigroupExpansions}) that given a non-saturated Fell bundle $\A = (A_s)_{s \in S}$ one can enlarge $S$ to an inverse semigroup $T \coloneqq \mathrm{Pr}(S)$ (with the same unit, the so-called \emph{prefix expansion of $S$}) and endow $\A$ with a natural saturated Fell bundle structure over $T$ in such a way that the full and reduced \cstar{}algebras of $\A$ over $S$ and $T$ are canonically isomorphic (see~\cite{BussExel:InverseSemigroupExpansions}*{Theorem~7.2}).
\end{remark}
\begin{remark}
We restrict ourselves to the case where $S$ has a unit for convenience of both the readers and the authors, but this is not necessary. In general, if $S$ is \emph{not} a monoid then a Fell bundle over $S$ is defined as in \cref{def:pre:fell-bundle}, but conditions \cref{item:def:pre:fell-bundle:unit,item:def:pre:fell-bundle:mu} are dropped. As $S$ is always assumed to have a unit $1 \in S$ throughout the paper and the \cstar{}algebra $A_1 = A$ is then already part of the structure of $\A$, we shall usually just say that $\A$ is a \emph{Fell bundle over $S$}, as opposed to a Fell bundle over $S$ with unit fibre $A$.
\end{remark}

The following theorem was proved in~\cite{Buss-Meyer:Actions_groupoids}*{Theorem~4.8}, and describes more structure inherently present in any Fell bundle (given as axioms in \cite{Exel:noncomm.cartan}). In fact, the following was only proved for \emph{saturated} Fell bundles (recall \cref{rem:saturated-vs-semi}), but the same ideas carry over to the non-saturated setting.
\begin{theorem} \label{thm:fell-bundle-extra-str}
  Let $\A = (A_s)_{s \in S}$ be a Fell bundle over $S$. Then there are
  \begin{itemize}
    \item unique Hilbert bimodule embeddings $\iota_{t, s} \colon A_t \rightarrow A_s$, where $t \leq s$, such that
    \begin{enumerate}[label=(\roman*)]
      \item for all $s_1, t_1, s_2, t_2 \in S$ such that $t_1 \leq s_1$ and $t_2 \leq s_2$ the diagram
        \begin{center}
          \begin{tikzcd}[scale=50em]
            A_{t_1} \otimes_A A_{t_2} \arrow{d}{\iota_{t_1, s_1} \otimes \iota_{t_2, s_2}} \arrow{r}{\mu_{t_1, t_2}} &[1.3em] A_{t_1 t_2} \arrow{d}{\iota_{t_1 t_2, s_1 s_2}} \\
            A_{s_1} \otimes_A A_{s_2} \arrow{r}{\mu_{s_1, s_2}} & A_{s_1 s_2}
          \end{tikzcd}
        \end{center}
        commutes;

      \item $\iota_{s, s} = \Id_{A_s}$ for every $s \in S$; and
      \item $\iota_{s, r} \circ \iota_{t, s} = \iota_{t, r}$ for every $t \leq s \leq r$;
    \end{enumerate}
    \item unique Hilbert bimodule isomorphisms $J_s \colon A_s^* \rightarrow A_{s^*}$ such that
    \[ \mu_{s,s^*,s}\left(a\otimes J_s\left(a\right)\otimes a\right) = a \cdot \braket{a}{a}_A = {}_A\braket{a}{a} \cdot a \]
    for all $s \in S$ and $a \in A_s$; here $A_s^*$ denotes the dual Hilbert bimodule of $A_s$.
  \end{itemize}
\end{theorem}
\begin{remark} \label{rem:thm:fell-bundle-extra-str}
  Henceforth, the embeddings $\mu_{s, t} \colon A_s \otimes_A A_t \rightarrow A_{st}$ are assumed to be part of the data of a Fell bundle. Likewise, so are the bimodule maps $\iota_{t, s}$ and $J_s$ from \cref{thm:fell-bundle-extra-str}. Thus, for convenience of the reader, a Fell bundle over $S$ will just be denoted by $\A = (A_s)_{s \in S}$. Furthermore, we shall use the following abbreviations:
  \begin{itemize}
    \item $a_s a_t \coloneqq \mu_{s, t}(a_s \otimes a_t)$ whenever $a_s \in A_s$ and $a_t \in A_t$;
    \item the map $\iota_{t, s}$, where $t \leq s$, will just be understood as an inclusion $A_t \sbe A_s$;
    \item $a_s^* \coloneqq J_s(a_s) \in A_{s^*}$ for every $s \in S$.
  \end{itemize}
  The previous notation should be enlightening. Indeed, $\mu_{s,t}$ will be the product of elements; $\iota_{t, s}$ is just a natural inclusion; and $J_s$ stands for taking the adjoint of an element. 
\end{remark}

Given a \cstar{}algebra $A$ we write $A''$ for the von Neumann enveloping algebra of $A$, that is, the von Neumann algebra generated by the image of its universal representation. The weak and strong topologies on $A''$ are the ones inherited from this representation. The ultraweak topology corresponds to the \Star{}weak topology of $A'' \cong (A')'$.

The following notion was already introduced in~\cites{Buss-Exel-Meyer:Reduced} and will be important to us also in this paper.
\begin{definition}
Given a Fell bundle $\A = (A_s)_{s \in S}$ over an inverse semigroup $S$, the \emph{von Neumann enveloping Fell bundle of $\A$}, denoted by $\A''$, is the Fell bundle whose fibers are the biduals $A''_s$ of $A_s$ and with multiplication maps $A_s''\overline{\otimes}_{A_1''} A_t''\to A_{st}''$ being the weakly continuous extensions of the multiplications on $\A$. 
\end{definition}

\begin{remark}
The Fell bundle $\A''$ is a \emph{von Neumann Fell bundle} or a \emph{$W^*$-Fell bundle} as defined in \cite{Abadie-Buss-Ferraro:Amenability}*{Definition~5.1} for Fell bundles over groups (and the same definition obviously extends to inverse semigroups), which essentially means the whole structure of the Fell bundle is compatible with the weak topology. One can look at the \emph{universal representation} of the Fell bundle $\A$ in order to represent every fibre $A_s$ with $s\in S$ into a single Hilbert space. Alternatively, one can view $A_s$ as a subspace of $\fullalg{\A}$ (the full cross-sectional \cstar{}algebra of $\A$, see~\cref{subsec:pre:cross-sectional-alg}) and embed everything into the enveloping von Neumann algebra $\fullalg{\A}''$. This allows us to endow the fibres $A_s$ with a weak topology, as we implicitly did above. And in this sense $A_s''$ is just the weak closure of $A_s$.
\end{remark}

We end this section with a discussion of some notation that will be used throughout the paper virtually without mention. Note that this will lead to the useful \cref{prop:unit-s-commutation}, which allows us to perform some commutation relations in certain computations. 

\begin{definition} \label{def:pre:ideal-s-1}
Let $\A = (A_s)_{s \in S}$ be a Fell bundle over an inverse semigroup $S$ with unit $1 \in S$. Given any $s, t \in S$, let:
\begin{enumerate}[label=(\roman*)]
  \item $I_{s, t} \sbe A$ be the (closed) ideal generated by $A_{v^*v} \sbe A_1$ with $v \leq s,t$;
  \item $1_{s, t} \in I_{s, t}''$ be the unit of $I_{s, t}'' \sbe A''$;
  \item $1_s \coloneqq 1_{s, 1} \in I_{s, 1}''$.
\end{enumerate}
\end{definition}
\begin{remark}
  Note that if $e \in E$ is an idempotent then $I_{e, 1} = A_e$, which is a two-sided closed ideal in $A_1$. Thus $1_e \in A_e''$ is the unit of the von Neumann algebra $A_e''$. We can also view $1_e$ as the unit of the multiplier \cstar{}algebra $\M(A_e)\sbe A_e''$. More generally we may view $1_s$ as the unit of the multiplier \cstar{}algebra $\M(I_{s,1})\sbe I_{s,1}''$. Notice that $I_{s,t}=I_{s^*t,1}=I_{t^*s,1}=I_{t,s}$ so that $1_{s,t}=1_{s^*t}=1_{t^*s}=1_{t,s}$ for all $s,t\in S$. This follows from the fact that if $e\leq s^*t,1$, then $v \coloneqq se = te\leq s,t$ (see the proof of~\cref{lemma:unit-s-conjugation} below) and, conversely, if $v\leq s,t$, then $e \coloneqq v^*v\leq s^*t,1$.
\end{remark}

The ideal $I_{s, 1}$ is intrinsic to the Fell bundle $\A$ and plays an important role in its structure. In a sense that will become clear in what follows, $I_{s,1}$ implements the \emph{trivial} part of the fibre $A_s$. The following results give certain important computation relations that will be used later.
\begin{lemma} \label{lemma:unit-s-conjugation}
  Let $S$ be an inverse semigroup with idempotent semilattice $E$. For $s\in S$, write $E_s:=\{e\in E:e\leq s\}$ and let $\omega_s\colon E\to E$ be the map $\omega_s(e) \coloneqq ses^*$.
  Then:
  \begin{enumerate}[label=(\roman*)]
    \item \label{lemma:unit-s-conjugation:doms}  $E_r=E_{r^*}$ for all $r\in S$, in particular $E_{s^*t}=E_{t^*s}$ for all $s,t\in S$;
    \item \label{lemma:unit-s-conjugation:maps} $\omega_s=\omega_t$ as maps $E_{s^*t} \to E_{st^*}$ and give a bijection with inverse $\omega_{s^*}=\omega_{t^*}$.
  \end{enumerate}
\end{lemma}
\begin{proof}
  Recall that, in general, it is well known (and straightforward to prove) that $a \leq b$ if and only if $a^* \leq b^*$ for any $a, b \in S$. Thus $E_r=E_{r^*}$, since idempotents are self-adjoint. Next, observe that if $e\in E$ and $e \leq t^*s$ then $t^*se = e=s^*te$, so that
  \[ te = tt^*s e =tt^*ss^*se=ss^*tt^*se= se. \]
  From this it also follows that $et^*=es^*$ and therefore $\omega_s=\omega_t$ on $E_{s^*t}$. Moreover, if $e\leq s^*t$, then $ses^*\leq ts^*$ because $ts^*ses^*=tes^*=ses^*$ and we have $\omega_{t^*}(ses^*)=t^*ses^*t=e$, therefore $\omega_{t^*}=\omega_{s^*}\colon E_{st^*}\to E_{s^*t}$ realizes the inverse of $\omega_s=\omega_t$.
\end{proof}
Recall the notion of \emph{saturated} Fell bundles from \cref{rem:saturated-vs-semi}, and observe that the following lemma states that, even in the non-saturated setting we are working in, some multiplication maps are isomorphisms.
\begin{lemma} \label{lemma:saturation-idempotents}
Let $\A = (A_s)_{s \in S}$ be a Fell bundle. The following assertions hold:
\begin{enumerate}[label=(\roman*)]
\item \label{lemma:saturation-idempotents:1} $\csp A_s \cdot A_{s^*s} = \cspn A_s \cdot A_{s^*} \cdot A_s$; and
\item \label{lemma:saturation-idempotents:3} $A_{se} = \cspn A_s \cdot A_e$
\end{enumerate}
for every idempotent $e \in E$ and $s \in S$.
\end{lemma}
\begin{proof}
In all assertions the inclusions $\supset$ are clear and follow from the canonical multiplication maps. Assertion \cref{lemma:saturation-idempotents:1} follows from \cref{def:pre:fell-bundle}~\cref{item:def:pre:fell-bundle:iso}. Assertion \cref{lemma:saturation-idempotents:3} follows from \cref{def:pre:fell-bundle}~\cref{item:def:pre:fell-bundle:iso} as well, since
\[ A_{se} = \cspn A_{se} A_{\left(se\right)^*} A_{se} \sbe \cspn A_{se} \cdot A_{es^*s} \sbe \cspn A_{s} \cdot A_e, \]
as desired.
\end{proof}
\begin{lemma} \label{lemma:unit-s-easy-commutation}
Let $\A = (A_s)_{s \in S}$ be a Fell bundle. Then $a_s 1_e = 1_{ses^*} a_s$ for every $a_s \in A_s, s \in S$ and idempotent $e \in E$ whenever $s^*s \geq e$.
\end{lemma}
\begin{proof}
  The claim follows from the fact that $A_{se}''$ is a Hilbert $A_{ses^*}''$-$A_e''$-bimodule, and hence $a_s 1_e = 1_{ses^*} a_s 1_e = 1_{ses^*} a_s$, as desired.
\end{proof}
\begin{proposition} \label{prop:unit-s-commutation}
Let $\A = (A_s)_{s \in S}$ be a Fell bundle. Then the following hold:
\begin{enumerate}[label=(\roman*)]
  \item \label{item:prop:unit-s-commutation:units} $1_{s^*t} = 1_{t^*s}$ and $1_{st^*} = 1_{ts^*}$ for every $s, t \in S$.
  \item \label{item:prop:unit-s-commutation:centrality} $1_s \in I_{s,1}''$ is central in $A_1''$.
  \item \label{item:prop:unit-s-commutation:comm} $a_s 1_{s^*t} = 1_{st^*} a_s$ for every $s, t \in S$ and $a_s \in A_s$.
\end{enumerate}
\end{proposition}
\begin{proof}
  \cref{item:prop:unit-s-commutation:units} follows from \cref{lemma:unit-s-conjugation}~\ref{lemma:unit-s-conjugation:doms} and the fact that, by construction, $1_{s^*t} \in I_{s^*t,1}'' = \langle A_e \mid e \leq s^*t, 1\rangle'' \sbe A_1''$. In fact, note that $I_{s^*t, 1}'' = I_{t^*s, 1}''$, and hence $1_{s^*t} = 1_{t^*s}$. For assertion \cref{item:prop:unit-s-commutation:centrality} just observe that every projection $1_e \in A_e''$ is central whenever $e \in E$ is an idempotent.
  
  In order to show \cref{item:prop:unit-s-commutation:comm} denote by $E_{s^*t} \coloneqq \{e \in E \mid e \leq s^*t\}$. In addition, given any finite subset $K \Subset E$ let $1_K \coloneqq \prod_{e \in K} 1_e$. With this notation, observe that for any finite set $F \Subset E_{s^*t}$ the element
  \[ p_F \coloneqq \sum_{i = 1}^{\left|F\right|} \sum_{\substack{K \sbe F \\ \left|K\right| = i}} \left(-1\right)^{i+1} 1_K \in \sum_{e \in F} A_e'' \sbe I_{s^*t, 1}'' \]
  is a self-adjoint projection (it realizes the join projection $\vee_{e\in F} 1_e$). Moreover, $p_F \rightarrow 1_{s^*t}$ when $F$ grows, where convergence happens in the weak topology of $I_{s^*t, 1}''$. Now, for a fixed $F\Subset E_{s^*t}$, we have
  \[ a_s p_F = a_s \sum_{i = 1}^{\left|F\right|} \sum_{\substack{K \sbe F \\ \left|K\right| = i}} \left(-1\right)^{i+1} 1_K = \sum_{i = 1}^{\left|F\right|} \sum_{\substack{K \sbe F \\ \left|K\right| = i}} \left(-1\right)^{i+1} 1_{sKs^*} a_s = p_{sFs^*} a_s, \]
  where the second equality follows from \cref{lemma:unit-s-easy-commutation} and the third from \cref{lemma:unit-s-conjugation} \cref{lemma:unit-s-conjugation:maps}. Therefore, since $p_F \rightarrow 1_{s^*t}$, the proof will be done once we show that $p_{sFs^*} \rightarrow 1_{st^*}$ in the weak operator topology of $I_{st^*, 1}''$. But, again, this is clear by \cref{lemma:unit-s-conjugation}~\cref{lemma:unit-s-conjugation:maps}, since conjugation by $s$ defines a bijective map between finite subsets of $E_{s^*t}$ and finite subsets of $E_{st^*}$.
\end{proof}

\subsection{Cross-sectional \cstar{}algebras of Fell bundles} \label{subsec:pre:cross-sectional-alg}
In this section we recall the construction of the \emph{full cross-sectional \cstar{}algebra} associated to a Fell bundle (see \cref{sec:regular-Fell-principle,def:reduced-cross-sect} below for the \emph{reduced} counterpart). To such end, let $\A=(A_s)_{s\in S}$ be a Fell bundle over an inverse semigroup $S$ with unit $1 \in S$. Notice that $A_e$ is a \cstar{}algebra as it is, in fact, a closed two sided-ideal of $A = A_1$ for every idempotent $e \in E$.  We denote by $\contc(\A) \coloneqq \contc(S,\A)$ the space of finitely supported sections $S \to \A$, that is, the algebraic direct sum $\oplus_{s\in S}^\alg A_s$. We shall write $a_s \delta_s \in \contc(\A)$ for the section supported in $\{s\}$ taking value $a_s \in A_s$ at $s \in S$. Recall that $\contc(\A)$ is a \Star{}algebra with respect to the operations
$$\left(a_s \delta_s\right) \cdot \left(b_t \delta_t\right) \coloneqq a_s b_t \delta_{st} \quad \text{and} \quad \left(a_s \delta_s\right)^* \coloneqq a_s^* \delta_{s^*},$$
where we implicitly use the Fell bundle and inverse semigroup operations (recall \cref{thm:fell-bundle-extra-str} and the subsequent \cref{rem:thm:fell-bundle-extra-str}). The following definition is central to the paper.
\begin{definition} \label{def:max-cross-c-star-alg}
  The \textit{maximal cross-sectional \cstar{}algebra} $\fullalg{\A}$ of a Fell bundle $\A = (A_s)_{s \in S}$ is the universal \cstar{}algebra of the quotient $\Qc(\A)$ of $\contc(\A)$ given by
  \[ \Qc\left(\A\right) \coloneqq \contc\left(\A\right)/\I_\A, \]
  where $\I_\A \sbe \contc(\A)$ is the $*$-ideal generated by the differences $a_s \delta_s - a_s \delta_t$ for all $s \leq t$ in $S$ and $a_s \in A_s \sbe A_t$.
\end{definition}

In other words, $\fullalg{\A}$ is the universal \cstar{}algebra generated by $a_s \delta_s$ with $a_s \in A_s$ and $s \in S$ satisfying the relations that defined the \Star{}algebra structure on $\contc(\A)$ above and the relation imposing that $a_s \delta_s = a_s \delta_t$ whenever $a_s \in A_s$ and $s \leq t$ in $S$.

\begin{remark}
It can be shown that the subspace of $\contc(\A)$ generated by differences $a_s \delta_s - a_s \delta_t$ with $s \leq t \in S$ and $a_s \in A_s \sbe A_t$ is already a \Star{}ideal of $\contc(\A)$, and therefore this equals $\I_\A$. This is already observed in \cite{Exel:noncomm.cartan}, where the above definition was first given.
\end{remark}
\begin{remark}
By the universal property of $\fullalg{\A}$ (see \cref{pre:fell:fullalg-univ}) we get a \Star{}homomorphism $\Qc\left(\A\right) \to \fullalg{\A}$ with dense image. But, as it turns out, this \Star{}homomorphism does not have to be injective in general, that is, its kernel is not just $\I_\A$, but a possibly larger ideal that we now describe.
\end{remark}

Given a Fell bundle $\A = (A_s)_{s \in S}$ over $S$ and elements
$t,u\in S$, it is proved in \cite{Buss-Exel-Meyer:Reduced}*{Lemma~2.5} 
that there exists a unique isomorphism of Hilbert bimodules:
\begin{equation}\label{eq:theta-t-u}
\theta_{t,u}\colon A_u\cdot I_{t,u}\congto A_t\cdot I_{t,u}
\end{equation}
extending the identity map $A_u\cdot A_{v^*v}= A_v\to A_v= A_t\cdot A_{v^*v}$ for every $v\in S$ with $v\leq t,u$, so that $uv^*v=tv^*v=v$. Recall that $I_{t,u}$ is the ideal of $A$ generated by $A_v^*A_v\sbe A_{v^*v}$ with $v\leq t,u$ (see \cref{lemma:saturation-idempotents}).
\begin{definition}
Given a Fell bundle $\A$, let $\J_\A$ be the subspace of $\contc(\A)$ generated by elements of the form $\theta_{t,u}(a)\delta_t-a\delta_u$ for $t,u\in S$ and $a\in A_u\cdot I_{t,u}$.
\end{definition}

Notice that $\J_\A$ obviously contains $\I_\A$, but the inclusion might be strict as $\J_\A$ may also contain some limits of elements of $\I_\A$. As will become clear in what follows, $\J_\A$ is also a \Star{}ideal of $\contc(\A)$. In order to show this, however, we first need to define a certain conditional expectation and describe $\J_\A$ in a different way.
\begin{definition} \label{def:alg-cross-sectional}
Given a Fell bundle $\A = (A_s)_{s \in S}$ over $S$, the \emph{cross-sectional algebra of $\A$}, denoted by $\algalg{\A}$, is the quotient $\algalg{\A} \coloneqq \contc(\A)/\J_\A$.
\end{definition}

Since $\J_\A$ contains $\I_\A$, we get surjective \Star{}homomorphisms (compare with \cref{prop:dense-image-in-universal} below)
\[ \contc(\A)\onto \Qc(\A)\onto \algalg\A. \]
Moreover, it is observed in \cite{Buss-Exel-Meyer:Reduced} that every representation of $\fullalg{\A}$ vanishes on $\J_\A$, so we get also a homomorphism with dense range
\[ \algalg{\A}\to \fullalg{\A}. \]
In order to understand all this better, we need to recall the construction of the canonical (weak) conditional expectation of $\A$.
For this, first notice that as a particular case of \cref{eq:theta-t-u}, for each $s\in S$ we get an isomorphism $\theta_{s,1} \colon A_s \cdot I_{s,1} \congto I_{s,1}$ extending the identity map
$A_s\cdot A_e= A_e\to A_e$ for $e \in E$ with $e\leq s$. This gives rise
to a canonical map 
\[ \Theta_{s,1} \colon A_s \rightarrow \M\left(I_{s,1}\right), \;\; \text{where} \;\, \Theta_{s,1}\left(a\right)x \coloneqq \theta_{s,1}\left(a\cdot x\right). \]
Here we write $\M(B)$ for the multiplier of a \cstar{}algebra $B$. Viewing $\M(I_{s,1})\subseteq I_{s,1}''\subseteq A''$ in the standard way, the above map can also be represented as
$$\Theta_{s,1}(a)=\lim_i\theta_{s,1}(a u_i),$$
where $\{u_i\}_i \sbe I_{s,1}$ is an approximate unit and the limit is with respect to the weak topology of the von Neumann algebra $A_1''$. Actually, the limit above can also be understood in the strict topology of $\M(I_{s,1})$, see \cite{Buss-Exel-Meyer:Reduced}*{Lemma~4.5}.

It is fruitful to look at the above structure from the point of view of the 
von Neumann enveloping Fell bundle $\A''$ (and this is the way this is done in \cite{Buss-Exel-Meyer:Reduced}). 
The isomorphism $\theta_{s,1}$ extends to a weakly continuous isomorphism (of Hilbert $W^*$-bimodules):
$$\theta_{s,1}''\colon A_s''\cdot I_{s,1}''\congto I_{s,1}''$$
and so we get an embedding $A_s''\cdot I_{s,1}''\into A_1''$ and a map $\Theta_{s,1}''\colon A_s''\to A_1''$ that extends $\Theta_{s,1}$ 
when we view  $\M(I_{s,1})\subseteq I_{s,1}''\subseteq A_1''$. Moreover, notice that $A_s''\cdot I_{s,1}''=A_s''\cdot 1_s$ so that $A_s''$ decomposes as a direct sum
\[A_s''=A_s''\cdot 1_s\oplus A_s''\cdot 1_s^\perp,\]
where $1_s^\perp$ denotes the unit of the complement ideal $(I_{s,1}'')^\perp\idealin A_1''$. In this way, $\Theta_{s,1}''$ (and hence also its restriction $\Theta_{s,1}$) 
is just the composition of the projection map $A_s''\onto A_s''\cdot 1_s$, $x\mapsto x\cdot 1_s$ 
with the embedding $A_s''\cdot 1_s \congto I_{s,1}''\sbe  A_1''$. As with \cref{thm:fell-bundle-extra-str}, the following has only been proved in the saturated case, but the same ideas carry over to the non-saturated setting.
\begin{proposition}[\cite{Buss-Exel-Meyer:Reduced}]\label{prop:cond-exp}
Given a Fell bundle $\A = (A_s)_{s \in S}$, consider the map
\begin{align*}
  P \colon \contc\left(\A\right) & \to A_1'',\quad P(a_s \delta_s)=\theta_{s,1}''(a_s 1_s)
\end{align*}
defined as the linear extension of the maps $\Theta_{s,1} \colon A_s \to \M(I_{s,1})\sbe A_1''$.
Then $P$ vanishes on $\J_\A$ and factors through a faithful (weak) conditional expectation 
$$\tilde P\colon \algalg\A\to A_1''.$$
\end{proposition}

It is clear from the construction above that $P$ extends to $\contc(\A'')$, that is, we get from the above construction also a conditional expectation $P''\colon \contc(\A'')\to A_1''$ which factors faithfully to $\algalg{\A''}\to A_1''$.

\begin{definition}
Given a Fell bundle $\A=(A_s)_{s\in S}$, we call the map $P\colon \contc(\A)\to A_1''$ defined above the \emph{canonical conditional expectation} of $\A$.
\end{definition}

Since $\contc(\A)$ and $\algalg\A$ are only \Star{}algebras, it might be not clear what we mean by a (weak) conditional expectation on these algebras. This only means $P$ is a linear map which is positive in the sense that $P(x^*x)\geq 0$ for all $x\in \contc(\A)$; it preserves adjoints, meaning that $P(x^*)=P(x)^*$, and that $P$ restricts to the identity on $A_1$, that is, $P(a\delta_1)=a$ for all $a\in A_1$. Similar properties hold for $\tilde P$. And the fact that $\tilde P$ is faithful means that $P(x^*x)=0$ happens only if $x=0$ for any $x\in \algalg\A$. This is proved in \cite{Buss-Exel-Meyer:Reduced}*{Proposition 3.6}. The fact that $\tilde P$ is faithful also implies that $P$ is symmetric in the sense that
$P(x^*x) = 0$ if and only if $P(xx^*) = 0$, for all $x \in \contc(\A)$. This means that
\begin{equation} \label{eq:ideal-n-p}
\NN_P \coloneqq \left\{x \in \contc\left(\A\right): P\left(x^*x\right) = 0\right\}
\end{equation}
is a \Star{}ideal of $\contc(\A)$. In fact, since $\tilde P$ is faithful, it also follows that
$\NN_P$ coincides with the kernel of the quotient homomorphism $\contc(\A)\onto \algalg\A$, that is, 
\[\NN_P = \J_\A = \spn\left\{\theta_{u,t}(a)\delta_u - a\delta_t \colon a \in I_{u,t}\right\}.\]

The following result from \cite{Buss-Exel-Meyer:Reduced} indicates that $\algalg\A$ yields the ``correct'' definition of the algebraic crossed product of a Fell bundle. 
\begin{proposition}\cite{Buss-Exel-Meyer:Reduced}*{Proposition~4.3} \label{prop:dense-image-in-universal}
Given a Fell bundle $\A = (A_s)_{s \in S}$, the \Star{}algebra $\algalg{\A}$ 
embeds densely in $\fullalg{\A}$.    
\end{proposition}

For the following lemma recall that we canonically embed $A_s$ into $\contc(\A)$ as the set $A_s \delta_s$ of sections $a_s \delta_s$ taking value $a_s$ at $s$ and $0$ otherwise. 
\begin{lemma} \label{lemma:n-p-intersection-fibers}
  Given a Fell bundle $\A = (A_s)_{s \in S}$, $A_s \delta_s \cap \NN_P = 0$ for every $s \in S$.
\end{lemma}
\begin{proof}
  For any $a_s \delta_s \in A_s \delta_s$ we have
  \[ P\left(\left(a_s \delta_s\right)^*\left(a_s \delta_s\right)\right) = P\left(a_s^* a_s \delta_{s^*s}\right) = \theta_{s^*s, 1}''\left(a_s^*a_s 1_{s^*s}\right) = a_s^*a_s 1_{s^*s} \]
  since $s^*s \leq 1$. Therefore if $a_s \delta_s \in \NN_P$ then $a_s = 0$, as claimed.
\end{proof}

Observe that, in general, the target of the canonical conditional expectation $P$ in \cref{prop:cond-exp} is only the von Neumann algebra $A_1''$, as opposed to $A_1$ itself. We end the section with a condition that guarantees that the image of $P$ lands in $A_1$.
\begin{corollary} \label{cor:image-p-a}
Let $\A = (A_s)_{s \in S}$ be a Fell bundle. Then $P(x) \in A_1$ for all $x \in \contc(\A)$ if and only if $I_{s,1}$ is complemented in $\cspn A_s^*A_s \sbe A_1$ for all $s \in S$. 
\end{corollary}
\begin{proof}
The proof is given in~\cite{Buss-Exel-Meyer:Reduced}*{Proposition~6.3}, and it does not assume that the Fell bundle is saturated. However, we now give a sketch of the proof of the forward implication and refer the reader to~\cite{Buss-Exel-Meyer:Reduced}*{Proposition~6.3} for the (more complicated) proof of the converse. Observe that if $I_{s,1}$ is a complemented ideal of $A_{s^*s}$ then $A_s$ decomposes as $A_s = A_s \cdot I_{s,1} \oplus A_s \cdot I_{s,1}^\perp$. And in this case $P$ acts on $A_s$ by projecting first to $A_s\cdot I_{s,1}$ and then composes this with the isomorphism $\theta_{s,1}\colon A_s\cdot I_{s,1}\congto I_{s,1}\sbe A_1$.
Hence $P(A_s\delta_s)\sbe A_1$ for all $s$ 
and therefore also $P(\contc(\A))\sbe A_1$.
\end{proof}

The following definition is from~\cite{Kwaniewski2019EssentialCP}. We will use it in the sequel, albeit not predominantly.
\begin{definition} \label{def:fell-bundle-closed}
We say a Fell bundle $\A = (A_s)_{s \in S}$ is \emph{closed} if the image of the canonical conditional expectation $P$ is contained in $A_1$.
\end{definition}


\begin{remark} \label{rem:image-p-a}
  Observe that for $\A$ to be closed it is \emph{not} enough that the ideals $A_e$ be complemented or even unital. Indeed, this is not even true in the basic case of a Fell bundle coming from just an ordinary inverse semigroup. For instance, take $S \coloneqq \{1, z\} \cup \{e_n\}_{n \in \N}$, where $1$ is a unit, $z^2 = 1$,  $z \geq e_n$ for all $n \in \N$, and $e_n e_m = e_m e_n = e_{\min\{n,m\}}$. Consider the Fell bundle associated to the action of $S$ on its spectrum $X = \dual E \cong \N\sqcup\{\infty\}$ (the one-point compactification of $\N$). Notice that the groupoid of germs of this action is the universal (Paterson) groupoid of $S$ (see~\cite{Paterson1999}). Then the ideals $A_e$ are all complemented (as they are all unital). Nevertheless, $P$ does not take values in $A_1$ as the  universal groupoid of $S$ is not Hausdorff (see \cite{Buss-Exel-Meyer:Reduced}*{Proposition 5.1} and \cite{Paterson1999}*{Theorem~4.4.1}).
\end{remark}

\subsection{Representations of Fell bundles} \label{subsec:pre:fell-reps}
In this section we introduce the central notion of a \emph{representation of a Fell bundle}. As expected, it preserves all the inherent structure of the bundle.
\begin{definition} \label{def:fell-bundle-rep}
Let $\A = (A_s)_{s \in S}$ be a Fell bundle over an inverse semigroup $S$.
\begin{enumerate}[label=(\roman*)]
  \item A \textit{representation of $\A$ on the Hilbert space $\H$} is a family $\pi = (\pi_s)_{s\in S}$ of linear maps $\pi_s \colon A_s \to \L(\H)$ such that
    \[ \pi_s\left(a_s\right)\pi_t\left(b_t\right) = \pi_{st}\left(a_sb_t\right), \;\; \pi_s\left(a_s\right)^* = \pi_{s^*}\left(a_s^*\right) \;\; \text{and} \;\; \pi_s\left(a_s\right) = \pi_u\left(a_s\right), \]
    whenever $s,t,u\in S$ with $s \leq u$ and $a_s \in A_s$ and $b_t \in A_t$.
  \item We say the representation $\pi = (\pi_s)_{s \in S}$ is \emph{nondegenerate} if $\pi_1$ is a nondegenerate representation of $A=A_1$.
\end{enumerate}
\end{definition}

By definition a representation of $\A$ always respects the inclusion maps $\iota_{t,s}\colon A_s\into A_t$ for every $s\leq t$ in $S$. It turns out that they also automatically respect the extended inclusion maps $\theta_{t,u}\colon A_u\cdot I_{t,u}\into A_t\cdot I_{t,u}$ for every $t,u\in S$, as the following result shows. Actually, the following was only proved in the saturated case, but the same ideas work in the non-saturated setting as well (recall \cref{rem:saturated-vs-semi}).
\begin{proposition}[\cite{Buss-Exel-Meyer:Reduced}*{Proposition~2.9}]
Every representation $\pi$ of $\A$ satisfies
\[ \pi_t\circ\theta_{t,u}\left(a_u\right)=\pi_u\left(a_u\right) \]
for all $t,u\in S$ and $a_u\in A_u\cdot I_{t,u}$.
\end{proposition}

The following notation will be useful throughout the paper, particularly when discussing Fell's absorption principle in the context of Fell bundles over inverse semigroups.
\begin{definition} \label{def:fell-rep-hilbert-fibers}
Given a Fell bundle $\A = (A_s)_{s \in S}$ and a representation $\pi = (\pi_s)_{s \in S}$, we define $\H_s:=\cspn \pi_s(A_s)\H\sbe \H$.
\end{definition}

Observe that, by~\cite{Buss-Exel-Meyer:Reduced}*{Lemma~2.8}, we have $\H_s = \H_{ss^*} = \pi_1(A_{ss^*}) \H \sbe \H$ for all $s \in S$. Likewise, notice that $\H = \H_1$ if and only if the representation $\pi$ is nondegenerate. 

\begin{remark}
Notice that every representation $\pi$ of $\A$ extends uniquely to a weakly continuous representation $\pi''=(\pi_s'')_{s\in S}$ of its von Neumann enveloping Fell bundle $\A''$. Moreover, since $A_s$ is weakly dense in $A_s''$, for all $s \in S$ we have 
\[\H_s = \cspn \pi_s\left(A_s\right)\H = \cspn \pi_s''\left(A_s''\right)\H = \cspn \pi_1''\left(A_{ss^*}''\right)\H. \]
\end{remark}

We end this section with the expected universal characterization of the full cross sectional \cstar{}algebra.
\begin{proposition} \label{pre:fell:fullalg-univ}
Let $\A = (A_s)_{s \in S}$ be a Fell bundle over the inverse semigroup $S$. Then $\fullalg{\A}$ is universal for representations of $\A$, that is, given any representation $\pi = (\pi_s)_{s \in S}$ of $\A$ on a Hilbert space $\H$, there is a (unique) $*$-homomorphism $\rho \colon \fullalg{\A} \rightarrow \L(\H)$ such that $\rho(a_s \delta_s) = \pi_s(a_s)$ for every $a_s \in A_s$ and $s \in S$.
\end{proposition}

\section{Regular representation and Fell's absorption principle}\label{sec:regular-Fell-principle}
Before going into the \emph{approximation property} (see \cref{sec:ap}), in this section we prove an analogue of Fell's absorption principle in the setting of Fell bundles over inverse semigroups (see \cref{thm:fell-principle}).

\subsection{The regular representation} In this section we introduce the \emph{left regular representation} of a Fell bundle, and prove some structural results about it. Note that this representation will be key in order to construct the \emph{reduced cross-sectional \cstar{}algebra} of a Fell bundle (see \cref{def:reduced-cross-sect}). We refer the reader to~\cite{Buss-Exel-Meyer:Reduced} for a more comprehensive discussion.

We first need to discuss a canonical inner product on $\contc(\A)$. Recall we naturally endow $\contc(\A)$ with a \Star{}algebra structure (see \cref{subsec:pre:cross-sectional-alg}). Likewise, recall the definition of the ideal $\NN_P \sbe \contc(\A)$ from \cref{eq:ideal-n-p}.
\begin{proposition} \label{prop:inner-product}
Let $\A = (A_s)_{s \in S}$ be a Fell bundle. Consider 
\[ \braket{\xi}{\eta}_{A_1''} \coloneqq P\left(\xi^* \cdot \eta\right), \;\; \text{where} \;\, \xi, \eta \in \contc\left(\A\right). \]
Then $\braket{\cdot}{\cdot}_{A_1''}$ is a \emph{semi-inner product} with values in $A_1''$ whose kernel is exactly $\NN_P$, that is, $\braket{\xi}{\xi}_{A_1''}=0$ if and only if $\xi\in \NN_P$. Thus, it factors through an $A_1''$-valued inner product on $\algalg{\A}$.
\end{proposition}
\begin{proof}
The proof follows from the properties of $P$ showed in~\cite{Buss-Exel-Meyer:Reduced}*{Proposition~3.6 and Remark~4.7}. Recall that $\braket{\cdot}{\cdot}_{A_1''}$ being a semi-inner product with values in $A_1''$ means that it is a sesquilinear form, $A_1$-linear in the second variable, and is positive meaning that $\braket{\xi}{\xi}_{A_1''}\geq 0$ in the \cstar{}algebra $A_1''$ for all $\xi\in \contc(\A)$.
\end{proof}

We could now proceed and take the completion of $\algalg\A$ with the respect to the norm induced from the above inner product.
But there is a technical complication here as this completion does not necessarily carry a (right) $A_1''$-module structure.
The easy way out of this is to go directly to the von Neumann enveloping Fell bundle $\A''$ and consider the similar completion there.
This is the way it is done in \cite{Buss-Exel-Meyer:Reduced}.

\begin{definition} \label{def:left-reg:l-two}
Given a Fell bundle $\A = (A_s)_{s \in S}$ we denote by \emph{$\ell^2(\A'')$} the completion of $\algalg{\A''}$ with respect to the norm induced from the inner product $\braket{\cdot}{\cdot}_{A_1''}$ given by~\cref{prop:inner-product} for $\A''$.
\end{definition}

Note that $\ell^2(\A'')$ is, in particular, a Hilbert module over $A_1''$. We can now introduce the left regular representation of a Fell bundle.
\begin{definition} \label{def:reduced-cross-sect}
Let $\A = (A_s)_{s \in S}$ be a Fell bundle over an inverse semigroup $S$.
\begin{enumerate}[label=(\roman*)] 
\item The \emph{(left) regular representation} of $\A$ is the canonical \Star{}homomorphism $\Lambda \colon \algalg{\A} \to \L(\ell^2(\A''))$ defined by (left) multiplication, that is,
\[ \Lambda\left(a\right)\left(f\right) \coloneqq af \]
for every $f \in \ell^2(\A'')$ and $a \in \algalg{\A}$.
\item The \emph{reduced cross-sectional \cstar{}algebra} of $\A$ is the completion $\redalg{\A}$ of $\algalg{\A}$ with respect to the \cstar{}norm
\[ \|a\|_\red \coloneqq \|\Lambda\left(a\right)\|_{\L\left(\ell^2\left(\A''\right)\right)}, \;\; \text{where} \;\, a \in \algalg{\A}. \]
\end{enumerate}
\end{definition}

\begin{remark}
By the universal property of $\fullalg{\A}$ (see also \cref{pre:fell:fullalg-univ}) the left regular representation of $\A$ gives a representation $\fullalg{\A} \to \redalg{\A} \sbe \L(\ell^2(\A''))$. By abuse of notation we will denote both by $\Lambda$.
\end{remark}

The following proposition gives an explicit formula for the left regular representation $\Lambda \colon \algalg{\A} \rightarrow \L(\ell^2(\A''))$.  For it, recall that the sections $A_s\cong A_s \delta_s$ canonically embed into $\algalg{\A}$ (see \cref{lemma:n-p-intersection-fibers}).
\begin{proposition} \label{prop:left-reg-rep:expression}
  Let $\A = (A_s)_{s \in S}$ be a Fell bundle over $S$. Then:
  \[ \Lambda_s\left(a_s \delta_s\right)\left(f + \NN_P\right) = \sum_{t \in S} a_s f\left(t\right) \delta_{st} + \NN_P \]
  for every $a_s \in A_s$ and $f \in \contc(\A'')$.
\end{proposition}
\begin{proof}
  We have $f = \sum_{r \in S} f(r) \delta_r + \NN_P$, where the sum has, at most, finitely many non-zero terms. By definition
  \[ \Lambda\left(a_s \delta_s\right)\left(f\right) = \sum_{t \in S} a_s \delta_s f\left(t\right) \delta_t + \NN_P = \sum_{t \in S} a_s f\left(t\right) \delta_{st} + \NN_P, \]
  as claimed.
\end{proof}
\begin{remark} \label{rem:induce-instead-of-extend}
  Note that, by \cref{prop:left-reg-rep:expression}, we have a commuting diagram
  \begin{center}
    \begin{tikzcd}
    \contc\left(\A\right) \arrow[two heads]{d}{\pi_P} \arrow{r}{\tilde{\Lambda}\left(a\right)} &[2.5em] \contc\left(\A\right) \arrow[two heads]{d}{\pi_P} \\
    \algalg{\A} \arrow{r}{\Lambda\left(a + \NN_P\right)} & \algalg{\A}
    \end{tikzcd}
  \end{center}
  for every $a \in \contc(\A)$, where $\pi_P \colon \contc(A) \rightarrow \algalg{\A}$ is the canonical quotient map and $\tilde{\Lambda}(a)(b) \coloneqq ab$. In particular, the map $\tilde{\Lambda}$ descends to a map $\Lambda$ on the quotients. Several different situations similar to this will happen throughout the paper (see, e.g., \cref{prop:fell-principle:unitary,prop:tensor-prod:red-comp}).
\end{remark}

We isolate the following notion due to its relevance within the paper.
\begin{definition} \label{def:weak-containment}
We say a Fell bundle $\A = (A_s)_{s \in S}$ has the \textit{weak containment property} if the left regular representation $\Lambda \colon \fullalg{\A} \to \redalg{\A}$ is a \Star{}isomorphism.
\end{definition}

We now turn our attention to the construction of new representations. In particular, we can induce representations of $A_1$ to representations of the Fell bundle $\A = (A_s)_{s \in S}$ using the left regular representation in the following way.
\begin{definition} \label{def:induced-reps}
Given a representation $\rho\colon A_1 \to \L(\H)$ on $\H$, we define:
\begin{enumerate}[label=(\roman*)]
  \item $\ell^2(\A'') \otimes_{\rho''} \H$ to be the \emph{balanced tensor product of $\ell^2(\A'')$ and $\H$}, that is, the Hausdorff completion of $\ell^2(\A'') \odot \H$ by the semi-inner product given by
  \[ \braket{f_1 \otimes h_1}{f_2 \otimes h_2}_{\ell^2\left(\A''\right) \otimes_{\rho''} \H} \coloneqq \braket{h_1}{\rho''\left(\braket{f_1}{f_2}_{A''}\right) h_2}_{\H}, \]
  where $f_1, f_2 \in \ell^2(\A'')$ and $h_1, h_2 \in \H$ and $\rho''$ is the weakly continuous extension of $\rho$ to $A_1''$; 
  \item the \emph{induced representation of $\rho$} to be the representation given by
    \begin{align*}
    \Ind\rho \colon \fullalg{\A} & \rightarrow \L\left(\ell^2\left(\A''\right) \otimes_{\rho''} \H \right), \\
    a_s \delta_s & \mapsto \left(\Ind\rho\right)_s\left(a_s \delta_s\right) \coloneqq \Lambda_s\left(a_s \delta_s\right) \otimes_{\rho''} \Id.
    \end{align*}
\end{enumerate}
\end{definition}

Henceforth, we will denote $\braket{\cdot}{\cdot}_{\H}, \braket{\cdot}{\cdot}_{A''}$ and $\braket{\cdot}{\cdot}_{\ell^2(\A'') \otimes_{\rho''} \H}$ in \cref{def:induced-reps} all by $\braket{\cdot}{\cdot}$ unless otherwise confusing. The following remarks are in order.

\begin{remark} \label{rem:induced-factors-reduced}
  Notice that $\Ind\rho$ is the induced representation from the weakly continuous extension $\rho''\colon A_1''\to \L(\H)$ via $\ell^2(\A'')$ viewed as a correspondence from $\fullalg{\A}$ to $A_1''$. Since the left action of $\fullalg{\A}$ on $\ell^2(\A'')$ factors through $\redalg{\A}$ by construction, the induced representation $\Ind\rho$ always factors through $\redalg{\A}$, regardless of the initial representation $\rho \colon A_1 \to \L(\H)$. Indeed, the reduced norm can be described via induced representations: take, e.g., the universal representation of $A_1$ and induce it up to $\algalg{\A}$. The associated representation of $\redalg{\A}$ is faithful because the universal representation of $A_1$ extends faithfully to $A_1''$.
\end{remark}

On the one hand a general faithful representation of $A_1$ need not induce a faithful representation of $\redalg{\A}$. And on the other hand an induced representation can be faithful on $\redalg{\A}$ without being faithful on $A_1''$ (or even on $A_1$) as $\redalg{\A}$ may be, for instance, simple. General criteria for faithfulness of induced representations on $\redalg{\A}$ are given in \cite{Buss-Exel-Meyer:Reduced}*{Section~4}. For instance, it is proved in \cite{Buss-Exel-Meyer:Reduced}*{Theorem~7.4} that the \emph{atomic} representation of $A_1$, i.e., the direct sum $\rho_a=\oplus_{[\pi]\in \hat A_1}\pi$ of all (classes of) irreducible representations of $A_1$, yields a faithful representation $\Ind\rho_a$ of $\redalg{\A}$.

\subsection{Fell's absorption principle for groups} \label{subsec:fell-principle:groups}
Let us now recall Fell's absorption principle for the case of Fell bundles over (discrete) groups (see \cite{Exel:Partial_dynamical}*{Chapter~18}). Let $\A = (A_g)_{g\in G}$ be a Fell bundle over a group $G$ and let $\pi=(\pi_g)_{g\in G}$ be a representation of $\A$ on some Hilbert space $\H$. Then we obtain a new representation $\pi^\lambda=(\pi^\lambda_g)_{g\in G}$ of $\A$ on $\ell^2(G,\H)\cong \ell^2(G)\otimes \H$ defined by $\pi^\lambda_g\coloneqq\lambda_g\otimes\pi_g$ for all $g\in G$, where $\lambda\colon G\to \L(\ell^2(G))$ denotes the left regular representation of $G$. Observe that
\[ \pi^\lambda_g\left(f\right)\left(h\right) = \pi_g\left(f\left(g^{-1}h\right)\right), \;\; \text{for} \;\, g,h \in G \;\, \text{and} \;\, f \in \ell^2\left(G,\H\right). \]
Here $\ell^2(\A)$ denotes the Hilbert $A_1$-module obtained as the completion of $\contc(\A)$ with respect to the norm induced by the $A_1$-valued inner product
\[ \braket{\xi}{\eta}_{A_1} \coloneqq \sum_{g\in G}\xi\left(g\right)^*\eta\left(g\right). \]
The so-called left regular representation of $\A$ is the representation 
$$\Lambda\colon \A\to \L(\ell^2(\A))\quad\mbox{given by }(\Lambda(a_g)\xi)(h)\coloneqq a_g\xi(g^{-1}h).$$
By definition, the reduced cross-sectional \cstar{}algebra $\redalg{\A}$ is $$\redalg{\A}\coloneqq \Lambda(\fullalg{\A})\sbe \L(\ell^2(\A)).$$
The following result is an analogue of Fell's absorption principle in the context of Fell bundles over discrete groups (see~\cite{Exel:Partial_dynamical}*{Proposition~18.4}). Before we state and prove the result, however, we need some notation. 
  Let $\ell^2_\pi(G,\H)$ be the (closed) subspace of $\ell^2(G,\H)$ consisting of the functions $\xi\in \ell^2(G,\H)$ satisfying $\xi(g)\in \H_g:=\cspn\pi_g(\A_g)\H\sbe \H$. A simple computation shows that this is an invariant subspace under $\pi\otimes\lambda$, hence it restricts to a representation of $\A$ on the Hilbert space $\ell^2_\pi(G,\H)$, which we denote by $\pi^\lambda$.
  
\begin{proposition} \label{prop:fell-absorption-groups}
  Let $\A = (A_g)_{g\in G}$ be a Fell bundle over a discrete group $G$ and let $\pi=(\pi_g)_{g\in G}$ be a representation of $\A$. Then the representation
  \[ \pi^\lambda\colon \fullalg{\A}\to \L\left(\ell^2_\pi\left(G,\H\right)\right) \]
  is weakly contained in the regular representation $\Lambda\colon \fullalg{\A} \to \redalg{\A}$, and therefore it factors through $\redalg{\A}$.
\end{proposition}
\begin{proof}
  A simple computation with the inner products shows that the canonical map
  \[ \contc\left(\A\right) \odot \H \to \contc\left(G,\H\right), \;\; \text{where} \;\, f \otimes v\mapsto \left[g \mapsto \pi_g\left(f\left(g\right)\right)v\right] \]
  is an isometry and extends to an isomorphism
  \[ U_\pi \colon \ell^2(\A)\otimes_{\pi_1} \H\congto \ell^2_\pi(G,\H) \]
  intertwining the representations $\Ind \pi_1 \coloneqq \Lambda\otimes_{\pi_1} \Id$ and the restriction of $\pi^\lambda$ to $\ell^2_\pi(G,\H)$. The representation $\Ind \pi_1$, being induced via $\ell^2(\A)$, is always weakly contained in the regular representation of $\A$.
\end{proof}

\begin{remark}
We note in passing that \cref{prop:fell-absorption-groups} is \emph{not} actually Fell's absorption principle, as employed in~\cite{Exel:Partial_dynamical}*{Proposition~18.4}, unless the Fell bundle $\A$ is assumed to be saturated (recall \cref{rem:saturated-vs-semi}) and the representation $\pi$ is nondegenerate. Indeed, in such case the space $\ell^2_\pi(G, \H) = \ell^2(G,\H)$ because $\H_g=\cspn \pi_1(A_g A_g^*)\H=\H$ for all $g\in G$, and hence we do recover~\cite{Exel:Partial_dynamical}*{Proposition~18.4}. Otherwise, \cref{prop:fell-absorption-groups} only proves that a sub-representation of $\pi \tensor \lambda$, namely $\pi^\lambda$, is weakly contained in $\Lambda$. Nevertheless, we shall stick to the naming of \emph{Fell's absorption principle} here and in the sequel because, in particular, the above version recovers the usual Fell's absorption principle for ordinary group actions, in special for group representations, where the name comes from.
\end{remark}

\subsection{Fell's absorption principle for inverse semigroups}
We are going to follow a similar strategy as the one outlined in \cref{subsec:fell-principle:groups} and \cref{prop:fell-absorption-groups}, but this time in the case of inverse semigroups, describing induced representations of a Fell bundle. The main complication in the more general case of inverse semigroups is the description of the regular representation $\Lambda$ of a Fell bundle over an inverse semigroup $S$. By construction, this representation acts not just on a completion of $\contc(\A'')$, but of a certain quotient of it. 
Thus it is natural to expect that an induced representation will act not just on a completion of a space of sections $S \to \H$, but on a quotient of it; and this is what will be happening indeed (see \cref{prop:fell-principle:inner-prod,prop:fell-principle:unitary} and \cref{thm:fell-principle}).

The main technical points of Fell's absorption principle in the setting of inverse semigroups are the key \cref{prop:fell-principle:inner-prod,prop:fell-principle:unitary}, which allow to construct a unitary operator that will intertwine two representations of the Fell bundle $\A$.
\begin{definition} \label{def:fell-principle:pre}
Let $\A = (A_s)_{s \in S}$ be a Fell bundle over $S$, and let $\pi = (\pi_s)_{s \in S}$ be a representation of $\A$.
\begin{enumerate}[label=(\roman*)]
\item We define $\contc^\pi(S,\H) \coloneqq \{f\in \contc(S,\H) \mid f(s)\in \H_s\;\, \text{for all} \;\, s\in S\},$ where $\H_s\coloneqq\cspn \pi_s(A_s) \H \sbe \H$ (recall \cref{def:fell-rep-hilbert-fibers}).
\item We denote by $U_{\pi,0} \colon \contc(\A) \odot \H \rightarrow \contc^\pi(S,\H)$ the linear map defined by
\[ U_{\pi,0}\left(f \otimes v \right)\left(s\right) \coloneqq \pi_s\left(f\left(s\right)\right) v \]
for $f \in \contc(\A)$ and $v \in \H$.
\item We also denote by $U''_{\pi,0}\colon \contc(\A'') \odot \H \rightarrow \contc^\pi(S,\H)$ the linear map defined in the same way by
\[ U_{\pi,0}''\left(f \otimes v \right)\left(s\right) \coloneqq \pi_s''\left(f\left(s\right)\right) v \]
$f \in \contc(\A'')$ and $v \in \H$. Recall that $\H_s=\cspn\pi_s''(A_s'')\H$, so the above makes sense.
\end{enumerate}
\end{definition}

For notational convenience we shall henceforth denote $U_{\pi,0}\left(f \otimes v \right)$ simply by $f_v$ whenever $f \in \contc(\A)$ and $v \in \H$. We shall use the same notation if $f\in \contc(\A'')$. 

Notice that $\contc^\pi(S,\H)$ is, as a vector space, just the algebraic direct sum $\oplus_{s\in S}^\alg \H_s \delta_s$. For technical reasons, we also need a certain special subspace of $\contc^\pi(S,\H)$.
For this we define:
\begin{align*}
  \H^0_s & \coloneqq \spn \pi_s\left(A_s\right) \H \sbe \H_s; \;\, \text{and} \\
  \contc^\pi\left(S,\H^0\right) & \coloneqq \left\{f \in \contc^\pi\left(S,\H\right) \mid f\left(s\right)\in \H_s^0\mbox{ for all }s \in S\right\}.
\end{align*}

Observe that, following from the definitions, $\H^0_s \sbe \H_s$ is dense for all $s$, and 
$\contc^\pi(S,\H^0)$ is exactly the image of $U_{\pi,0}$.
Hence we may see $U_{\pi,0}$ as a surjective linear map $U_{\pi,0}\colon \contc(\A) \odot \H \rightarrow \contc^\pi(S,\H^0)$. Furthermore, we now wish to construct a unitary $U_\pi$ from $U_{\pi,0}$ between certain Hausdorff completions of its domain and target spaces (see \cref{prop:fell-principle:unitary}). For this, however, we first need to define the adequate semi-inner product on $\contc^\pi(S,\H)$ (see \cref{prop:fell-principle:inner-prod}). 

\begin{lemma} \label{lemma:fell-principle:sesquilinear-form}
Let $\A = (A_s)_{s \in S}$ be a Fell bundle, and let $\pi = (\pi_s)_{s \in S}$ be a representation of $\A$ on $\H$. For any given $s, t \in S$, consider 
\[ \braket{x}{y}_{s,t} \coloneqq \sum_{i,j} \braket{v_i}{\pi_1''\left(P\left(a_i^*b_j\delta_{s^*t}\right)\right)w_j}, \]
where $x = \sum_i \pi_s(a_i)v_i \in \H_s^0$ and $y = \sum_j \pi_t(b_j)w_j \in \H_t^0$. Then $\braket{\cdot}{\cdot}_{s, t}$ is a well-defined sesquilinear form on $\contc^\pi(S, \H^0)$.

Furthermore, $\braket{x}{y}_{s, t} = \braket{x}{\pi_1''\left(1_{st^*}\right) y}$ for every $x \in \H^0_s$ and $y \in \H^0_t$.
\end{lemma}
\begin{proof}
  First, recall that $1_{s^*t}$ is the unit of $I_{s^*t,1}''$ (see \cref{def:pre:ideal-s-1}) and note that this is a central projection of $A_1''$. By \cref{prop:cond-exp} we have
  \[ P\left(a_s^*b_t\delta_{s^*t}\right) = \theta_{s^*t,1}''\left(a_s^*b_t 1_{s^*t}\right) \;\; \mbox{for all } a_s \in A_s \;\, \mbox{and } b_t \in A_t, \]
  where $\theta_{s^*t,1}''$ denotes the canonical embedding $A_{s^*t}'' \cdot I_{s^*t,1}'' \into A_1''$. Since these embeddings are compatible with representations (see \cite{Buss-Exel-Meyer:Reduced}*{Proposition~2.9}), we have
  \[ \pi_1''\left(\theta_{s^*t,1}''\left(a_s^*b_t 1_{s^*t}\right)\right) = \pi_{s^*t}''\left(a_s^*b_t 1_{s^*t}\right) = \pi_s\left(a_s\right)^*\pi_t\left(b_t\right) \pi_1''\left(1_{s^*t}\right). \]
  Thus, by the commutation relations given in \cref{prop:unit-s-commutation}~\cref{item:prop:unit-s-commutation:comm},
  \begin{align*}
    \braket{x}{y}_{s,t} & = \sum_{i,j}\braket{v_i}{\pi_1''\left(P\left(a_i^*b_j\delta_{s^*t}\right)\right) w_j} = \sum_{i,j}\braket{v_i}{\pi_1''\left(1_{s^*t}\right) \pi_s\left(a_i\right)^* \pi_t\left(b_j\right) w_j} \\
    & = \sum_{i,j}\braket{\pi_s\left(a_i\right) \pi_1''\left(1_{s^*t}\right) v_i}{\pi_t\left(b_j\right) w_j} = \sum_{i}\braket{\pi_s\left(a_i\right) \pi_1''\left(1_{s^*t}\right) v_i}{y} \\
    & = \sum_i \braket{\pi_s\left(a_i\right) v_i}{\pi_1''\left(1_{st^*}\right) y} = \braket{x}{\pi_1''\left(1_{st^*}\right) y},
  \end{align*}
  as desired. The fact that $\braket{\cdot}{\cdot}_{s, t}$ is a sesquilinear form is now immediate.
\end{proof}

We now assemble the sesquilinear forms from \cref{lemma:fell-principle:sesquilinear-form} in order to get a semi-inner product $\braket{\cdot}{\cdot}_\pi$ on $\contc^\pi(S, \H)$, as the following key proposition shows.
\begin{proposition} \label{prop:fell-principle:inner-prod}
Let $\A = (A_s)_{s \in S}$ be a Fell bundle, and let $\pi = (\pi_s)_{s \in S}$ be a representation of $\A$ on $\H$. Then
\[ \braket{x}{y}_\pi \coloneqq \sum_{s, t \in S} \braket{x\left(s\right)}{\pi_1''\left(1_{st^*}\right) y\left(t\right)}, \]
where $x, y \in \contc^\pi(S, \H)$, defines a semi-inner product on $\contc^\pi(S, \H)$. Furthermore,
\[ \braket{f_v}{g_w}_\pi = \braket{v}{\pi_1''\left(P\left(f^* \cdot g\right)\right)w} \]
for all $f, g \in \contc(\A)$ and $v, w \in \H$.
\end{proposition}
\begin{proof}
  By \cref{lemma:fell-principle:sesquilinear-form} it follows that $\braket{x}{y}_\pi = \sum_{s, t \in S} \braket{x(s)}{y(t)}_{s,t}$. Observe that, by construction and \cref{lemma:fell-principle:sesquilinear-form} again, for $f, g \in \contc(\A)$ and $v, w \in \H$, we have
  \begin{align*}
    \braket{f_v}{g_w}_\pi & = \sum_{s,t \in S}\braket{f_v\left(s\right)}{g_w\left(t\right)}_{s,t} = \sum_{s,t \in S}\braket{\pi_s\left(f\left(s\right)\right)v}{\pi_t\left(g\left(t\right)\right) w}_{s,t} \\
    &= \sum_{s,t \in S}\braket{v}{\pi_1''\left(P\left(f\left(s\right)^*g\left(t\right)\delta_{s^*t}\right)\right) w} = \braket{v}{\pi_1''\left(P\left(f^* \cdot g\right)\right) w},
  \end{align*}
  as stated in the claim. The form $\braket{\cdot}{\cdot}_\pi$ is clearly symmetric and linear, so in order to end the proof it is enough to show it is also positive semi-definite. For that, fix any $x \in \contc^\pi(S, \H)$ and observe that
  \[ \braket{x}{x}_\pi = \sum_{s, t \in S} \braket{x\left(s\right)}{\pi_1''\left(1_{st^*}\right) x\left(t\right)} = \sum_{s, t \in S} \braket{\pi_1''\left(1_{ts^*}\right) x\left(s\right)}{\pi_1''\left(1_{st^*}\right) x\left(t\right)} \geq 0 \]
  by \cref{prop:unit-s-commutation}~\cref{item:prop:unit-s-commutation:units}.
\end{proof}

In light of \cref{prop:fell-principle:inner-prod}, the following definition is only natural (compare with \cref{def:left-reg:l-two}).
\begin{definition} \label{def:fell-principle:l-two}
Given a Fell bundle $\A = (A_s)_{s \in S}$ we denote by \emph{$\ell^2_\pi(S, \H)$} the Hausdorff completion of $\contc^\pi(S, \H)$ with respect to the semi-inner product $\braket{\cdot}{\cdot}_\pi$.
\end{definition}

We briefly observe in passing that $\braket{\cdot}{\cdot}_\pi$ is a semi-inner product, that is, there may be non-zero elements $x \in \contc^\pi(S, \H)$ such that $\braket{x}{x}_\pi = 0$. Thus, $\ell^2_\pi(S, \H)$ is only the completion of a \emph{quotient} of $\contc^\pi(S, \H)$. However, the fibers $\H_s$ embed isometrically into $\ell^2(S, \H)$ (compare the following with \cref{lemma:n-p-intersection-fibers}):
\begin{lemma} \label{lemma:fell-principle:fibers}
  For any Fell bundle $\A = (A_s)_{s \in S}$ and representation $\pi$, the canonical map $\H_s^0 \to \contc^\pi(S,\H^0) \to \ell^2_\pi(S,\H)$ is isometric. In particular, it extends to an isometry $\H_s \cong \H_s \delta_s \into \ell^2_\pi(S,\H)$.
\end{lemma}
\begin{proof}
  The proof boils down to the fact that $s^*s$ is an idempotent, and hence $P$ acts as the identity on $A_{s^*s} \sbe A_1$. In fact, we get
  \[ \braket{v\delta_s}{v\delta_s}_\pi = \sum_{t, r \in S} \braket{\left(v\delta_s\right)\left(t\right)}{\pi_1''\left(1_{tr^*}\right) \left(v \delta_s\right)\left(r\right)} = \braket{v}{v} \]
  for every $v \in \H_s$ and $s \in S$.
\end{proof}
The following is a technical point that needs to be resolved for later reference.
\begin{lemma} \label{lemma:fell-principle:l-two-h-zero}
The Hilbert space $\ell^2(S, \H)$ is the Hausdorff completion of $\contc^\pi(S, \H^0)$ with respect to the semi-inner product $\braket{\cdot}{\cdot}_{\pi}$.
\end{lemma}
\begin{proof}
By \cref{def:fell-principle:l-two}, the Hilbert space $\ell^2_\pi(S, \H)$ is the Hausdorff completion of $\contc^\pi(S, \H)$ with respect to the semi-inner product $\braket{\cdot}{\cdot}_{\pi}$. Thus, in order to prove the claim it is enough to show the canonical map $\contc^\pi(S, \H^0) \sbe \contc^\pi(S, \H) \to \ell^2_\pi(S, \H)$ has dense image. 
But $\H^0_s \sbe \H_s$ is dense for every $s \in S$, and by \cref{lemma:fell-principle:fibers} $\H_s$ embeds isometrically into $\ell^2_\pi(S,\H)$, and hence so does $\H_s^0$. Therefore the image of $\contc^\pi(S, \H^0)$ in $\ell^2_\pi(S, \H)$ is dense in that of $\contc^\pi(S, \H)$, as desired.
\end{proof}

\begin{proposition} \label{prop:fell-principle:unitary}
  Let $\A = (A_s)_{s \in S}$ be a Fell bundle, and let $\pi = (\pi_s)_{s \in S}$ be a representation of $\A$ on $\H$. Then the map $U_{\pi,0}''$ from \cref{def:fell-principle:pre} induces a unitary operator $U_\pi \colon \ell^2(\A'') \otimes_{\pi_1^{''}} \H \to \ell^2_\pi(S,\H)$ making the diagram
  \begin{center}
    \begin{tikzcd}
      \contc\left(\A''\right) \odot \H \arrow{d}{} \arrow{r}{U_{\pi,0}''} & \contc^\pi\left(S, \H\right) \arrow{d}{} \\
      \ell^2\left(\A''\right) \otimes_{\pi_1^{''}} \H \arrow{r}{U_\pi} & \ell^2_\pi\left(S,\H\right),
    \end{tikzcd}
  \end{center}
  commute.
\end{proposition}
\begin{proof}
  By \cref{prop:fell-principle:inner-prod}, notice that
  \[ \braket{f_v}{g_w}_\pi = \braket{v}{\pi_1''\left(P\left(f^*\cdot g\right)\right) w} = \braket{f \otimes_{\pi_1''} v}{g \otimes_{\pi_1''}w}, \]
  that is, the semi-inner product $\braket{\cdot}{\cdot}_\pi$ is compatible with $U_{\pi,0}'' \colon \contc(\A'')\odot \H\to \contc^\pi(S,\H^0)$. By sesquilinearity the above equation extends to arbitrary elements of the tensor product, i.e., we get
  \[ \braket{U_{\pi,0}''\left(x\right)}{U_{\pi,0}''\left(y\right)}_\pi = \braket{x}{y} \;\; \mbox{for all } x,y \in  \contc\left(\A\right)\odot \H. \]
  The above equation implies the statement. Indeed, recall (see \cref{def:induced-reps,lemma:fell-principle:l-two-h-zero}) that $\ell^2(\A'') \otimes_{\pi_1''} \H$ and $\ell^2_\pi(S, \H)$ are, respectively, the Hausdorff completions of $\contc(\A'') \odot \H$ w.r.t. $\braket{\cdot}{\cdot}_{\ell^2(\A'') \otimes_{\pi_1''} \H}$ and $\contc^\pi(S, \H)$ w.r.t. $\braket{\cdot}{\cdot}_\pi$.
\end{proof}

\begin{remark}\label{rem:dense-subspace}
It follows from \cref{prop:fell-principle:unitary} that the canonical map $\contc(\A)\odot \H\to \ell^2(\A'')\otimes_{\pi_1''}\H$ has dense image because the latter is isomorphic to $\ell^2_\pi(S,\H)$ and the image of $U_{\pi,0}\colon \contc(\A)\odot \H\to \contc^\pi(S,\H)$ is dense as a subspace of $\ell^2_\pi(S,\H)$.
\end{remark}

We are now ready to provide a version of Fell's absorption principle for the context of inverse semigroups; this is the main result of the section. We start with a representation $\pi$ of $\A$, restrict it to the representation $\pi_1$ of $A_1=A$ and then induce it back to $\A$ via the regular representation $\Lambda$ on $\ell^2(\A'')$. This results in a description of the induced representation $\Lambda\otimes_{\pi_1''}\Id$ on $\ell^2(\A'')\otimes_{\pi_1''}\H$ in a more explicit way, identifying it with a representation on $\ell^2_\pi(S, \H)$, which is a quotient of a certain space of sections $S \to \H$.

\begin{theorem} \label{thm:fell-principle}
Let $\A = (A_s)_{s \in S}$ be a Fell bundle over an inverse semigroup $S$, and let $\pi = (\pi_s)_{s \in S}$ be a nondegenerate representation of $\A$. Consider the induced representation $\Lambda \otimes_{\pi_1^{''}} \Id = ((\Lambda \otimes_{\pi_1^{''}} \Id)_s)_{s \in S}$ given by
\begin{align}
(\Lambda \otimes_{\pi_1^{''}} \Id)_s \colon A_s & \rightarrow \L\left(\ell^2\left(\A''\right) \otimes_{\pi_1^{''}} \H\right), \label{eq:fell-principle:induced-rep} \\
(\Lambda \otimes_{\pi_1^{''}} \Id)_s & \left(a_s\right) \left(f \otimes v\right) \coloneqq \Lambda_s\left(a_s\right)\left(f\right) \otimes v, \nonumber
\end{align}
where $a_s \in A_s, f \in \ell^2(\A'')$ and $v \in \H$. Then the unitary isomorphism $U_\pi$ from \cref{prop:fell-principle:unitary} conjugates $\Lambda\otimes_{\pi_1''}\Id$ into the representation $\pi^\Lambda = (\pi^\Lambda_s)_{s \in S}$ given by
\begin{align}
\pi^\Lambda_s \colon A_s & \rightarrow \L\left(\ell^2_\pi\left(S, \H\right)\right), \label{eq:fell-principle:pi-lambda} \\
\pi^\Lambda_s & \left(a_s \delta_s\right)\left(v_t \delta_t\right) \coloneqq \pi_s\left(a_s\right) v_t \delta_{st}.
\nonumber
\end{align}
In other words, we have $\pi^\Lambda_s(a_s \delta_s) U_\pi = U_\pi (\Lambda \otimes_{\pi_1''} \id)(a_s\delta_s)$ for all $a_s \in A_s$ and $s \in S$, that is, $U_\pi$ intertwines $\pi^\Lambda$ and $\Lambda \otimes_{\pi_1''} \id$.
\end{theorem}
\begin{proof}
Given the machinery of \cref{prop:fell-principle:inner-prod,prop:fell-principle:unitary} the proof of this result boils down to a computation. Indeed, note that for all basic tensors $a_t \delta_t \otimes v \in \ell^2(\A'') \otimes_{\pi_1''} \H$ (recall \cref{lemma:n-p-intersection-fibers,lemma:fell-principle:fibers}), where $t \in S$, $a_t\in A_t$ and $v\in \H$, we have
\[ U_\pi \left(\Lambda_s \left(a_s \delta_s\right) \otimes_{\pi_1''} \id\right) \left(a_t \delta_t \otimes v\right) = U_\pi\left(a_s a_t \delta_{st} \otimes v\right) = \pi_{st}\left(a_s a_t\right) v \delta_{st}.
\]
Likewise,
\[ \pi^\Lambda_s\left(a_s \delta_s\right)\left(U_\pi \left(a_t \delta_t \otimes v\right) \right) = \pi^\Lambda_s\left(a_s \delta_s\right)\left(\pi_t\left(a_t\right) v \delta_t\right) = \pi_s\left(a_s\right) \pi_t\left(a_t\right) v \delta_{st}. \]

Using \cref{rem:dense-subspace} it follows that $\pi^\Lambda_s(a_s\delta_s) = U_\pi (\Lambda_s(a_s\delta_s) \otimes_{\pi_1''} \id) U_\pi^*$, in particular this shows that $\pi^\Lambda$ defined as in the statement gives a well-defined representation of $\A$ since $\Lambda \otimes_{\pi_1''} \id$ is, itself, a representation of $\A$.
\end{proof}



Our main use of Fell's absorption principle comes from the following consequence of \cref{thm:fell-principle}.
\begin{corollary} \label{cor:fell-principle:red}
Let $\A = (A_s)_{s \in S}$ be a Fell bundle over an inverse semigroup $S$, and let $\pi = (\pi_s)_{s \in S}$ be a representation of $\A$. Suppose that $\pi_1$ extends to a faithful representation of $A''$. Then the reduced cross-sectional \cstar{}algebra $\redalg{\A}$ is isomorphic to the \cstar{}algebra
\[ C^*\left(\left\{\pi^\Lambda_s\left(a_s\right) \; \mid \; a_s \in A_s \; \text{and} \; s \in S\right\}\right) \sbe \L\left(\ell^2_\pi\left(S, \H\right)\right), \]
generated by the image of the representation $\pi^\Lambda = (\pi^\Lambda_s)_{s \in S}$ given by \cref{thm:fell-principle}.
\end{corollary}
\begin{proof}
By definition, $\redalg{\A}$ is the \cstar{}algebra generated by the image of the regular representation $\Lambda\colon \A\to \L(\ell^2(\A''))$, so that $\ell^2(\A'')$ may be viewed as a faithful correspondence from $\redalg{\A}$ to $A_1''$, that is, it is a Hilbert module over $A_1''$ with a faithful representation of $\redalg{\A}$. Since $\pi_1''$ is assumed to be faithful, the induced representation $\Ind \pi_1'' = \Lambda \otimes_{\pi_1''} \id$ is also faithful as a representation of $\redalg{\A}$. And therefore so is the equivalent representation $\pi^\Lambda$ by~\cref{thm:fell-principle}, from which the result follows.
\end{proof}

We end the discussion about Fell's absorption principle with the following remarks.
\begin{remark}
Notice that, in \cref{cor:fell-principle:red}, with the same proof we may actually only assume that $\pi_1$ extends to a faithful representation of the \cstar{}algebra generated by $P(\contc(\A)) \sbe A_1''$.
\end{remark}
\begin{remark}
Observe that if $S$ is a group then all the discussion above reduces to Fell's absorption principle for discrete groups in the sense of \cref{subsec:fell-principle:groups}. In particular, in such case \cref{thm:fell-principle} reduces to \cref{prop:fell-absorption-groups}.
\end{remark}

\section{The approximation property} \label{sec:ap}
The present section introduces the central notion of the paper, namely the \emph{approximation property} of a Fell bundle $\A = (A_s)_{s \in S}$. The approach here taken generalizes the approximation property for Fell bundles over discrete groups, which was first defined by Exel in~\cite{Exel:Amenability} as follows.
\begin{definition} \label{def:fell-groups}
A Fell bundle $\A=(A_s)_{s\in G}$ over a discrete group $G$ has the \textit{approximation property} if there is a net $\{\xi_i\colon G\to A_1\}_{i \in I}$ of finitely supported sections such that
\begin{enumerate}[label=(\roman*)]
\item \label{def:fell-groups:bdd} $\sup_{i \in I} \|\sum_{s\in G}\xi_i(s)^*\xi_i(s)\| < \infty$;
\item \label{def:fell-groups:inv} and for every $s \in G$ and $a_s \in A_s$
\[ \sum_{t \in G}\xi_i(st)^* a_s \xi_i(t) \xrightarrow{i} a_s \]
in norm.
\end{enumerate}
\end{definition}

\begin{remark} \label{rem:ap-groups:bdd}
Observe that assertion \cref{def:fell-groups}~\cref{def:fell-groups:bdd} is actually stating that $\{\xi_i\}_{i \in I}$ is a bounded net in the Hilbert $A_1$-module $\ell^2(G,A_1)$, since $\|\xi_i\|_2^2 = \|\sum_{s\in G}\xi_i(s)^*\xi_i(s)\|$.

We use the word \emph{section} in \cref{def:fell-groups} to refer to functions $\xi \colon G \to A_1$. We do this because, in fact, in the setting of inverse semigroups the functions $\xi \colon S \to A_1$ will be honest sections of a certain bundle (see \cref{def:fell-isg}). The following notation will, hence, be useful throughout the rest of the text.
\end{remark}

\begin{definition} \label{def:gamma-sections}
Given a Fell bundle $\A = (A_s)_{s \in S}$ over $S$, we denote by $\gammac(S, \A)$ the space of finitely supported sections of the bundle $(A_{ss^*})_{s\in S}$, that is, functions $\xi \colon S \to A_1$ such that $\xi(s) \in A_{ss^*}$ for every $s \in S$.
\end{definition}
Working, as we are, in the setting of Fell bundles over inverse semigroups makes the definition of the approximation property a little more subtle. For it, recall that $I_{s,1} \idealin A_1$ is a two-sided closed ideal, and that $1_s$ denotes the unit of the von Neumann algebra $I_{s,1}''$ (recall \cref{def:pre:ideal-s-1,prop:unit-s-commutation}).
\begin{definition} \label{def:fell-isg}
A Fell bundle $\A=(A_s)_{s\in S}$ over an inverse semigroup with unit $1 \in S$ has the \textit{approximation property} if there is a net $\{\xi_i \colon S \rightarrow A_1\}_{i \in I} \sbe \gammac(S, \A)$ of finitely supported sections such that 
\begin{enumerate}[label=(\roman*)]
\item \label{def:fell-isg:bounded} $\sup_{i \in I} \|\sum_{p, t \in S} \xi_i(p)^*\xi_i(t) 1_{pt^*}\| < \infty$;
\item \label{def:fell-isg:invar} and for every $s \in S$ and $a_s \in A_s$
\[\sum_{p, t \in S} 1_{p\left(st\right)^*} \xi_i\left(p\right)^* a_s \xi_i\left(t\right) \to a_s \]
in norm.
\end{enumerate}
\end{definition}

We note in passing that assertion~\cref{def:fell-isg:bounded} in \cref{def:fell-isg} states that $\{\xi_i\}_{i \in I}$ is bounded in \emph{some} $\ell^2$-space of sections $\ell^2_\Gamma(S, \A)$ (as is the case for groups, recall \cref{rem:ap-groups:bdd}). Nevertheless, we will not introduce such a space as we do not need it in this paper. The following remarks are in order.
\begin{remark}
Note that, by \cref{prop:unit-s-commutation}, the element $1_s \in I_{s,1}''$ is central in $A_1''$, and hence, as $\xi(s) \in A_{ss^*} \sbe A_1''$, we have $1_{p(st)^*} \xi(p) = \xi(p) 1_{p(st)^*}$. This fact will be used virtually without mention in the sequel.
\end{remark}
\begin{remark}
Observe that \cref{def:fell-isg} does, in fact, generalize \cref{def:fell-groups}. Indeed, if $S = G$ is a group then it is routine to show that $1_{s^{-1}t} = 0$ unless $s = t$, in which case $1_{s^{-1}t} = 1 \in A_1''$. Therefore the double sums in \cref{def:fell-isg} simplify to simple sums, and we recover \cref{def:fell-groups} (recall that we are considering non-saturated Fell bundles, see \cref{rem:saturated-vs-semi}).
\end{remark}
\begin{remark}
\cref{def:fell-isg} indicates the expected behaviour of the inverse semigroup as \emph{bisections} over a groupoid. Indeed, the \emph{local} behaviour of the groupoid might be implemented by different bisections, and hence the nuclearity of $\redalg{\A}$ might be witnessed over possibly different choices of elements $s \in S$. \cref{def:fell-isg} fixes these choices, by summing over all possible labels $p, t \in S$ and then restricting to the appropriate ideals, namely $I_{p,t}''$ and $I_{st,p}''$ in \cref{def:fell-isg:bounded} and \cref{def:fell-isg:invar} respectively.
\end{remark}

%

We end the section with the following simple examples.
\begin{example} \label{ex:ap:group-actions}
Let $\alpha \colon G \curvearrowright X$ be an action of a discrete group $G$ on a compact Hausdorff space $X$. Then we may construct a canonical Fell bundle $\A_\alpha = (\A_{\alpha, g})_{g \in G}$ by setting $\A_{\alpha, g} \coloneqq \cont(X) \times \{g\}$. The multiplication and involution maps of $\A_\alpha$ are given by
\[ \left(f_1, g_1\right) \cdot \left(f_2,g_2\right) \coloneqq \left(f_1 \alpha_{g_1}\left(f_2\right),g_1g_2\right), \quad \left(f_1,g_1\right)^* \coloneqq \left(\alpha_{g_1^{-1}}(f_1^*),g_1^{-1}\right). \]
It is not hard to see that $\A_\alpha$ is then a Fell bundle over $G$, and that $\redalg{\A_\alpha} \cong \cont(X) \rtimes_{\text{red}} G$ and $\fullalg{\A_\alpha} \cong \cont(X) \rtimes_{\text{max}} G$. With this notation, it is also routine to show that $\A_\alpha$ has the approximation property, in the sense of \cref{def:fell-isg}, if, and only if, $\alpha$ is amenable (as an action, see~\cite{Brown-Ozawa:Approximations}*{Definition~4.3.5}). In fact, if $m_i \colon X \to \Prob{G}$ witnesses the amenability of $\alpha$ in the sense of \cite{Brown-Ozawa:Approximations}*{Definition~4.3.5}, then $\xi_i(g) \coloneqq [x \mapsto (m_i^x(g))^{1/2}]$ witnesses the approximation property of $\A_\alpha$ in the sense of \cref{def:fell-isg}. Likewise, given $\xi_i$ witnessing the approximation property, the maps $m_i^x(g) \coloneqq \xi_i(g)(x)^2$ witness the amenability of $\alpha$ (see \cref{prop:fell-isg-vs-groupoids} below for a similar argument).
\end{example}

\begin{example} \label{ex:ap:red-vs-ess}
Let $G$ be a discrete group, and let $S \coloneqq G \sqcup \{0\}$ be the inverse semigroup obtained from $G$ by adjoining a $0$ element. Consider the Banach spaces $A_g \coloneqq \cont[0,1]$ whenever $g \in G$ and $A_0 \coloneqq C(0,1]$. It is shown in~\cite{Kwaniewski2019EssentialCP}*{Example~4.7} that $\A = (A_s)_{s \in S}$ forms a Fell bundle over $S$. Moreover, it is also mentioned that both $\fullalg{\A}$ and $\redalg{\A}$ are $\cont[0,1]$-algebras with trivial fibers at $x \in (0, 1]$. Nevertheless, their fiber at $0$ may differ. Indeed, if $G$ is not amenable then $\fullalg{\A}$ has fiber $\fullalg{G}$ at $0$, whereas $\redalg{\A}$ has fiber $\C \oplus \redalg{G}$. On the contrary, if $G$ is amenable then both have $\redalg{G}$ as fiber at $0$, and hence $\fullalg{\A} \cong \redalg{\A}$ via the left regular representation, that is, $\A$ has the weak containment property. In fact, it is not hard to show that $\A$ has the approximation property if, and only if, $G$ is amenable (compare with \cref{thm:ap-implies-wk,ex:ap:red-vs-ess:2} below). Indeed, if $\{F_n\}_{n}$ is a F\o lner net for $G$ and $\{u_i\}_{i \in I} \sbe \contz(0, 1]$ is any approximate unit, then we can define
\[ \xi_{i, n}\left(0\right) \coloneqq u_i, \;\; \text{and} \;\; \xi_{i,n}\left(g\right) \coloneqq \left\{\begin{array}{rl} \frac{1}{\left|F_n\right|^{1/2}} 1_{\cont\left[0, 1\right]} & \text{if} \; g \in F_n, \\ 0 & \text{otherwise.} \end{array}\right.\]
It then follows from a computation that the net $\{\xi_{i,n}\}_{i,n}$ witnesses the approximation property for $\A$. Likewise, by \cite{Kwaniewski2019EssentialCP}*{Example~4.7}, if $G$ is non-amenable then $\fullalg{\A} \not\cong \redalg{\A}$. By \cref{thm:ap-implies-wk} below it follows that $\A$ does not have the approximation property.
\end{example}

\subsection{The approximation property for groupoids and inverse semigroups} \label{subsec:ap:ap-def}
This section builds upon \cref{ex:ap:group-actions,ex:ap:red-vs-ess} above, and deals with the case of Fell bundles over inverse semigroups that arise from Fell bundles over groupoids. Indeed, it is well known, and discussed below, that Fell bundles $\A = (A_s)_{s \in S}$ over inverse semigroups generalize in a canonical way Fell bundles over (possibly non-Hausdorff) locally compact, \'{e}tale groupoids. Recently, Kranz~\cite{Kranz:AP-groupooids-Fell-bundles} defined the \emph{approximation property} for Fell bundles over (second countable) \emph{Hausdorff} \'{e}tale groupoids as in \cref{def:fell-groupoids}. This section has the goal of proving that our approach generalizes that of~\cite{Kranz:AP-groupooids-Fell-bundles} in a canonical way (see \cref{prop:fell-isg-vs-groupoids}). In fact, we will not need our groupoids to be Hausdorff or second countable, nor will we need the Fell bundle to be separable or saturated, as is assumed in the main results from \cite{Kranz:AP-groupooids-Fell-bundles}.

We first recall that, given a locally compact groupoid $G$, a set $u \sbe G$ is a \emph{bisection} if the range and source maps $r,s \colon u \to G^0$ are homeomorphisms. Likewise, recall that $G$ is called \emph{\'{e}tale} if its topology admits a basis formed solely by open bisections. Moreover, the set $\Bis(G)$ of open bisections is well known to be an inverse semigroup. For instance, given $u, v \in \Bis(G)$ their product is defined to be
\[ u \cdot v \coloneqq \left\{ gh \; \mid \; g \in u, \, h \in v \; \text{and} \; s\left(g\right) = r\left(h\right) \right\}. \]

For the sequel, we refer the reader to~\cite{Kranz:AP-groupooids-Fell-bundles} for the appropriate definition of a Fell bundle over a groupoid, along with several examples of Fell bundles over locally compact Hausdorff \'etale groupoids. Moreover, recall that in the non-Hausdorff case $\contc(\A)$ denotes the space spanned by sections $G\to \A$ that are compactly supported and continuous in a bisection of $G$. Moreover, one may define $\contc(G, \rg^*(\A^0))$ in a similar vein. If $G$ is not Hausdorff, these sections are generally not globally continuous, but they are always Borel measurable functions. In any case, it makes sense to talk about pointwise or uniform convergence of these sections, as is done in the following.

\begin{definition} \label{def:fell-groupoids}
Let $\A = (A_g)_{g\in G}$ be a Fell bundle over a locally compact \'{e}tale groupoid $G$ with Hausdorff unit space $G^0$. We say that $\A$ has the \emph{approximation property} if there is a net $\{\xi_i\}_{i \in I} \sbe \contc(G,\rg^*(\A^0))$ such that
\begin{enumerate}[label=(\roman*)]
\item \label{def:fell-groupoids:1} $\sup_{i \in I} \sup_{x\in G^0}\|\sum_{g\in G^x} \xi_i(g)^*\xi_i(g)\| < \infty$;
\item \label{def:fell-groupoids:2} and for every $f \in \contc(\A)$ the net
\[ f_i \colon G \to \A, \;\; g \mapsto \sum_{h \in G^{\s(g)}} \xi_i\left(gh\right)^* f\left(g\right) \xi_i\left(h\right)=
\sum_{h \in G^{\rg(g)}} \xi_i\left(h\right)^* f\left(g\right) \xi_i\left(g^{-1}h\right)\]
converges to $f$ uniformly over compact sets of $G$.
\end{enumerate}
\end{definition}

The following definition first appeared in~\cite{BussExel:Fell.Bundle.and.Twisted.Groupoids}*{Definition~2.12}, but was based on the previous~\cite{Exel:Inverse_combinatorial}*{Proposition~5.4}. It precisely captures those inverse subsemigroups $S \sbe \Bis(G)$ that retain all the topological information of $G$.
\begin{definition} \label{def:wide-isg}
We say an inverse semigroup $S \sbe \Bis(G)$ is \emph{wide} if
\begin{enumerate}[label=(\roman*)]
  \item $S$ covers $G$, that is, $\cup_{u \in S} u = G$; and
  \item for any $g \in G$, $u, v \in S$, if $g \in u \cap v$ then there is some $w \in S$ such that $g \in w \sbe u \cap v$.
\end{enumerate}
\end{definition}
Note that if $S \sbe \Bis(G)$ then $S$ naturally acts on $G^0$. Indeed, given $u \in S$ and $x \in G^0$, the action of $u$ on $x$ is only defined if $s^{-1}(x) \cap u \neq \emptyset$, i.e., there is some $g \in u$ with $s(g) = x$. In such a case, $ux = r(g)$, as $g \in u$ is then uniquely determined, since $u$ is assumed to be a bisection. Therefore, if $S \sbe \Bis(G)$ we may consider the \emph{groupoid of germs $G^0 \rtimes S$} of the natural action of $S$ on $G^0$, that is, the set of equivalence classes $[s, x]$, where $s \in S$ and $s$ is defined at $x$, where $[s, x] = [t, y]$ if, and only if, $x = y$ and there is some idempotent $e \in E$ such that $x \in e$ and $se = te$. The groupoid structure of $X \rtimes S$ comes from the multiplication and involution maps
\[\left[s, ty\right] \cdot \left[t, y\right] \coloneqq \left[st, y\right] \;\; \text{and} \;\; \left[s, x\right]^{-1} \coloneqq \left[s^*, sx\right]. \]
We then call $[s,x] \in X \rtimes S$ the \emph{germ of $s$ at $x$}. Note that $G^0 \rtimes S$ is an \'etale locally compact groupoid, and that $s$ defines a bisection of $G^0 \rtimes S$. It follows that $[s, x]$ is the unique arrow in the usual groupoid of germs $G^0 \rtimes S$ contained in the bisection defined by $s$ and with source $x$.
\begin{proposition} \label{prop:wide}
Let $G$ be a locally compact étale groupoid, and let $S \sbe \Bis(G)$. The following conditions are equivalent:
\begin{enumerate}[label=(\roman*)]
  \item \label{prop:wide:groupoids} $G^0 \rtimes S \cong G$ as topological groupoids;
  \item \label{prop:wide:wide} $S$ is wide.
\end{enumerate}
In particular, $G^0 \rtimes \Bis(G) \cong G$.
\end{proposition}
\begin{proof}
The proof this is mostly given in~\cite{Exel:Inverse_combinatorial}*{Proposition~5.4}, 
see \cite{Kwaniewski2019EssentialCP}*{Proposition 2.2} for a complete proof.
\end{proof}

We now turn our attention to \cref{prop:fell-isg-vs-groupoids}, which shows that our approximation property for Fell bundles over inverse semigroups is, at least, as general as their counterparts over \'{e}tale groupoids. In order to prove it, however, we first need to set some notation and prove some preliminary results. Let $G$ be a locally compact étale groupoid, and fix any wide $S \sbe \Bis(G)$. Then, given a Fell bundle $\A = (A_g)_{g \in G}$ over $G$ we can construct a Fell bundle $\B = (B_u)_{u \in S}$, where $B_u$ is the Banach space of $\contz$-sections $u \to \A$, i.e., continuous functions $f \colon u \to \A$ that vanish at infinity and such that $f(g) \in A_g$ for every $g \in u$; see \cite{BussExel:Fell.Bundle.and.Twisted.Groupoids}*{Example 2.11} for further details on this construction.

\begin{lemma} \label{lemma:fell-isg-vs-groupoids}
Let $G$ be a locally compact étale groupoid, and let $S \sbe \Bis(G)$ be wide. With notations as above, given any $\eta \in \contc(G,\rg^*(\A^0))$ there are finitely many bisections $\{u_j\}_{j = 1}^n \sbe S$ and $\{\eta_j\}_{j = 1}^n \sbe \contc(u_j, \rg^*(\A^0))$ such that $\eta = \eta_1 + \dots + \eta_n$.
\end{lemma}
\begin{proof}
Recall that $\contc(G, \rg^*(\A^0))$ is, by definition, the set of sections $G \rightarrow \A^0$ that can be written as linear combinations of functions that are compactly supported in a bisection. Therefore, there are finitely many bisections $\{v_i\}_{i = 1}^{m} \sbe \Bis(G)$ and $\{\eta_i\}_{i = 1}^{m} \sbe \contc(v_i, \rg^*(\A^0))$ such that $\eta = \eta_1 + \dots + \eta_m$. The statement will follow once we prove that the bisections $v_i$ may be taken to be in $S$.

Since $S \sbe \Bis(G)$ is wide we have, for every $i$,
\[ K_i \coloneqq \supp\left(\eta_i\right) \sbe v_i \sbe G = \cup_{v \in S} v. \]
By hypothesis $K_i$ is compact, and hence we may choose a finite subcover, i.e., finitely many bisections $\{u_{i, \ell}\}^{n_i}_{\ell = 1} \sbe S$ such that $K_i \sbe \cup_{\ell = 1}^{n_i} u_{i, \ell}$. Now take any partition of unity $\{\varphi_{i, \ell}\}^{n_i}_{\ell = 1}$ subject to the cover $\{u_{i, \ell}\}^{n_i}_{\ell = 1}$, and let $\eta_{i, \ell} \coloneqq \eta_i \varphi_{i, \ell}$. Then $\{\eta_{i, \ell}\}_{i, \ell}$ and $\{u_{i, \ell}\}_{i, \ell}$ give the desired realization of $\eta$ as a linear combination of functions compactly supported in elements of $S$.
\end{proof}

We highlight the following result, which details the (expected) relation between the projections $1_s$ in \cref{def:fell-isg} and the germ relation. For the following and in the sequel, we shall take an étale groupoid $G$, fix a wide inverse semigroup $S\sbe\Bis(G)$ and identify $G$ with the groupoid of germs $G\cong G^0\rtimes S$ as in \cref{prop:wide}. We shall also tacitly use the convention that whenever a germ $[s, x]$ is written it is already implied that $s$ is defined at $x$, so that $s\cdot x$ makes sense.
\begin{lemma} \label{lemma:bisections-restriction}
Let $G$ be a locally compact \'{e}tale groupoid and let $S \sbe \Bis(G)$ be wide. Given a Fell bundle $\A = (A_g)_{g \in G}$ let $\B = (B_u)_{u \in S}$ be constructed as above. For any $p, t \in S$ and $x \in G^0$, the following assertions are equivalent:
\begin{enumerate}[label=(\roman*)]
\item \label{lemma:bisections-restriction:unit} $1_{p^*t}(x) = 1_{A_x''} \in \multiplieralg{A_x} \sbe A_x''$.
\item \label{lemma:bisections-restriction:coin} There is an idempotent $e \in E$ such that $x \in e \leq p^*p t^*t$ and $pe = te$.
\item \label{lemma:bisections-restriction:germ} $[p, x] = [t, x]$.
\end{enumerate}
\end{lemma}
\begin{proof}
Before we start the proof, we need to explain why the point evaluations in \cref{lemma:bisections-restriction:unit} make sense, that is, we need to explain why it makes sense to evaluate $1_s(x)$ for $s\in S$. Recall that $1_s$ is the unit of the multiplier \cstar{}algebra $\M(I_{s,1})\sbe I_{s,1}''$, and $I_{s,1}$ is a certain ideal of $B_1$ which, in turn, is the \cstar{}algebra of $\contz$-sections of the upper semi-continuous \cstar{}bundle $(A_x)_{x \in G^0}$. Hence $I_{s,1}$ also admits such a description as an upper semi-continuous bundle. In fact, one can see that $I_{s,1} \cong \contz(O_{s,1},(A_x)_{x \in O_{s,1}})$ for the open set $O_{s,1} = \cup_{e \leq s,1} e \sbe G^0$. Note, however, we do not actually need such description in this proof. We may then describe the multiplier \cstar{}algebra $\M(I_{1,s})$ as the \cstar{}algebra of sections $f \colon O_{s,1} \to (\M(A_x))_{x \in O_{s,1}}$ satisfying $f \cdot g\in \contz(O_{s,1},(A_x)_{x\in O_{s,1}})$ for all $g\in \contz(O_{s,1},(A_x)_{x\in O_{s,1}})$.

We now proceed with the actual proof. First note that, by construction, either $1_{p^*t}(x) = 1_{A_x''}$ or $1_{p^*t}(x) = 0$. Moreover, observe that if $1_{p^*t}(x) \neq 0$ then, by construction there is some $e \in E$ such that $x \in e$ and $e \leq p^*t$. Thus, $p^*p e = p^*p p^*t e = p^*t e = e$. Likewise, $e t^*t = p^*t t^*t e = e$, and hence $x \in e \leq p^*p t^*t$. Moreover, recall that, in general, $a \leq b$ if, and only if, $a^* \leq b^*$. Hence $e \leq (p^*t)^* = t^*p$, and
\[ pe = pp^* t t^*t e = t \left(t^* pp^* t\right)e = t \left(\left(p^*t\right)^* p^*t\right) e = te, \]
which proves that \cref{lemma:bisections-restriction:unit} implies \cref{lemma:bisections-restriction:coin}. For the reciprocal, let $e \in E$ witness \cref{lemma:bisections-restriction:coin}, and observe that then $p^*t e = p^*p e = e$, and therefore $1_{p^*t}(x) = 1_e(x) = 1_{A_x''}$, as claimed. The equivalence between \cref{lemma:bisections-restriction:coin} and \cref{lemma:bisections-restriction:germ} is well known (and routine to check), so we omit its proof.
\end{proof}

\begin{remark}
It follows from \cref{lemma:bisections-restriction} that for $1_{p^*t}(x) \neq 0$ it is \emph{not} enough that $px = tx$ (recall that $S$ canonically acts on $G^0$).
\end{remark}

We can now prove the approach taken in this paper generalizes that of~\cite{Kranz:AP-groupooids-Fell-bundles} to \'{e}tale groupoids that may be non-Hausdorff.
\begin{theorem} \label{prop:fell-isg-vs-groupoids}
Let $G$ be a locally compact \'{e}tale groupoid with Hausdorff unit space $G^0$, and let $S \sbe \Bis(G)$ be wide. 
If $\A$ is a Fell bundle over $G$ and $\B$ is the corresponding Fell bundle over $S$, then
$\A$ has the approximation property if, and only if, $\B$ has the approximation property.
\end{theorem}
\begin{proof}
First, let $\eta_i \in \contc(G,\rg^*(\A^0))$ witness the approximation property for $\A = (A_g)_{g \in G}$. By \cref{lemma:fell-isg-vs-groupoids} for every $i \in I$ there are $\{\eta_{i,j}\}_{j = 1}^{n_i} \sbe \contc(u_{i,j}, \rg^*(\A^0))$ where $u_{i, j} \in S$ and $\eta_i = \eta_{i,1} + \dots + \eta_{i, n_i}$. For convenience, let $K_{i, j} \coloneqq \supp(\eta_{i, j})$, and note that $K_{i, j}$ is compact. In this context, consider
\[
\xi_i \colon S \rightarrow B_1, \;\; \text{where} \;\, \xi_i\left(v\right)\left(r\left(g\right)\right) \coloneqq 
\left\{
\begin{array}{rl}
\eta_{i, j}\left(g\right) & \text{if} \;\; g \in v = u_{i, j}, \\
0 & \text{otherwise.}
\end{array}
\right. \]
We claim that the net $\{\xi_i\}_{i \in I}$ witnesses the approximation property of $\B$. First of all, note that $\xi_i$ is finitely supported as a section $S \rightarrow B_1$, as $\supp{(\xi)} \sbe \{u_{i, 1}, \dots, u_{i, n_i}\} \sbe S$. Furthermore, for every $g \in G$
\begin{equation} \label{eq:sum-xi-equals-eta}
\left(\sum_{g \in s \in S} \xi_i\left(s\right)\right)\left(r\left(g\right)\right) = \sum_{\substack{j = 1 \\ g \in u_{i,j}}}^{n_i} \xi_i\left(u_{i,j}\right)\left(r\left(g\right)\right) = \sum_{j = 1}^{n_i} \eta_{i, j} \left(g\right) = \eta_i\left(g\right)
\end{equation}
by construction. Moreover, fixing $i$ and $j \in \{1, \dots, n_i\}$, and putting $v \coloneqq u_{i, j}$, it follows that $\supp(\xi_i(v)) \sbe vv^*$ is compact. Indeed, the range map $r$ restricted to $v$ is a homeomorphism between
\begin{align*}
v \supset \supp{\left(\eta_{i, j}\right)} = & \left\{g \in v \; \colon \; \eta_{i, j}\left(g\right) \neq 0\right\} \xrightarrow{r|_v} \\
& \left\{r\left(g\right) \in vv^* \; \colon \; \eta_{i, j}\left(g\right) \neq 0\right\} = \supp{\left(\xi_i\left(v\right)\right)} \subset vv^*,
\end{align*}
and therefore $\xi_i(u) \in B_{uu^*}$ and $\xi_i \in \gammac(S, \B)$, as desired. We will now check that conditions~\cref{def:fell-isg:bounded} and~\cref{def:fell-isg:invar} in \cref{def:fell-isg} are also met. For~\cref{def:fell-isg:bounded}, recall that $S \sbe \Bis(G)$ is wide by assumption, and hence $G^0 \rtimes S \cong G$ by \cref{prop:wide}. We will use this fact to identify a germ $[s,x] \in G^0 \rtimes S$ as an arrow $[s,x] \in G$. In such fashion, recall \cref{lemma:bisections-restriction} and observe that
\begin{align*}
\sup_{i \in I} & \Bigg\|\sum_{p, t \in S} \xi_i\left(p\right)^*\xi_i\left(t\right) 1_{pt^*}\Bigg\| = \sup_{i \in I} \sup_{x \in G^0} \Bigg\|\sum_{p, t \in S} \left(\xi_i\left(p\right)^*\xi_i\left(t\right) 1_{pt^*}\right) \left(x\right) \Bigg\| \allowdisplaybreaks \\
& = \sup_{i \in I} \sup_{x \in G^0} \Bigg\|\sum_{\substack{p, t \in S \\ \left[p^*,x\right] = \left[t^*,x\right]}} \xi_i\left(p\right)\left(x\right)^* \xi_i\left(t\right)\left(x\right) \Bigg\| \allowdisplaybreaks \\
& = \sup_{i \in I} \sup_{x \in G^0} \Bigg\|\sum_{\substack{l, j = 1 \\ \left[u_{i,j}^*,x\right] = \left[u_{i,l}^*,x\right]}}^{n_i} \xi_i\left(u_{i,j}\right)\left(x\right)^* \xi_i\left(u_{i,l}\right)\left(x\right) \Bigg\| \allowdisplaybreaks \\
& = \sup_{i \in I} \sup_{x \in G^0} \Bigg\|\sum_{g \in G^x} \sum_{\substack{l, j = 1 \\ g \in u_{i,j} \cap u_{i,l}}}^{n_i} \xi_i\left(u_{i,j}\right)\left(r\left(g\right)\right)^* \xi_i\left(u_{i,l}\right)\left(r\left(g\right)\right) \Bigg\| \allowdisplaybreaks \\
& = \sup_{i \in I} \sup_{x \in G^0} \Bigg\|\sum_{g \in G^x} \sum_{\substack{l, j = 1 \\ g \in u_{i,j} \cap u_{i,l}}}^{n_i} \eta_{i,j}\left(g\right)^* \eta_{i,l}\left(g\right) \Bigg\|  \allowdisplaybreaks \\
& = \sup_{i \in I} \sup_{x \in G^0} \Bigg\| \sum_{g \in G^x} \eta_i\left(g\right)^* \eta_i\left(g\right)\Bigg\| < \infty
\end{align*}
by the assumptions on $\{\eta_i\}_{i \in I}$ and \cref{eq:sum-xi-equals-eta} (for the previous-to-last equality). The invariance condition \cref{def:fell-isg:invar} is proved similarly. Indeed, fix any section $b_s \in B_s$ and arrow $[s,x] \in s \sbe G$, and observe that
\begin{align*}
\Bigg(\sum_{p, t \in S} 1_{p\left(st\right)^*} & \xi_i\left(p\right)^* b_s \xi_i\left(t\right)\Bigg) \left(\left[s,x\right]\right) = \sum_{\substack{p, t \in S \\ \left[p^*, sx\right] = \left[(st)^*, sx\right]}} \xi_i\left(p\right)\left(sx\right)^* b_s\left(\left[s,x\right]\right) \xi_i\left(t\right)\left(x\right) \allowdisplaybreaks \\
& = \sum_{\substack{l,j = 1 \\ \left[u_{i,j}^*, sx\right] = \left[(su_{i,l})^*, sx\right]}}^{n_i} \xi_i\left(u_{i,j}\right)\left(sx\right)^* b_s\left(\left[s,x\right]\right) \xi_i\left(u_{i,l}\right)\left(x\right) \allowdisplaybreaks \\
& = \sum_{\substack{l,j = 1 \\ \left[u_{i,j}^*, sx\right] = \left[(su_{i,l})^*, sx\right]}}^{n_i} \eta_{i,j}\left(\left[u_{i,j}, u_{i,j}^*sx\right]\right)^* b_s\left(\left[s,x\right]\right) \eta_{i,l}\left(\left[u_{i,l}, u_{i,l}^*x\right]\right) \allowdisplaybreaks \\
& = \sum_{h \in G^{x}} \sum_{l,j = 1}^{n_i} \eta_{i,j}\left(\left[s, x\right] h\right)^* b_s\left(\left[s,x\right]\right) \eta_{i,l}\left(h\right) \allowdisplaybreaks \\
& = \sum_{h \in G^{x}} \eta_i\left(\left[s, x\right]h\right)^* b_s\left(\left[s,x\right]\right) \eta_i\left(h\right).
\end{align*}
Therefore, by \cref{eq:sum-xi-equals-eta} and the hypothesis on $\{\eta_i\}_{i \in I}$, it follows that for every compact set $K \sbe G$
\begin{align*}
\sup_{\left[s,x\right] \in K} & \left\| \Bigg(\sum_{p, t \in S} 1_{p\left(st\right)^*} \xi_i\left(p\right)^* b_s \xi_i\left(t\right) - b_s\Bigg) \left(\left[s,x\right]\right) \right\| \\
& = \sup_{\left[s,x\right] \in K} \left\|\sum_{h \in G^{x}} \eta_i\left(\left[s, x\right] h\right)^* b_s\left(\left[s,x\right]\right) \eta_i\left(h\right)-b_s\left(\left[s,x\right]\right)\right\| \xrightarrow{i} 0,
\end{align*}
which proves assertion \cref{def:fell-isg:invar}, as desired.

We now turn to the proof that if $\B$ has the approximation property then so does $\A$. To this end, for every idempotent $e \in E$ let $\{\varphi_{e,j}\}_{j} \sbe \contz(e, \A^0)$ be a quasi-central approximate unit such that $\varphi_{e,j}$ has compact support (in $e \sbe G^0$) and such that $\|\varphi_{e,j}\| \leq 1$. In addition, fix a net $\{\xi_i\}_{i \in I}$ witnessing the approximation property of $\B$ and positive numbers $\varepsilon_i \in (0,1)$ such that $\delta_i \coloneqq \varepsilon_i \cdot |\supp(\xi_i)|^2 \leq 1$ and $\delta_i \to 0$ when $i$ grows.\footnote{ \, For instance, if $I$ was countable then $\varepsilon_i \coloneqq (i \cdot |\supp(\xi_i)|)^{-2}$ would suffice.} In this context, consider the net $\{\eta_{i,j}\}_{i,j}$ given by
\[
\eta_{i,j}\left(g\right) \coloneqq \sum_{\substack{s \in S \\ g \in s}} \xi_i\left(s\right) \left(r\left(g\right)\right) \cdot \varphi_{ss^*,j}\left(r\left(g\right)\right).
\]
We shall prove that for every $i$ there is a $j = j(i)$ such that the subnet $\{\eta_{i,j(i)}\}_{i \in I}$ witnesses the approximation property of $\A$. To this end, observe $\eta_{i,j}$ is, by construction, a finite linear combination of functions that are continuous and have compact support on a bisection, and thus $\eta_{i,j} \in \contc(G,r^*(\A^0))$. Moreover, observe that for all $x \in G^0$ and $g \in G^x$
\begin{align}
\eta_{i,j}\left(g\right)^* \eta_{i,j}\left(g\right) & = \sum_{\substack{p,t \in S \\ g \in p \cap t}} \left(\varphi_{pp^*,j}\left(x\right) \cdot \xi_i\left(p\right)\left(x\right)^*\right)\left( \xi_i\left(t\right) \left(x\right) \cdot \varphi_{tt^*,j}\left(x\right)\right) \allowdisplaybreaks \nonumber \\
& \approx_{\varepsilon_i} \sum_{\substack{p,t \in S \\ g \in p \cap t}} \xi_i\left(p\right)\left(x\right)^* \xi_i\left(t\right) \left(x\right) \label{eqn:cond-eta-i-j-1}
\end{align}
for all large $j$. Therefore, since $\delta_i \leq 1$ for all large $j$, we have
\begin{align*}
\sup_{i \in I} \sup_{x \in G^0} \left\| \sum_{g \in G^x} \eta_{i,j}\left(g\right)^* \eta_{i,j}\left(g\right)\right\| & \approx_{1} \sup_{i \in I} \sup_{x \in G^0} \left\| \sum_{g \in G^x} \sum_{\substack{p, t \in S \\ g \in p \cap t}} \xi_i\left(p\right)\left(x\right)^* \xi_i\left(t\right)\left(x\right) \right\| \allowdisplaybreaks \\
& = \sup_{i \in I} \sup_{x \in G^0} \left\| \sum_{g \in G^x} \sum_{\substack{p, t \in S \\ \left[p^*, x\right] = \left[t^*,x\right] = g}} \xi_i\left(p\right)\left(x\right)^* \xi_i\left(t\right)\left(x\right) \right\| \allowdisplaybreaks \\
& = \sup_{i \in I} \left\| \sum_{p, t \in S} 1_{pt^*}\xi_i\left(p\right)^* \xi_i\left(t\right) \right\| < \infty.
\end{align*}
For the invariance condition, observe that given any $f \in \contc(\A)$, by assumption we may decompose it as a finite sum $f = \sum_{s \in S} b_s$, where $b_s$ is compactly supported in the bisection $s \in S$. By a simple approximation argument we may, without loss of generality, assume that $f = b_s \in B_s$ for some $s \in S$. In such case, observe that for every $g = [s,x] \in s$ and $h \in G$ such that $r(h) = r(g) = sx$, we have
\begin{align}
\eta_{i,j}\left(h\right)^* b_s\left(g\right) \eta_{i,j}(g^{-1}h) & = \sum_{\substack{p \in S \\ h \in p}} \sum_{\substack{t \in S \\ g^{-1} h \in t}} \varphi_{pp^*,j}\left(sx\right) \xi_i\left(p\right)\left(sx\right)^* b_s\left(g\right) \xi_i\left(t\right)\left(x\right) \varphi_{tt^*, j}\left(x\right) \nonumber \\
& \approx_{\varepsilon_i} \sum_{\substack{p \in S \\ h \in p}} \sum_{\substack{t \in S \\ g^{-1} h \in t}} \xi_i\left(p\right)\left(sx\right)^* b_s\left(g\right) \xi_i\left(t\right)\left(x\right) \label{eqn:cond-eta-i-j-2}
\end{align}
for all large $j$. Hence, and since $\delta_i = \varepsilon_i \cdot |\supp(\xi_i)|^2 \to 0$, we have that for all compact set $K \sbe G$
\begin{align*}
\sup_{[s,x] \in K} & \left\|\sum_{h \in G^{sx},} \eta_{i,j}\left(h\right)^* b_s\left(\left[s,x\right]\right) \eta_{i,j}\left(\left[s,x\right]^{-1}h\right) - b_s\left(\left[s,x\right]\right) \right\| \allowdisplaybreaks \\
& \approx_{\delta_i} \sup_{\left[s,x\right] \in K} \left\|\sum_{h \in G^{sx}} \sum_{\substack{p \in S \\ h \in p}} \sum_{\substack{t \in S \\ \left[s,x\right]^{-1} h \in t}} \xi_i\left(p\right)\left(sx\right)^* b_s\left(\left[s,x\right]\right) \xi_i\left(t\right)\left(x\right) - b_s\left(\left[s,x\right]\right)\right\| \allowdisplaybreaks \\
& = \sup_{\left[s,x\right] \in K} \left\|\sum_{\substack{p, t \in S \\ \left[p^*,sx\right] = \left[\left(st\right)^*,sx\right]}} \xi_i\left(p\right)\left(sx\right)^* b_s\left(\left[s,x\right]\right) \xi_i\left(t\right)\left(x\right) - b_s\left(\left[s,x\right]\right)\right\| \allowdisplaybreaks \\
& = \sup_{\left[s,x\right] \in K} \left\|\left(\sum_{p,t \in S} 1_{p\left(st\right)^*} \xi_i\left(p\right)^* b_s \xi_i\left(t\right) - b_s\right) \left(\left[s,x\right]\right)\right\| \xrightarrow{i} 0
\end{align*}
for all large $j$. Thus, for each $i$ let $j = j(i)$ be large enough so that \cref{eqn:cond-eta-i-j-1,eqn:cond-eta-i-j-2} are both satisfied. Then, by the above computations, the net $\{\eta_{i,j(i)}\}_{i \in I}$ witnesses the approximation property of $\A$, finishing the proof.
\end{proof}

\begin{remark}
Observe that, in \cref{prop:fell-isg-vs-groupoids}, if $G$ is assumed to be Hausdorff then the canonical expectation $P$ lands in $B_1$, as opposed to $B_1''$, i.e., $\B$ is closed (recall \cref{def:fell-bundle-closed}). In fact, by the arguments in the proof of \cref{prop:fell-isg-vs-groupoids}, all the sums in \cref{def:fell-isg} are actually in $B_1$, and we recover the definition of the approximation property for Hausdorff \'{e}tale groupoids in~\cite{Kranz:AP-groupooids-Fell-bundles}.
\end{remark}

\begin{remark} \label{rem:fell-isg-vs-groupoids:homo}
We note in passing that \cref{prop:fell-isg-vs-groupoids} also holds in a slightly more general setting. Indeed, if $T \sbe \Bis(G)$ is wide, and $\pi \colon S \rightarrow T$ is a homomorphism, then the canonical action of $T$ on $G^0$ induces via $\pi$ an action of $S$ on $G^0$. Given a Fell bundle $\A= (A_g)_{g \in G}$ over $G$, we translate it into a Fell bundle over $T$ and pull it back to $S$, so this gives a Fell bundle $\B = (B_s)_{s \in S}$. Assume that the action of $S$ on $G^0$ gives $G$ back, that is, $G^0 \rtimes S \cong G$. In such scenario, the conclusions of \cref{prop:fell-isg-vs-groupoids} also hold.
\end{remark}

\begin{remark}
In \cref{prop:fell-isg-vs-groupoids} it is apparent that the net $\{\xi_i\}_{i \in I}$ witnessing the approximation property for $\B$ depends on the choice of bisections $\{u_{i, j}\} \sbe S$ and $\eta_{i, j} \in \contc(u_{i,j}, r^*(\A^0))$ (see \cref{lemma:fell-isg-vs-groupoids}). This is related to the fact that $\contc(G, r^*(\A^0))$
 is a quotient of $\gammac\left(S, \B\right)$. More precisely, we have a canonical surjective linear map
\begin{align}
\Phi \colon \gammac\left(S, \B\right) & \twoheadrightarrow \contc\left(G, r^*\left(\A^0\right)\right), \nonumber \\
\Phi\left(\xi\right) \left(g\right) & \coloneqq \sum_{s \in S} \xi\left(s\right)\left(r\left(g\right)\right). \nonumber
\end{align}
In fact, one may view this as a bornological quotient map and describe its kernel, as is done in \cite{Buss-Meyer:Actions_groupoids}*{Proposition B.2}. Nevertheless, we are not going to need this, but notice that given any $\eta \in \contc(G, r^*(\A^0))$ the possibly different $\xi \in \gammac(S, \B)$ that are used in the proof of \cref{prop:fell-isg-vs-groupoids} represent the same element in the quotient space, that is, $\Phi(\xi) = \eta$.
\end{remark}

Let us now compare our approximation property with groupoid amenability. We would like to thank Julian Kranz for pointing us to \cite{Renault_AnantharamanDelaroche:Amenable_groupoids}*{Proposition~2.2.13}, which we use as the definition of \emph{(topological) amenability} of an \'{e}tale groupoid.
\begin{definition} \label{def:groupoid-amenable}
Let $G$ be a locally compact \'{e}tale groupoid with Hausdorff unit space $G^0$. We say $G$ is \emph{topologically amenable} if there is a net $\{\eta_i\}_{i \in I} \sbe \contc(G)$ such that the following conditions hold:
\begin{enumerate}[label=(\roman*)]
\item $\|\eta_i\|_2 \coloneqq \sup_{x\in G^0}(\sum_{h \in G^x} |\eta_i(h)|^2)^{1/2} \leq 1$ for all $i \in I$; and
\item $\sum_{h \in G^{r(g)}}  \overline{\eta_i(h)}\eta_i(g^{-1} h) \xrightarrow{i} 1$ for $g \in G$ uniformly in compact subsets of $G$.
\end{enumerate}
\end{definition}
\begin{remark}\label{rem-amenable-groupoid}
Normalizing the nets appropriately, one can actually drop the first boundedness condition (i) in the above definition (this is part of the statements in \cite{Renault_AnantharamanDelaroche:Amenable_groupoids}*{Proposition~2.2.13}). 
The functions $g\mapsto \varphi_i(g):=\sum_{h \in G^{r\left(g\right)}}  \overline{\eta_i(h)}\eta_i\left(g^{-1} h\right)$ appearing in (ii) are prototypical examples of \emph{positive type functions} on $G$, as defined and used to characterize amenability in \cite{Renault_AnantharamanDelaroche:Amenable_groupoids}.
\end{remark}

We shall use in the following the terminology that an action of an inverse semigroup $S$ on a \cstar{}algebra $A$ (or on a locally compact space $X$) has the approximation property if the associated Fell bundle over $S$ has the approximation property. Observe that in the following we do \textit{not} require $G$ to be globally Hausdorff, only its unit space.
\begin{corollary} \label{thm:ap:groupoids-inv-smgps}
Let $G$ be a locally compact \'{e}tale groupoid with Hausdorff unit space $G^0$. Then the following assertions are equivalent:
\begin{enumerate}[label=(\roman*)]
\item \label{thm:ap:groupoids-inv-smgps:amenable} $G$ is amenable.
\item \label{thm:ap:groupoids-inv-smgps:ap-1} The canonical action $\alpha \colon \Bis(G) \curvearrowright G^0$ has the approximation property.
\item \label{thm:ap:groupoids-inv-smgps:ap-2} The canonical action $\beta_S \colon S \curvearrowright G^0$ has the approximation property, where $S \sbe \Bis(G)$ is any wide subsemigroup.
\end{enumerate}
\end{corollary}
\begin{proof}
This is a special case of \cref{prop:fell-isg-vs-groupoids} for the trivial Fell bundle $\A$ over $G$, that is, the Fell bundle with $A_g\coloneqq\C$ for all $g\in G$. In this case the approximation property of $\A$ is equivalent to the amenability of $G$: the only small possible difference lies in the boundedness conditions required in the first item of \cref{def:fell-groupoids} and \cref{def:groupoid-amenable}. However in this case these conditions are redundant by \cref{rem-amenable-groupoid}.  

If $S\sbe \Bis(G)$ is a wide inverse semigroup of bisections of $G$, then the associated Fell bundle $\B$ over $S$ coming from $\A$ is the Fell bundle associated to the canonical action of $S$ on $G^0$. 
\end{proof}

Before we end this section, let us discuss the relationship between our approximation property with \emph{amenability} of inverse semigroups in the following sense.
\begin{definition}
We say that an inverse semigroup $S$ is \emph{\cstar{}amenable} if its canonical action on its spectrum $\dual E$ has the approximation property. 
\end{definition}

Let $G(S) \coloneqq \dual E\rtimes S$ be the universal groupoid of $S$; which is sometimes also called the \emph{Paterson groupoid} of $S$ (see~\cite{Paterson1999}*{Theorem~4.4.1}). Essentially, $G(S)$ is the universal groupoid for `ample' actions of $S$, that is, actions on locally compact spaces with clopen domains.\footnote{ \, In the language of Fell bundles, such actions give rise to Fell bundles $\A = (A_s)_{s \in S}$ such that $A_e \sbe A_1$ is a complemented ideal for every $e \in E$.} By \cref{thm:ap:groupoids-inv-smgps} an inverse semigroup $S$ is \cstar{}amenable if and only if $G(S)$ is amenable as a topological (étale) groupoid. By \cite{ExelStarling:Amenable_Actions_Inv_Sem}*{Theorem~3.3}, $S$ is \cstar{}amenable if and only if every action of $S$ on a locally compact Hausdorff space $X$ has the approximation property, that is, the étale groupoid $X\rtimes S$ is amenable. This combined with our previous results yields the following.
\begin{corollary}
Let $S$ be an inverse semigroup.
Then the following assertions are equivalent:
\begin{enumerate} [label=(\roman*)]
    \item $S$ is \cstar{}amenable, that is, the canonical action of $S$ on $\dual E$ has the approximation property;
    \item $G(S) = \dual E \rtimes S$ is amenable;
    \item every étale groupoid which is `represented' by $S$, in the sense that $S$ acts on $G^0$ and $G\cong G^0\rtimes S$, is amenable;
    \item every Fell bundle over an étale groupoid $G$ represented by $S$ has the approximation property.
\end{enumerate}
\end{corollary}
\begin{proof}
By the previous observations and results, it only remains to prove that if $G = G^0 \rtimes S$ is amenable, then every Fell bundle $\A$ over $G$ has the approximation property. But if $\{\eta_i\}_{i \in I} \sbe \contc(G)$ is a net witnessing the amenability of $G$ and $\{u_j\}_{j \in J} \sbe \contz(G^0,\A^0)$ is an approximate unit (as usual, with $0 \leq u_j \leq 1$), then a simple computation shows that the net $\{\xi_{i,j}\}_{i,j} \sbe \contc(G,\rg^*(\A^0))$ defined by $\xi_{i,j}(g) \coloneqq \eta_i(g)u_j(\rg(g))$ gives the approximation property of $\A$.
\end{proof}

\section{Tensor products of Fell bundles} \label{sec:tensor}
This section is dedicated to introducing tensor products of a Fell bundle $\A = (A_s)_{s \in S}$ over an inverse semigroup $S$ and any fixed \cstar{}algebra $B$. First we define such tensor products in \cref{subsec:tensor:definition} (see \cref{def:tensor-prod-fell-bundle}). Then, in \cref{subsec:tensor:compatibility}, we will prove some useful compatibility results between, on the one hand, the maximal tensor product and the full cross-sectional \cstar{}algebra (see \cref{prop:tensor-prod:full-comp}); and on the other hand, between the minimal tensor product and the reduced cross-sectional \cstar{}algebra (see \cref{prop:tensor-prod:red-comp}). Lastly, in \cref{prop:tensor-prod:ap} we show that the approximation property is preserved under taking tensor products with arbitrary \cstar{}algebras.

\subsection{Construction of tensor products} \label{subsec:tensor:definition}

Given a Hilbert bimodule $\E$ from $A$ to $B$ and another Hilbert bimodule $\F$ from $C$ to $D$, it is well known that we can form the \emph{minimal} and \emph{maximal} external tensor products $\E \mintensor \F$ from $A \mintensor C$ to $B \mintensor D$ and $\E \maxtensor \F$ from $A \maxtensor C$ to $B \maxtensor D$, see for instance \cite{Echterhoff-Kaliszewski-Quigg-Raeburn:Categorical}. Moreover, if both $\E$ and $\F$ are imprimitivity bimodules, then so are $\E\mintensor \F$ and $\E \maxtensor\F$. In fact, we shall only need here the case where $\F$ is a fixed \cstar{}algebra $C$ viewed as an (imprimitivity) $C$-$C$-bimodule. Using these constructions we may define the tensor product of a Fell bundle and an arbitrary \cstar{}algebra in the following way.
\begin{definition} \label{def:tensor-prod-fell-bundle}
Given a Fell bundle $\A = (A_s)_{s \in S}$ over an inverse semigroup $S$ and a \cstar{}algebra $B$, we define:
\begin{enumerate}[label=(\roman*)]
\item $\A\mintensor B$ to be the bundle over $S$ with fibres $(\A \mintensor B)_s \coloneqq A_s \mintensor B$;
\item $\A\maxtensor B$ to be the bundle over $S$ with fibres $(\A \maxtensor B)_s \coloneqq A_s \maxtensor B$.
\end{enumerate}
\end{definition}

The following lemma just checks that $\A \mintensor B$ and $\A \maxtensor B$ do carry a natural structure of Fell bundles.
\begin{lemma} \label{lemma:tensor-prod:multiplication}
With notation as above, for every $s, t \in S$ the map
\begin{align}
\left(A_s \algtensor B\right) \algtensor \left(A_t \algtensor B\right) \rightarrow A_{st} \algtensor B, \;\; \text{where} \;\, \left(a_s \tensor b_1\right) \algtensor \left(a_t \tensor b_2\right) \mapsto a_s a_t \tensor b_1b_2 \nonumber
\end{align}
extends to an embedding of Hilbert bimodules $(A_s \mintensor B) \mintensor (A_t \mintensor B) \rightarrow A_{st} \mintensor B$. Moreover, the map
\begin{align}
A_s \algtensor B \rightarrow A_{s^*} \algtensor B, \;\; \text{where} \;\, a_s \tensor b \mapsto a_s^* \tensor b^* \nonumber
\end{align}
extends to an isomorphism of Hilbert bimodules $(A_s \mintensor B)^* \congto A_{s^*} \mintensor B$. Likewise, the same results hold for the maximal tensor product $\maxtensor$.
\end{lemma}
\begin{proof}
The same proof given in \cite{Abadie:Tensor} for groups extends to inverse semigroups.
\end{proof}

Using \cref{lemma:tensor-prod:multiplication} we may now see $\A \mintensor B$ and $\A \maxtensor B$ as Fell bundles over $S$ whose bundle structure comes as a continuous extension of that of the ``algebraic'' bundle $\A \algtensor B=(A_s\algtensor B)_{s\in S}$. 

We end this section with two basic observations. First, notice that
\begin{equation}\label{lemma:tensor-prod:norms}    
\left\|x\right\|_{A_s \maxtensor B} = \left\| x^*x\right\|^{1/2}_{A_{s^*s} \maxtensor B} \;\; \text{and} \;\; \left\|y\right\|_{A_s \mintensor B} = \left\| y^*y\right\|^{1/2}_{A_{s^*s} \mintensor B}
\end{equation}
for every $x \in A_s \maxtensor B$ and $y \in A_s \mintensor B$. The second observation states that maximal tensor products of Fell bundles enjoy the universal property one would expect.
\begin{lemma} \label{lemma:tensor-prod:max-univ-prop}
Given two representations $\pi = (\pi_s)_{s \in S}$ of $\A = (A_s)_{s \in S}$ and $\rho \colon B \rightarrow \L(\H)$ of $B$ on the some Hilbert module $\H$ satisfying $\pi_s(a)\rho(b)=\rho(b)\pi_s(a)$ for all $a\in A_s$, $s\in S$ and $b\in B$, there is a unique representation $\pi \maxtensor \rho = (\pi_s \maxtensor \rho)_{s \in S}$ of $\A \maxtensor B$ on $\H$ extending the map $\pi_s \otimes \rho \colon \A \algtensor B \to \L(\H)$, $a\otimes b\mapsto \pi_s(a)\rho(b)$.
\end{lemma}
\begin{proof}
Observe that, whenever $e \in E$ is an idempotent, the maps $\pi_e \tensor \rho$ extend to $A_e \maxtensor B$ by the universal property of $\maxtensor$, since $A_e$ is then a \cstar{}algebra. The existence of $\pi_s \maxtensor \rho$ for general $s \in S$ then follows from \cref{lemma:tensor-prod:norms}, since the norm on $A_s \maxtensor B$ is induced from that of $A_{s^*s} \maxtensor B$. Moreover, it is routine to show that the resulting family $\pi \maxtensor \rho = (\pi_s \maxtensor \rho)_{s \in S}$ is a representation of $\A \maxtensor B$.
\end{proof}

In particular the above shows that given representations $\pi\colon\A\to \L(\H_1)$ and $\rho\colon B\to\L(\H_2)$ on two Hilbert modules $\H_1$ and $\H_2$ (over possibly different \cstar{}algebras), we get a representation $\rho\colon \A\maxtensor B\to \L(\H_1\maxtensor \H_2)$ sending $a\otimes b\mapsto \pi_s(a)\otimes\rho(b)$. If we look at the minimal tensor product of Hilbert modules instead, this is a prototypical representation of the minimal tensor product. Indeed, using the same idea of the above proof we see that the above representation factors through a representation of the minimal tensor product Fell bundle:
$$\pi\mintensor\rho\colon\A\mintensor B\to \L(\H_1\mintensor\H_2).$$

\subsection{Compatibility of tensor products and cross-sectional \cstar{}algebras} \label{subsec:tensor:compatibility}
We may now prove the expected compatibility conditions between the maximal tensor product and the universal cross-sectional \cstar{}algebra (see \cref{prop:tensor-prod:full-comp}); and likewise between their minimal/reduced companions (see \cref{prop:tensor-prod:red-comp}). We start with the approach to the universal objects, as it is, unsurprisingly, simpler.
\begin{proposition} \label{prop:tensor-prod:full-comp}
Let $\A = (A_s)_{s \in S}$ be a Fell bundle over an inverse semigroup $S$, and let $B$ be a \cstar{}algebra. Then the identity map on $\contc(\A) \algtensor B$ descends to a \Star{}isomorphism $\fullalg{\A} \maxtensor B \cong \fullalg{\A \maxtensor B}$.
\end{proposition} 
\begin{proof}
To prove the claim we shall construct mutually inverse \Star{}homomorphisms
\[ \Psi \colon \fullalg{\A} \maxtensor B \to \fullalg{\A \maxtensor B} \]
and
\[ \Phi \colon \fullalg{\A \maxtensor B} \to \fullalg{\A} \maxtensor B, \]
both of which will be induced by the identity map on $\contc(\A) \algtensor B$. 

Observe that, in order to construct $\Phi \colon \fullalg{\A \maxtensor B} \to \fullalg{\A} \maxtensor B$, by the universal property of the full cross-sectional \cstar{}algebra (see \cref{pre:fell:fullalg-univ}), it is enough to construct a representation $\phi = (\phi_s)_{s \in S}$ of $\A \maxtensor B$. For this, simply note that the identity map on $\contc(\A) \algtensor B$ induces the maps
\[ \phi_s \colon A_s \algtensor B \to \fullalg{\A} \maxtensor B, \;\; \text{where} \;\; a_s \tensor b \mapsto a_s \tensor b \]
(recall here \cref{lemma:n-p-intersection-fibers} and the fact that $\algalg{\A} \sbe \fullalg{\A}$ is dense). Note, by the universal property of the maximal tensor product $\maxtensor$, and since $A_e$ is a \cstar{}algebra, it follows that $\phi_e$ is $\max$-continuous whenever $e \in E$ is an idempotent. Likewise, whenever $x \in A_s \algtensor B$, by \cref{lemma:tensor-prod:norms}, we have  $\|x\|^2 = \|x^*x\|$, and hence, using the \cstar{}identity, we get
\begin{align*}
\left\|\phi_s\left(x\right)\right\|^2 & = \left\|\phi_s\left(x\right)^* \phi_s\left(x\right)\right\| = \left\| \left(\phi_e \tensor \id\right) \left(\sum_{k, l = 1}^n a_k^*a_l \tensor b_k^* b_l \right)\right\| \\
& \leq \left\|\sum_{k, l = 1}^n a_k^* a_l \tensor b_k^* b_l\right\| = \left\|x^*x\right\| = \left\|x\right\|^2
\end{align*}
for all $x = \sum_{k = 1}^n a_k \tensor b_k \in A_s \algtensor B$. Therefore, the maps $\phi_s$ are $\max$-continuous for every $s \in S$, and hence extend to maps $\phi_s \colon A_s \maxtensor B \to \fullalg{\A} \maxtensor B$ that we denote with the same symbol by abuse of notation. It is now routine to see that $\phi = (\phi_s)_{s \in S}$ forms a representation of $\A \maxtensor B$. Thus, its integrated form $\Phi \colon \fullalg{\A \maxtensor B} \to \fullalg{\A} \maxtensor B$ is a \Star{}homomorphism as well (see \cref{pre:fell:fullalg-univ}).

Consider the \emph{unitization Fell bundle} $\tilde\A$, that is, consider $\tilde A_s \coloneqq A_s$ if $s\not=1$ and $\tilde A_1$ to be the unitization of $A_1$ as a \cstar{}algebra. Notice that $\A$ is a Fell bundle ideal of $\tilde\A$ as defined in \cite{Kwasniewski-Meyer:Pure_infiniteness}. Likewise, $\A \maxtensor B$ is also a Fell bundle ideal of $\tilde\A\maxtensor \tilde B$. So by \cite{Kwasniewski-Meyer:Pure_infiniteness}*{Proposition~4.12} we get a canonical embedding
$$C^*_\max(\A\maxtensor B)\into C^*_\max(\tilde\A\maxtensor \tilde B)$$
extending the canonical inclusion $\A\maxtensor B\into \tilde\A\maxtensor \tilde B$.
Now we consider the maps
\[ \rho_s \colon A_s \to \fullalg{\tilde\A \maxtensor \tilde B}, \;\; \quad \;\; a_s \mapsto a_s \delta_s \tensor 1_{B},  \]
and
\[ \pi \colon B \to \fullalg{\tilde\A \maxtensor \tilde B}, \;\; \quad \;\; b \mapsto 1_{A_1} \tensor b, \]
It is straightfoward to check that $\rho$ is a representation of $\A$. By the universal property of the full cross-sectional \cstar{}algebra (see \cref{pre:fell:fullalg-univ}) we obtain that the integrated form
\[ \rho \colon \fullalg{\A} \to \fullalg{\tilde\A \maxtensor \tilde B} \]
is a \Star{}homomorphism that extends $\rho_s$ in every fiber $A_s\into \fullalg{\A}$. Likewise, $\pi$ is a \Star{}homomorphism. Furthermore, note that $\rho$ and $\pi$ have commuting ranges. Indeed, for all $a_s \in A_s$ and $b \in B$
\begin{align*}
\rho\left(a_s\right) \pi\left(b\right) & = \left(a_s \delta_s \tensor 1_B\right) \cdot \left(1_{A_1} \tensor b\right) = a_s \delta_s \tensor b = \pi\left(b\right) \rho\left(a_s\right).
\end{align*}
Thus, by the universal property of the maximal tensor product $\maxtensor$, there is a \Star{}homomorphism $\Psi \colon \fullalg{\A} \maxtensor B \to \fullalg{\tilde\A \maxtensor \tilde B}$ such that $\Psi(x \tensor b) = \rho(x) \pi(b)$ for all $x \in \fullalg{\A}$ and $b \in B$. In particular, for every $a_s \delta_s \tensor b \in A_s \delta_s \algtensor B \sbe \contc(\A) \algtensor B$, we have
\[ \Psi\left(a_s \delta_s \tensor b\right) = \rho(x) \pi(b) = a_s \delta_s \tensor b, \]
and hence $\Psi$ is induced by the identity map on $\contc(\A) \algtensor B$. In particular, the image of $\Psi$ is contained in $\fullalg{\A \maxtensor B}\sbe \fullalg{\tilde\A \maxtensor \tilde B}$, so we may see it as a \Star{}homomorphism $\Psi\colon \fullalg{\A} \maxtensor B \to \fullalg{\A \maxtensor B}$, as desired.

It is now routine to show that $\Psi$ and $\Phi$ are mutually inverses. Therefore any one of them gives the desired isomorphism $\fullalg{\A} \maxtensor B \cong \fullalg{\A \maxtensor B}$ induced by the identity on $\contc(\A) \algtensor B$.
\end{proof}

\begin{remark}
We briefly observe that the proof of \cref{prop:tensor-prod:full-comp}, as expected, heavily relies on the universal properties of both the full cross-sectional \cstar{}algebra of a Fell bundle and of the maximal tensor product.
\end{remark}

\begin{remark}
In the following results, we shall be working with several Fell bundles at the same time. Therefore, given a Fell bundle $\A = (A_s)_{s \in S}$ we shall denote the canonical conditional expectation $P \colon \contc(\A) \to A_1''$ by $P_\A$ instead. Likewise, we label other objects assigned to $\A$ in a similar way, like the ideal $\NN_P$, which will be denoted by $\NN_\A$ (recall \cref{eq:ideal-n-p}), and $\pi_\A$ will be the canonical quotient map
\[ \pi_\A \colon \contc\left(\A\right) \twoheadrightarrow \contc\left(\A\right)/\NN_\A = \algalg{\A}. \]
\end{remark}

The analogous result to \cref{prop:tensor-prod:full-comp} for the reduced/minimal objects is technically harder to prove (see \cref{prop:tensor-prod:red-comp}). Indeed, we will have to construct a unitary operator intertwining two representations (this is akin to the proof of \cref{thm:fell-principle}). Moreover, there is a subtlety that appears in this context, as it does in the context of \cstar{}bundles (see~\cites{Kranz:AP-groupooids-Fell-bundles,MR3383622} and references therein). We will, however, leave this discussion to after the proof of \cref{prop:tensor-prod:red-comp} (see \cref{cor:tensor-prod:red-comp:cond-exp,rem:tensor-prod:problem-comp}).

The following lemma partially describes a canonical faithful conditional expectation of $\redalg{\A} \mintensor B$. In particular, observe that the following describes the canonical conditional expectation $P_{\A \mintensor B}$ of $\A \mintensor B$ when applied to the fibers $A_s \delta_s \algtensor B \sbe \contc(\A) \algtensor B$ (recall \cref{lemma:fell-principle:fibers}).
\begin{lemma} \label{lemma:tensor-prod:red-cond-exp}
Given a Fell bundle $\A = (A_s)_{s \in S}$, the map
\[ P_\A \tensor \id_B \colon \contc(\A) \algtensor B \to A_1'' \algtensor B \]
coincides with 
$$P_{\A \mintensor B}\colon \contc(\A\mintensor B)\to (A_1\mintensor B)''$$
if we view $\contc(\A) \algtensor B\cong\contc(\A\algtensor B)\into \contc(\A\mintensor B)$ and
$A_1''\algtensor B\into (A_1\mintensor B)''$, that is, we have
$$(P_\A \tensor \id_B)(f \tensor b) = P_{\A \mintensor B}(f \tensor b)\quad\mbox{for every }f \in \contc(\A),\, b \in B.$$
In particular it follows that $\NN_\A\algtensor B\sbe \NN_{\A\mintensor B}$.
\end{lemma}
\begin{proof}
For the first assertion, it is enough to show that $P_A\otimes\Id(a_s\delta_s\otimes b)=P_{\A\mintensor B}(a_s\delta_s\otimes b)$ for all $a_s\in A_s$, $b\in B$. First notice that $I_{s,1}^{\A\mintensor B}=I_{s,1}^\A\mintensor B$. Recall that $P_\A(a_s\delta_s)=\lim \theta_{s,1}^\A(a_s u_i)$, where $(u_i)$ is an approximate unit for $I_{s,1}$ and the limit is taken with respect to the strict topology of $\M(I_{s,1})$. Now, if $(v_j)$ is an approximate unit for $B$, then $(u_i\otimes v_j)$ is an approximate unit for $I_{s,1}\otimes B$. Therefore,
$$P_{\A\mintensor}(a_s\delta_s\otimes b)=\lim \theta_{s,1}^{\A\mintensor B}(a_s u_i\otimes b v_j)=\lim \theta_{s,1}^\A(a_s u_i)\otimes b=P_\A(a_s\delta_s)\otimes b.$$
This proves the first assertion, and the second follows directly from this and the definition of the ideals $\NN_\A$ and $\NN_{\A\mintensor B}$.
\end{proof}

\begin{proposition} \label{prop:tensor-prod:red-comp}
Let $\A = (A_s)_{s \in S}$ be a Fell bundle over an inverse semigroup $S$, and let $B$ be any \cstar{}algebra. Then canonical map  $\contc(\A) \algtensor B\to \contc(\A\mintensor B)$ induces a \Star{}isomorphism $\redalg{\A} \mintensor B \congto \redalg{\A \mintensor B}$.
\end{proposition}
\begin{proof}
Using \cref{lemma:tensor-prod:red-cond-exp} and the canonical isomorphism
$$\algalg\A\algtensor B=\left(\contc(\A)/\NN_\A\right)\algtensor B\cong (\contc(\A)\algtensor B)/(\NN_\A\algtensor B),$$
it follows that the homomorphism
$\contc(\A) \algtensor B\cong\contc(\A\algtensor B)\into \contc(\A\mintensor B)$
factors through a homomorphism
$$\phi_0\colon \algalg\A\algtensor B\to \algalg{\A\mintensor B}\sbe \redalg{\A\mintensor B}.$$
which clearly has dense range, as the image contains all elementary tensors of the form $a_s\delta_s\otimes b$ with $a_s\in A_s$, $s\in S$ and $b\in B$.

We now prove that $\phi_0$ extends to the desired isomorphism $\phi \colon \redalg{\A} \mintensor B \congto \redalg{\A \mintensor B}$. In order to do this, fix a faithful nondegenerate representation $\pi_1 \colon A_1'' \to \L(\H_A)$ of $A_1''$ -- one can take here, e.g., the universal representation of $A_1$ -- and a faithful nondegenerate representation $\sigma \colon B \to \L(\H_B)$ of $B$. Now, consider the induced representations
\[ \Ind\left(\pi_1\right) \tensor \sigma \colon \redalg{\A} \mintensor B \to \L\left(\left(\ell^2\left(\A''\right) \tensor_{\pi_1''} \H_A\right) \tensor \H_B\right) \]
and
\[ \Ind\left(\pi_1 \tensor \sigma\right) \colon \redalg{\A \mintensor B} \to \L\left(\ell^2\left(\A \mintensor B\right)'' \tensor_{\left(\pi_1 \tensor \sigma\right)''} \left(\H_A \tensor \H_B \right)\right), \]
which are given as in \cref{def:induced-reps}. Since $\pi_1$ and $\sigma$ are faithful, so are $\Ind(\pi_1) \tensor \sigma$ and $\Ind(\pi_1 \tensor \sigma)$. Consider the operator
\begin{align} \label{eq:uni-op:inter-induc-reps}
U \colon \ell^2\left(\A \mintensor B\right)'' \tensor_{\left(\pi_1 \tensor \sigma\right)''} \left(\H_A \tensor \H_B \right) & \to \left(\ell^2\left(\A''\right) \tensor_{\pi_1''} \H_A\right) \tensor \H_B, \nonumber \\
f \tensor b \tensor v \tensor w & \mapsto \left(f \tensor v\right) \tensor \sigma\left(b\right) w 
\end{align}
for $f \in \contc(\A), b \in B, v \in \H_A$ and $w \in \H_B$. 
Notice that the elementary tensors of the form $f \tensor b \tensor v \tensor w$ span a dense subspace of the domain of $U$ by \cref{rem:dense-subspace}. Likewise, the $\left(f \tensor v\right) \tensor \sigma\left(b\right) w$ are dense in the codomain of $U$.
However, we need to prove that $U$ is well defined, that it is bounded, and extends to the whole domain. We will check this by showing at the same time that $U$ is an isometry. Indeed, given $f_1,f_2\in \contc(\A)$, $v_1,v_2\in \H_A$ and $w_1,w_2\in \H_B$, we compute
\begin{align*}
&\braket{f_1\otimes b_1\otimes v_1\otimes w_1}{f_2\otimes b_2\otimes v_2\otimes w_2}\\
&=\braket{v_1\otimes w_1}{\pi_1\otimes\sigma(P_{\A\mintensor B}((f_1\otimes b_1)^*(f_2\otimes b_2)))(v_2\otimes w_2}\\
&=\braket{v_1\otimes w_1}{\pi_1(P_{\A}(f_1^*f_2))v_2\otimes \sigma(b_1^*b_2)w_2}\\
&=\braket{v_1}{\pi_1(P_{\A}(f_1^*f_2))v_2}\braket{w_1}{\sigma(b_1^*b_2)w_2}\\
&=\braket{f_1\otimes v_1}{f_2\otimes v_2}\braket{\sigma(b_1)w_1}{\sigma(b_2)w_2}\\
&=\braket{(f_1\otimes v_1)\otimes\sigma(b_1)w_1}{(f_2\otimes v_2)\otimes \sigma(b_2)w_2}.
\end{align*}
This shows that $U$ is well-defined and extends to an isometry. And since it has dense range, it is a unitary.
Moreover, we claim that 
\begin{equation} \label{eq:tensor-prod:red-comp:intertwine}
U^* \left(\Ind\left(\pi_1\right) \tensor \sigma\right) \left(x\right) U = \Ind\left(\pi_1 \tensor \sigma\right) \left(\phi_0\left(x\right)\right)
\end{equation}
for all $x \in \contc(\A)/\NN_\A \algtensor B$. The latter equation follows from a computation on elementary tensor that we carry out for the sake of the reader. In fact, for every $a_s \in A_s$ and  $c\in B$, we have
\begin{align*}
U\big(\Ind\left(\pi_1 \tensor \sigma\right) & \left(\phi_0\left(a_s \delta_s \tensor c\right)\right)  \left(f \tensor b \tensor v \tensor w\right)\big) \\
& = U\left(\Lambda_{\A \mintensor B}\left( a_s \delta_s \tensor c\right) \left(f \tensor b\right) \tensor v \tensor w\right) \\
& = U\left(a_s \delta_s f \tensor cb \tensor v \tensor w\right) = a_s \delta_s f \otimes v\otimes \tensor \sigma\left(cb\right) w.
\end{align*}
Likewise,
\begin{align*}
\left( \left(\Ind\left(\pi_1\right) \tensor \sigma\right) \left(a_s \delta_s \tensor c\right) U\right) & \left(f \tensor b \tensor v \tensor w\right) \\
& = \left(\left(\Ind\left(\pi\right) \tensor \sigma\right) \left(a_s \delta_s \tensor c\right)\right) \left(f \tensor v \tensor \sigma\left(b\right) w\right) \\
& = a_s \delta_s f \tensor v \tensor \sigma\left(cb\right) w
\end{align*}
which, by linearity, proves \cref{eq:tensor-prod:red-comp:intertwine} for all $x \in \contc(\A)/\NN_\A \algtensor B$. 
But then it follows that $\phi_0$ is isometric with respect to the reduced norms on $\redalg\A\mintensor B$ and $\redalg{\A\mintensor B}$, and therefore it extends to an isomorphism $\phi\colon \redalg{\A} \mintensor B \congto \redalg{\A \mintensor B}$, as desired.
\end{proof}

In case $A_1$ is well-behaved \cref{prop:tensor-prod:red-comp} can be improved. Indeed, recall that we have already used the canonical inclusion $\iota \colon A_1'' \algtensor B \hookrightarrow (A_1 \mintensor B)''$, $a \tensor b \mapsto a \tensor b$ in \cref{lemma:tensor-prod:red-cond-exp}. However, we should recall here that this map is not min-continuous in general. This is related to the \emph{locally reflexivity} of $A_1$.
\begin{lemma}[see \cite{Brown-Ozawa:Approximations}*{Proposition~9.2.5}] \label{lemma:tensor-prod:red-locally-refl}
A \cstar{}algebra $A$ is \emph{locally reflexive} if for every \cstar{}algebra $B$ the canonical inclusion $\iota \colon A'' \algtensor B \hookrightarrow (A \mintensor B)''$ is min-continuous, and hence extends to an inclusion $\iota \colon A'' \mintensor B \hookrightarrow (A \mintensor B)''$.
\end{lemma}
\begin{corollary} \label{cor:tensor-prod:red-comp:cond-exp}
Let $\A = (A_s)_{s \in S}$ be a Fell bundle over an inverse semigroup $S$, and let $B$ be any \cstar{}algebra. Suppose that $A_1$ is locally reflexive. Then the identity map on $\contc(\A) \algtensor B$ induces a \Star{}isomorphism $\phi \colon \redalg{\A} \mintensor B \congto \redalg{\A \mintensor B}$ such that the diagram
\begin{center}
\begin{tikzcd}
\redalg{\A} \mintensor B \arrow{d}{P_{\A} \tensor \id_B} \arrow{r}{\phi} & \redalg{\A \mintensor B} \arrow{d}{P_{\A \mintensor B}} \\
A_1'' \mintensor B \arrow[hook]{r}{\iota} & \left(A_1 \mintensor B\right)''
\end{tikzcd}
\end{center}
commutes.
\end{corollary}
\begin{proof}
\cref{lemma:tensor-prod:red-cond-exp} can be reformulated by saying that the diagram
\begin{center}
\begin{tikzcd}
\left(\contc\left(\A\right)/\NN_{\A}\right) \algtensor B \arrow{d}{P_{\A} \tensor \id_B} \arrow{r}{\phi} & \redalg{\A \mintensor B} \arrow{d}{P_{\A \mintensor B}} \\
A_1'' \algtensor B \arrow[hook]{r}{\iota} & \left(A_1 \mintensor B\right)'',
\end{tikzcd}
\end{center}
commutes. This holds for every $\A$ and $B$, regardless of any property $A_1$ might or might not have. Now, if $A_1$ is locally reflexive, then \cref{lemma:tensor-prod:red-locally-refl} shows that $\iota$ extends to an inclusion $\iota \colon A_1'' \mintensor B \to (A_1 \mintensor B)''$ that makes the diagram in the statement commute.
\end{proof}

\begin{remark} \label{rem:tensor-prod:problem-comp}
By our discussions above, we see that $P_{\A \mintensor B} = P_{\A} \mintensor \id_B$ on the fibers $A_s \delta_s \algtensor B$ (see \cref{lemma:tensor-prod:red-cond-exp}), but this does not mean that $P_{\A \mintensor B} = P_{\A} \mintensor \id_B$ globally unless $A_1$ is assumed to be locally reflexive. A related subtlety already appeared in works of Kranz~\cite{Kranz:AP-groupooids-Fell-bundles}, LaLonde~\cite{MR3383622}, and previous work of Kirchberg-Wassermann~\cite{Kirchberg-Wassermann:Operations}, albeit in a different guise, which is related to continuity of $C^*$-bundles.

Let $\A = (A_x)_{x \in X}$ be an upper semi-continuous (u.s.c. for short) \cstar{}bundle over a locally compact Hausdorff space $X$. Moreover, let $S$ be the set of open sets of $X$, and let $\B = (B_u)_{u \in S}$ be the associated Fell bundle over $S$ (see the discussion before \cref{lemma:fell-isg-vs-groupoids}). Then, for any \cstar{}algebra $C$, we have $(\B \mintensor C)_u \coloneqq B_u \mintensor C$ defines a Fell bundle structure that does \emph{not} appear from a u.s.c. \cstar{}bundle over $X$. In fact, it appears from the bundle structure given by $\A \mintensor C \coloneqq (A_x \mintensor C)_{x \in X}$, which defines a, not necessarily u.s.c., \cstar{}bundle. In addition, note that $\A \mintensor C$ is u.s.c. for every \cstar{}algebra $C$ precisely when $A_1$ is exact (see~\cite{Kirchberg-Wassermann:Operations}*{Lemma~2.3} and recall that exact \cstar{}algebras are always locally reflexive~\cite{Brown-Ozawa:Approximations}).
\end{remark}

We end the discussion on tensor products of Fell bundles with the following proposition, which shows that they behave appropriately with respect to the approximation property. 
\begin{proposition} \label{prop:tensor-prod:ap}
Let $\A = (A_s)_{s \in S}$ be a Fell bundle over an inverse semigroup $S$, and let $B$ be any fixed \cstar{}algebra. If $\A$ has the approximation property, then $\A \mintensor B$ and $\A \maxtensor B$ do as well.
\end{proposition}
\begin{proof}
Fix a net $\{\xi_i\}_{i \in I} \sbe \gammac(S, \A)$ witnessing the approximation property of $\A$, and consider the net
\[ \zeta_{i,j}\left(s\right) \coloneqq \xi_i\left(s\right) \otimes u_j, \]
where $\{u_j\}_{j\in J}$ is an approximate unit for $B$. We claim that $\{\zeta_{i,j}\}_{i \in I, j\in J}$ witnesses the approximation property for both $\A \maxtensor B$ and $\A \mintensor B$. In fact, both proofs are identical, and hence we shall restrict ourselves to the proof for $\A \maxtensor B$. First, notice that $\zeta_{i,j} \in \gammac(S, A_1 \odot B)$, since $\supp(\zeta_{i,j}) = \supp(\xi_i)$. In order to avoid confusion, we will denote by $1^{\text{m}}_s$ the unit of the ideal defined as in \cref{def:pre:ideal-s-1} for the bundle $\A \maxtensor B$ (we try to omit naming it). In particular, observe that $1^{\text{m}}_s = 1_s \tensor 1_B$, where $1_B$ denotes the unit of $\M(B) \sbe B''$. We can now check that $\{\zeta_{i,j}\}_{i \in I, j\in J}$ does witness the approximation property for $\A \maxtensor B$. First,
\begin{align*}
\left\|\sum_{p, t \in S} \zeta_{i,j}\left(p\right)^* \zeta_{i,j}\left(t\right) 1^{\text{m}}_{pt^*}\right\| & = \left\| \sum_{p, t \in S} \left(\xi_i\left(p\right)^*\xi_i\left(t\right) 1_{pt^*}\right) \tensor u_j \right\| \\
& \leq \left\| \sum_{p, t \in S} \xi_i\left(p\right)^*\xi_i\left(t\right) 1_{pt^*} \right\| < \infty,
\end{align*}
which proves assertion~\cref{def:fell-isg:bounded} in \cref{def:fell-isg}. Assertion~\cref{def:fell-isg:invar} is proved similarly. First, given any $x_s = \sum_{k = 1}^n a_k \tensor b_k \in A_s \algtensor B$ observe that
\begin{align*}
\sum_{p, t \in S} 1^{\text{m}}_{p\left(st\right)^*} & \zeta_{i,j}\left(p\right)^* x_s \zeta_{i,j}\left(t\right) = \sum_{p, t \in S} \sum_{k = 1}^n 1_{p \left(st\right)^*} \xi_i\left(p\right)^* a_k \xi_i\left(t\right) \tensor u_jb_ku_j \\
& = \sum_{k = 1}^n \left(\sum_{p, t \in S} 1_{p \left(st\right)^*} \xi_i\left(p\right)^* a_k \xi_i\left(t\right) \right) \tensor u_jb_ku_j \xrightarrow{i,j} \sum_{k = 1}^n a_k \tensor b_k = x_s. \\
\end{align*}
Given any general $y_s \in (\A \maxtensor B)_s = A_s \maxtensor B$ the same result follows since $A_s \algtensor B \sbe A_s \maxtensor B$ is dense and $\{\zeta_{i,j}\}_{i,j}$ is uniformly bounded in the sense of \cref{def:fell-isg}~\cref{def:fell-isg:bounded}.
\end{proof}

\section{Nuclearity of the full and reduced cross-sectional \texorpdfstring{$C^*-$}{C*-}algebras}\label{sec:nuclearity}
In this section we prove the two main results of the paper. Firstly, \cref{thm:ap-implies-wk} shows that if a Fell bundle has the approximation property then it also has the weak containment property. Secondly, \cref{thm:ap-nuclear} shows, using the tensor product machinery of \cref{sec:tensor}, that Fell bundles with the weak containment property and nuclear unit fiber give rise to nuclear (full or reduced) cross-sectional \cstar{}algebras. 

We start the discussion with the proof of \cref{thm:ap-implies-wk}, which has two main ingredients. The first is Fell's absorption principle in the context of inverse semigroups, particularly \cref{cor:fell-principle:red}. The other main ingredient is the following lemma, which allows us to create completely positive maps \emph{in the wrong direction}, namely from $\redalg{\A}$ to $\fullalg{\A}$.
\begin{lemma} \label{lemma:ap-implies-wk:maps}
Let $\A = (A_s)_{s \in S}$ be a Fell bundle over the inverse semigroup $S$. Let $\xi \in \gammac(S, \A)$ be such that $\|\sum_{p, t \in S} \xi(p)^*\xi(t) 1_{pt^*}\| < \infty$. Then there are
\begin{enumerate}[label=(\roman*)]
\item a \cstar{}algebra $B$ containing a copy of $\fullalg{\A}$;
\item and a bounded completely positive map $\Psi_\xi \colon \redalg{\A} \rightarrow B$
\end{enumerate}
such that
\[ \Psi_\xi\left(a_t \delta_t\right) \coloneqq \sum_{p, t \in S} 1_{p\left(st\right)^*} \xi\left(p\right)^* a_s \xi\left(t\right) \]
for every $a_s \in A_s$ and $s \in S$. Furthermore, $\|\Psi_\xi\| \leq \|\sum_{p, t \in S} \xi(p)^*\xi(t) 1_{pt^*} \|$.
\end{lemma}
\begin{proof}
Fix a faithful representation $\pi \colon \fullalg{\A} \to \L(\H)$ such that the normal extension $\pi_1'' \colon A_1'' \to \L(\H)$ is also faithful. Let $B \sbe \L(\H)$ be the \cstar{}algebra generated by the image of $\pi$ along with $\pi_1''(A_1'')$ in $\L(\H)$. It is clear that $B$ contains a copy of $\fullalg{\A}$. In addition, let $\ell^2_\pi(S, \H)$ be as in \cref{def:fell-principle:l-two}, that is, $\ell^2_\pi(S, \H)$ is the Hausdorff completion of $\contc^\pi(S, \H)$ with respect to the semi-inner product $\braket{\cdot}{\cdot}_\pi$. Moreover, let $\pi^\Lambda$ be as in \cref{thm:fell-principle}, that is, $\pi^\Lambda_s (a_s \delta_s)(v_t \delta_t) \coloneqq \pi_{st}(a_sa_t) v_t \delta_{st}$. By \cref{cor:fell-principle:red}, the \cstar{}algebra generated by the image of $\pi^\Lambda$ is canonically isomorphic to $\redalg{\A}$, since $\pi$ is assumed to be faithful.

Consider the operator
\[ T_\xi \colon \H \to \ell^2_\pi(S, \H), \;\; \text{where} \;\, v \mapsto \sum_{t \in S} \pi_1\left(\xi\left(t\right)\right)v \delta_t. \]
Observe that for every $v \in \H$ and $w_r \delta_r \in \H_r \delta_r \sbe \ell^2_\pi(S, \H)$, we have
\begin{align*}
\braket{T_\xi v}{w_r \delta_r}_\pi = \sum_{t \in S} \braket{\pi_1\left(\xi\left(t\right)\right) v \delta_t}{w_r \delta_r}_\pi & = \sum_{t \in S} \braket{\pi_1\left(\xi\left(t\right)\right) v}{\pi_1''\left(1_{tr^*}\right) w_r}_\H \\
& = \braket{v}{\sum_{t \in S} \pi_1''\left(\xi\left(t\right)^* 1_{tr^*}\right) w_r}_\H,
\end{align*}
and hence $T_\xi^*(w_r \delta_r) = \sum_{t \in S} \pi_1''(\xi(t)^* 1_{tr^*}) w_r$. In addition, observe that $T_\xi$ is bounded, as
\begin{align*}
\left\|T_\xi v\right\|^2 & = \braket{T_\xi v}{T_\xi v}_\pi = \sum_{t, p \in S} \braket{\pi_1\left(\xi\left(t\right)\right)v \delta_t}{\pi_1\left(\xi\left(p\right)\right)v \delta_p}_\pi \\ 
& = \sum_{t, p \in S} \braket{\pi_1\left(\xi\left(t\right)\right)v}{\pi_1''\left(1_{tp^*} \xi\left(p\right)\right)v}_\H = \braket{v}{\sum_{t, p \in S} \pi_1''\left(\xi\left(t\right)^* 1_{tp^*} \xi\left(p\right)\right)v}_\H,
\end{align*}
and hence $\|T_\xi\|^2 \leq \| \sum_{p, t \in S} \xi(p)^* \xi(t) 1_{pt^*} \| < \infty$ by the assumption on $\xi$. Now consider the map $\Psi_\xi \colon \L(\ell^2_\pi(S, \H)) \to \L(\H)$ given by $\Psi_\xi(x) \coloneqq T_\xi^* x T_\xi$, and observe that $\Psi_\xi$ is bounded, since $T_\xi$ is bounded. Likewise, it is completely positive. Furthermore, observe that for every $v \in \H$
\begin{align*}
\Psi_\xi\left(\pi^\Lambda_s\left(a_s\delta_s\right)\right)v & = \left(T_\xi^* \pi^\Lambda_s\left(a_s \delta_s\right)T_\xi\right) v = \sum_{t \in S} T_\xi^* \pi^\Lambda_s\left(a_s \delta_s\right) \pi_1\left(\xi\left(t\right)\right) v \delta_t \\
& = \sum_{t \in S} T_\xi^* \pi_s\left(a_s \xi\left(t\right)\right) v \delta_{st} = \sum_{p, t \in S} \pi_1''\left(\xi\left(p\right)^* 1_{p\left(st\right)^*} \right) \pi_s\left(a_s \xi\left(t\right)\right) v \\
& = \sum_{p, t \in S} \pi_s''\left(1_{p\left(st\right)^*} \xi\left(p\right)^* a_s \xi\left(t\right)\right) v,
\end{align*}
as desired. In addition, note that $\|\Psi_\xi\| \leq \|T_\xi^* T_\xi\| = \|T_\xi\|^2 < \infty$, as proved above. Lastly, observe that the image of $\Psi$ is contained in $B$, as $1_{r} \in A_1'', \xi(p) \in A_1$ and $a_s \in A_s \delta_s \sbe \fullalg{A}$, all of which are embedded into $B$ via $\pi$.
\end{proof}

\begin{theorem} \label{thm:ap-implies-wk}
Let $\A = (A_s)_{s \in S}$ be a Fell bundle over the inverse semigroup $S$. If $\A$ has the approximation property then it has the weak containment property, i.e., the regular representation $\Lambda \colon \fullalg{\A} \rightarrow \redalg{\A}$ is a \Star{}isomorphism. 
\end{theorem}
\begin{proof}
Let $\{\xi_i\}_{i \in I} \sbe \gammac(S, \A)$ witness the approximation property of $\A$, and let $\Psi_{\xi_i} \colon \redalg{\A} \to B$ be as in \cref{lemma:ap-implies-wk:maps}. In fact, observe that it follows from the proof of \cref{lemma:ap-implies-wk:maps} that the \cstar{}algebra $B$ does \emph{not} depend on $\xi_i$, so we may take it to be generated by $\fullalg{\A}$ and $A_1''$. Consider then maps $\Phi_i \colon \fullalg{\A} \to B$ given by $\Phi_i \coloneqq \Psi_{\xi_i} \circ \Lambda$, and observe they are completely positive and bounded. Moreover, 
\begin{equation}\label{eq:AP-convergence}
\Phi_i\left(a_s\delta_s\right) = \Psi_{\xi_i} \left(\Lambda_s\left(a_s \delta_s\right)\right) = \Psi_{\xi_i}\left(a_s \delta_s\right) = \sum_{p, t \in S} 1_{p\left(st\right)^*} \xi_i\left(p\right)^* a_s \xi_i\left(t\right) \xrightarrow{i} a_s \delta_s
\end{equation}
since $\{\xi_i\}_{i \in I}$ is assumed to witness the approximation property for $\A$. Therefore, since the subspaces $A_t \delta_t$ span $\algalg{\A}$ and this algebra is dense in $\fullalg{\A}$, we obtain that $\Phi_i(x) \to x$ for every $x \in \fullalg{\A}$.

Now, for any $x \in \fullalg{\A}$ in the kernel of $\Lambda$, we have
\[ x = \lim_{i} \Phi_i\left(x\right) = \lim_{i} \Psi_{\xi_i}\left(\Lambda\left(x\right)\right) = 0, \]
which shows that $\Lambda$ is, indeed, injective, and hence a \Star{}isomorphism.
\end{proof}

\begin{remark}
It is apparent from the proof of \cref{thm:ap-implies-wk} that we could start with a net $\{\xi_i\}_{i}$ in $\gammac(S,\A'')$ witnessing the approximation property of $\A''$ in order to get the weak containment of $\A$. Moreover, in this situation it is actually enough to have weak*-convergence in \cref{def:fell-isg:invar} in order to get the conclusion in the proof above, as this would then yield weak convergence instead of norm convergence in \cref{eq:AP-convergence} after taking a representation as in \cref{lemma:ap-implies-wk:maps}.
\end{remark}

We now turn to the proof of \cref{thm:ap-nuclear}. We start by recording the following known general result.
\begin{lemma} \label{lemma:nuclear-fiber:general-fact}
Let $A \sbe B$ is an inclusion of \cstar{}algebras with a weak conditional expectation $P \colon B \to A''$. If $B$ is nuclear, then so is $A$.
\end{lemma}
\begin{proof}
Observe that $P$ extends to a normal conditional expectation $P'' \colon B'' \to A''$. In such case, $B$ is nuclear if, and only if, $B''$ is an injective von Neumann algebra, and this property passes to von Neumann subalgebras which are the image of a conditional expectation.
\end{proof}

\begin{lemma} \label{lemma:tensor-prod-nuclear-fiber}
Let $\A = (A_s)_{s \in S}$ be a Fell bundle over $S$, and let $B$ be any given \cstar{}algebra. If $A_1$ is nuclear, then $\A \mintensor B \cong \A \maxtensor B$.
\end{lemma}
\begin{proof}
We have to show that $A_s \mintensor B \cong A_s \maxtensor B$ for all $s \in S$. By \cref{lemma:tensor-prod:norms} this is immediate once we have $A_e \mintensor B \cong A_e \maxtensor B$ for all idempotents $e \in E$, which holds since, by assumption, $A_1$ is nuclear, and nuclearity is well-known to pass to ideals.
\end{proof}

\begin{theorem} \label{thm:ap-nuclear}
Let $\A = (A_s)_{s \in S}$ be a Fell bundle over an inverse semigroup $S$. Then the following assertions hold:
\begin{enumerate}[label=(\roman*)]
\item \label{thm:ap-nuclear:unit-fiber} If either $\redalg{\A}$ or $\fullalg{\A}$ is nuclear, then so is $A_1$.
\item \label{thm:ap-nuclear:nuclearity} If $A_1$ is nuclear and $\A$ has the approximation property, then $\fullalg{\A} \cong \redalg{\A}$ is nuclear. 
\end{enumerate}
\end{theorem}
\begin{proof}
\cref{thm:ap-nuclear:unit-fiber} follows from \cref{lemma:nuclear-fiber:general-fact}. The proof of~\cref{thm:ap-nuclear:nuclearity} uses the tensor product machinery of \cref{sec:tensor}. Indeed, first observe that, by \cref{thm:ap-implies-wk}, we have $\fullalg{\A} \cong \redalg{\A}$ via the left regular representation, that is, the map induced from the identity map on $\contc(\A)$. Then, for any unital \cstar{}algebra $B$ the identity map on $\contc(\A) \algtensor B$ descends into \Star{}isomorphisms
\begin{align}
\fullalg{\A} \maxtensor B & \cong \fullalg{\A \maxtensor B} & \quad \quad \left(\text{Prop.}~\ref{prop:tensor-prod:full-comp}\right) \nonumber \\
& \cong \fullalg{\A \mintensor B} &  \quad \quad \left(\text{Lem.}~\ref{lemma:tensor-prod-nuclear-fiber}\right) \nonumber \\
& \cong \redalg{\A \mintensor B} & \quad \quad \left(\text{Prop.}~\ref{prop:tensor-prod:ap} \; \text{and} \; \text{Thm.}~\ref{thm:ap-implies-wk}\right) \nonumber \\
& \cong \redalg{\A} \mintensor B &  \quad \quad \left(\text{Prop.}~\ref{prop:tensor-prod:red-comp}\right) \nonumber \\
& \cong \fullalg{\A} \mintensor B, & \quad \quad \left(\text{Thm.}~\ref{thm:ap-implies-wk}\right) \nonumber
\end{align}
as desired.
\end{proof}

An immediate consequence is the following.
\begin{corollary}
If $S$ is a \cstar{}amenable inverse semigroup then $\fullalg{S}\cong\redalg{S}$ via the left regular representation, and both are nuclear \cstar{}algebras.
\end{corollary}

Another consequence of our main results (\cref{prop:wide,prop:fell-isg-vs-groupoids,thm:ap-nuclear}) is the following corollary, which generalizes the main results of~\cite{Kranz:AP-groupooids-Fell-bundles} in several ways. In particular, we do not require the groupoid to be Hausdorff or second-countable, nor do we need the Fell bundle be saturated or separable.
\begin{corollary} \label{cor:groupoids-final}
Let $G$ be an \'{e}tale groupoid with Hausdorff unit space, and let $\A = (A_g)_{g \in G}$ be a Fell bundle over $G$. The following assertions then hold:
\begin{enumerate}[label=(\roman*)]
\item If $\A$ has the approximation property then $\redalg{\A} \cong \fullalg{\A}$ via the left regular representation.
\item If $\A$ has the approximation property and $\contz(\A^0)$ is nuclear then $\redalg{\A} \cong \fullalg{\A}$ is nuclear.
\end{enumerate}
\end{corollary}
\begin{proof}
 This follows as a combination of \cref{thm:ap-implies-wk,thm:ap-nuclear} and the fact that if $S$ is a wide inverse semigroup of bisections of $G$ and $\B$ is the Fell bundle over $S$ associated to $\A$, then there are canonical isomorphisms $\fullalg\A\cong \fullalg\B$ and $\redalg\A\cong \redalg\B$, see \cites{Buss-Meyer:Actions_groupoids,Buss-Exel-Meyer:Reduced}.
\end{proof}

We end the paper with a brief discussion on the \emph{essential \cstar{}algebra} of the Fell bundle $\A$, henceforth denoted by $\essalg{\A}$. This \cstar{}algebra was introduced, in this generality, in the groundbreaking~\cite{Kwaniewski2019EssentialCP}*{Section~4} (see also~\cite{AraMathieu:book} for the definition of the \emph{local multiplier algebra}). For the purposes of this paper, however, the construction of $\essalg{\A}$ is irrelevant, albeit for the fact that it is, canonically, a quotient of $\redalg{\A}$ in the following way.
\begin{theorem}[\cite{Kwaniewski2019EssentialCP}*{Theorem~4.11}] \label{thm:ess-construction}
Given a Fell bundle $\A = (A_s)_{s \in S}$ over $S$, the identity map on $\contc(\A)$ induces canonical quotient maps
\[ \fullalg{\A} \twoheadrightarrow \redalg{\A} \twoheadrightarrow \essalg{\A} \]
that restrict to the identity on $A_1$.
\end{theorem}
As an immediate consequence of the above theorem we get the following.
\begin{corollary} \label{cor:ess-nuclearity}
Let $\A = (A_s)_{s \in S}$ be a Fell bundle over a unital inverse semigroup $S$. If $A_1$ is nuclear and $\A$ has the approximation property, then $\essalg{\A}$ is nuclear.
\end{corollary}
\begin{proof}
The claim follows from \cref{thm:ap-nuclear,thm:ess-construction} and the very deep result of Kirchberg that states that quotients of nuclear \cstar{}algebras are nuclear as well (see, e.g.,~\cite{Brown-Ozawa:Approximations}*{Corollary~9.4.4}).
\end{proof}
Observe that the reverse of \cref{cor:ess-nuclearity} does not hold, that is, that $\essalg{\A}$ may be nuclear and $\A$ not satisfy the approximation property, as the following example shows.
\begin{example} \label{ex:ap:red-vs-ess:2}
Let $G$ be a group, and let $S$ and $\A = (A_s)_{s \in S}$ be as in \cref{ex:ap:red-vs-ess} (see also~\cite{Kwaniewski2019EssentialCP}*{Example~4.7}). By the discussion in \cref{ex:ap:red-vs-ess}, if $G$ is non-amenable then $\A$ does not have the approximation property, and the left regular representation $\Lambda \colon \fullalg{\A} \to \redalg{\A}$ is not injective. However, as stated in~\cite{Kwaniewski2019EssentialCP}*{Example~4.7}, $\essalg{\A}$ is nuclear, regardless of the amenability of $G$. Indeed, in this case $\essalg{\A} \cong \cont[0, 1]$, since $(0, 1] \sbe [0,1]$ is open and dense, and hence $\contz(0, 1] \sbe \cont[0,1]$ is an essential (closed two-sided) ideal.
\end{example}



\begin{bibdiv}
  \begin{biblist}
\bib{Abadie:Tensor}{article}{
  author={Abadie, Fernando},
  title={Tensor products of Fell bundles over discrete groups},
  status={eprint},
  note={\arxiv{funct-an/9712006}},
  date={1997},
}

\bib{Abadie-Buss-Ferraro:Amenability}{article}{
  author={Abadie, Fernando},
  author={Buss, Alcides},
  author={Ferraro, Dami\'an},
  title={Amenability and Approximation Properties for Partial Actions and Fell Bundles},
  journal={Bulletin Brazilian Mathematical Society},
  volume={53},
  pages={173-227},
  date={2022},
  doi={10.1007/s00574-021-00255-8},
}

\bib{Anantharaman-Delaroche:Systemes}{article}{
  author={Anantharaman-Delaroche, Claire},
  title={Syst\`emes dynamiques non commutatifs et moyennabilit\'e},
  journal={Math. Ann.},
  volume={279},
  date={1987},
  number={2},
  pages={297--315},
  issn={0025-5831},
  review={\MR{919508}},
  doi={10.1007/BF01461725},
}

\bib{Renault_AnantharamanDelaroche:Amenable_groupoids}{book}{
  author={Anantharaman-Delaroche, Claire},
  author={Renault, Jean},
  title={Amenable groupoids},
  series={Monographies de L'Enseignement Math\'ematique},
  volume={36},
  publisher={L'Enseignement Math\'ematique, Geneva},
  year={2000},
  pages={196},
  isbn={2-940264-01-5},
  review={\MR{1799683}},
}

\bib{Ara-Exel-Katsura:Dynamical_systems}{article}{
  author={Ara, Pere},
  author={Exel, Ruy},
  author={Katsura, Takeshi},
  title={Dynamical systems of type $(m,n)$ and their $\textup {C}^*$\nobreakdash -algebras},
  journal={Ergodic Theory Dynam. Systems},
  volume={33},
  date={2013},
  number={5},
  pages={1291--1325},
  issn={0143-3857},
  review={\MR {3103084}},
  doi={10.1017/S0143385712000405},
}

\bib{AraLledoMartinez-2020}{article}{
title={Amenability and paradoxicality in semigroups and C*-algebras},
journal={Journal of Functional Analysis},
volume={279},
number={2},
pages={108530},
year={2020},
issn={0022-1236},
doi={https://doi.org/10.1016/j.jfa.2020.108530},
url={https://www.sciencedirect.com/science/article/pii/S0022123620300732},
author={Ara, Pere},
author={Lledó, Fernando},
author={Martínez, Diego},
}

\bib{AraMathieu:book}{book}{
  author={Ara, Pere},
  author={Mathieu, Martin},
  title={Local multipliers of $C^*$\nobreakdash -algebras},
  series={Monographs in Mathematics},
  publisher={Springer},
  place={London},
  date={2003},
  doi={10.1007/978-1-4471-0045-4}
}

\bib{chung-martinez-szakacs-2022}{article}{
  author={Chyuan Chung, Yeong},
  author={Mart\'{i}nez, Diego},
  author={Szak\'{a}cs, N\'{o}ra},
  title={Quasi-countable inverse semigroups as metric spaces, and the uniform Roe algebras of locally finite inverse semigroups},
  status={preprint},
  note={\arxiv {2211.09624}},
  journal={},
  volume={},
  date={2022},
  number={},
  pages={},
  issn={},
  review={},
  doi={},
}

\bib{Brown-Ozawa:Approximations}{book}{
  author={Brown, Nathanial P.},
  author={Ozawa, Narutaka},
  title={$C^*$\nobreakdash -algebras and finite-dimensional approximations},
  series={Graduate Studies in Mathematics},
  volume={88},
  publisher={Amer. Math. Soc.},
  place={Providence, RI},
  date={2008},
  pages={xvi+509},
  isbn={978-0-8218-4381-9},
  isbn={0-8218-4381-8},
  review={\MR {2391387}},
}

\bib{Buss-Echterhoff-Willett:amenability}{article}{
  author={Buss, Alcides},
  author={Echterhoff, Siegfried},
  author={Willett, Rufus},
  title={Amenability and weak containment for actions of locally compact groups on C*-algebras},
  status={preprint},
  note={\arxiv {2003.03469}},
  date={2020},
}

\bib{BussExel:InverseSemigroupExpansions}{article}{
  author={Buss, Alcides},
  author={Exel, Ruy},
  title={Inverse semigroup expansions and their actions on \(C^*\)\nobreakdash -algebras},
  journal={Illinois J. Math.},
  volume={56},
  date={2012},
  number={4},
  pages={1185--1212},
  issn={0019-2082},
  eprint={http://projecteuclid.org/euclid.ijm/1399395828},
  review={\MR {3231479}},
}

\bib{BussExel:Fell.Bundle.and.Twisted.Groupoids}{article}{
  author={Buss, Alcides},
  author={Exel, Ruy},
  title={Fell bundles over inverse semigroups and twisted \'etale groupoids},
  journal={J. Operator Theory},
  volume={67},
  date={2012},
  number={1},
  pages={153--205},
  issn={0379-4024},
  review={\MR{2881538}},
  eprint={http://www.theta.ro/jot/archive/2012-067-001/2012-067-001-007.html},
}

\bib{Buss-Exel-Meyer:Reduced}{article}{
  author={Buss, Alcides},
  author={Exel, Ruy},
  author={Meyer, Ralf},
  title={Reduced \(C^*\)\nobreakdash -algebras of Fell bundles over inverse semigroups},
  journal={Israel J. Math.},
  date={2017},
  volume={220},
  number={1},
  pages={225--274},
  issn={0021-2172},
  review={\MR {3666825}},
  doi={10.1007/s11856-017-1516-9},
}

\bib{Buss-Meyer:Actions_groupoids}{article}{
  author={Buss, Alcides},
  author={Meyer, Ralf},
  title={Inverse semigroup actions on groupoids},
  journal={Rocky Mountain J. Math.},
  issn={0035-7596},
  date={2017},
  volume={47},
  number={1},
  pages={53--159},
  doi={10.1216/RMJ-2017-47-1-53},
  review={\MR {3619758}},
}

\bib{Echterhoff-Kaliszewski-Quigg-Raeburn:Categorical}{article}{
  author={Echterhoff, Siegfried},
  author={Kaliszewski, Steven P.},
  author={Quigg, John},
  author={Raeburn, Iain},
  title={A categorical approach to imprimitivity theorems for $C^*$\nobreakdash-dynamical systems},
  journal={Mem. Amer. Math. Soc.},
  volume={180},
  date={2006},
  number={850},
  pages={viii+169},
  issn={0065-9266},
  review={\MR{2203930}},
  doi={10.1090/memo/0850},
}

\bib{Exel:TwistedActions}{article}{
title={Twisted partial actions: a classification of regular C*-algebraic bundles},
volume={74},
DOI={10.1112/S0024611597000154},
number={2},
journal={Proceedings of the London Mathematical Society},
publisher={Cambridge University Press},
author={Exel, Ruy},
year={1997},
pages={417–443},
}

\bib{Exel:Amenability}{article}{
  author={Exel, Ruy},
  title={Amenability for Fell bundles},
  journal={J. Reine Angew. Math.},
  volume={492},
  date={1997},
  pages={41--73},
  issn={0075-4102},
  review={\MR {1488064}},
  doi={10.1515/crll.1997.492.41},
}

\bib{Exel:Inverse_combinatorial}{article}{
  author={Exel, Ruy},
  title={Inverse semigroups and combinatorial $C^*$\nobreakdash-algebras},
  journal={Bull. Braz. Math. Soc. (N.S.)},
  volume={39},
  date={2008},
  number={2},
  pages={191--313},
  issn={1678-7544},
  review={\MR{2419901}},
  doi={10.1007/s00574-008-0080-7},
}

\bib{Exel:noncomm.cartan}{article}{
  author={Exel, Ruy},
  title={Noncommutative Cartan subalgebras of $C^*$\nobreakdash -algebras},
  journal={New York J. Math.},
  issn={1076-9803},
  volume={17},
  date={2011},
  pages={331--382},
  eprint={http://nyjm.albany.edu/j/2011/17-17.html},
  review={\MR {2811068}},
}

\bib{Exel:Partial_dynamical}{book}{
  author={Exel, Ruy},
  title={Partial dynamical systems, Fell bundles and applications},
  series={Mathematical Surveys and Monographs},
  volume={224},
  date={2017},
  pages={321},
  isbn={978-1-4704-3785-5},
  isbn={978-1-4704-4236-1},
  publisher={Amer. Math. Soc.},
  place={Providence, RI},
  review={\MR {3699795}},
}

\bib{ExelStarling:Amenable_Actions_Inv_Sem}{article}{
  author={Exel, Ruy},
  author={Starling,Charles},
  title={Amenable actions of inverse semigroups},
  journal={Ergod. Theor. Dyn. Syst.},
  volume={37},
  date={2017},
  pages={481--489},
}

\bib{ExelNg:ApproximationProperty}{article}{
  author={Exel, Ruy},
  author={Ng, {Ch}i-Keung},
  title={Approximation property of $C^*$\nobreakdash -algebraic bundles},
  journal={Math. Proc. Cambridge Philos. Soc.},
  volume={132},
  date={2002},
  number={3},
  pages={509--522},
  issn={0305-0041},
  doi={10.1017/S0305004101005837},
  review={\MR {1891686}},
}

\bib{KellendonkLawson:PartialActions}{article}{
  author={Kellendonk, Johannes},
  author={Lawson, Mark V.},
  title={Partial actions of groups},
  journal={Internat. J. Algebra Comput.},
  volume={14},
  date={2004},
  number={1},
  pages={87--114},
  issn={0218-1967},
  doi={10.1142/S0218196704001657},
  review={\MR{2041539}},
}

\bib{Kirchberg-Wassermann:Operations}{article}{
  author={Kirchberg, Eberhard},
  author={Wassermann, Simon},
  title={Operations on continuous bundles of $C^*$\nobreakdash -algebras},
  journal={Math. Ann.},
  volume={303},
  date={1995},
  number={4},
  pages={677--697},
  issn={0025-5831},
  review={\MR {1359955}},
  doi={10.1007/BF01461011},
}

\bib{Kwaniewski2019EssentialCP}{article}{
  title={Essential crossed products for inverse semigroup actions: simplicity and pure infiniteness},
  author={Kwaśniewski, Bartosz Kosma},
  author={Meyer, Ralf},
  journal={Documenta Mathematica},
  year={2019}
}

\bib{Kwasniewski-Meyer:Pure_infiniteness}{article}{
  author={Kwa\'sniewski, Bartosz Kosma},
  author={Meyer, Ralf},
  title={Ideal structure and pure infiniteness of inverse semigroup crossed products},
  note={\arxiv{2112.07420.pdf}},
  date={2021},
}

\bib{Kranz:AP-groupooids-Fell-bundles}{article}{
  author={Kranz, Julian},
  title={Amenability for actions of \'etale groupoids on C*-algebras and Fell bundles},
  date={2022},
  note={to appear in Trans. Ame. Math. Soc.},
}

\bib{Kumjian:Fell_bundles}{article}{
  author={Kumjian, Alex},
  title={Fell bundles over groupoids},
  journal={Proc. Amer. Math. Soc.},
  volume={126},
  date={1998},
  number={4},
  pages={1115--1125},
  issn={0002-9939},
  review={\MR{1443836}},
  doi={10.1090/S0002-9939-98-04240-3},
}

\bib{MR3383622}{article}{
  author={LaLonde, Scott M.},
  title={Nuclearity and exactness for groupoid crossed products},
  journal={J. Operator Theory},
  volume={74},
  date={2015},
  number={1},
  pages={213--245},
  issn={0379-4024},
  review={\MR {3383622}},
  doi={10.7900/jot.2014jun06.2032},
}

 \bib{Lawson:InverseSemigroups}{book}{
  author={Lawson, Mark V.},
  title={Inverse semigroups: the theory of partial symmetries},
  publisher={World Scientific Publishing Co.},
  place={River Edge, NJ},
  date={1998},
  pages={xiv+411},
  isbn={981-02-3316-7},
}

\bib{Lawson2012:duality}{article}{
  author={Lawson, Mark},
  title={Noncommutative Stone duality: inverse semigroups, topological groupoids and C*-algebras},
  journal={International Journal of Algebra and Computation},
  volume={22},
  number={6},
  date={2012},
  doi={10.1142/S0218196712500580},
}

\bib{Li:Classifiable}{article}{
author={Li, Xin},
year={2020},
title={Every classifiable simple C*-algebra has a Cartan subalgebra},
journal={Inventiones Mathematicae},
volume={219},
number={2},
pages={653--699},
doi={10.1007/s00222-019-00914-0}
}

\bib{lledo-martinez-2021}{article}{
  author={Lled\'{o}, Fernando},
  author={Mart\'{i}nez, Diego},
  title={The uniform Roe algebra of an inverse semigroup},
  journal={J. Math. Anal. Appl.},
  volume={499},
  date={2021},
  number={1},
  pages={124996},
  issn={124996},
  review={\MR {MR4207321}},
  doi={https://doi.org/10.1016/j.jmaa.2021.124996},
}

\bib{martinez-2022}{article}{
  author={Mart\'{i}nez, Diego},
  title={A note on the quasi-diagonality of an inverse semigroup reduced C*-algebras},
  journal={J. Op. Th.},
  volume={},
  date={2022},
  number={},
  pages={},
  issn={},
  review={},
  doi={},
}

\bib{Ozawa-Suzuki:Amenable_examples}{article}{
  author={Ozawa, Narutaka},
  author={Suzuki, Yuhei},
  title={On characterizations of amenable \(\mathrm{C}^*\)-dynamical systems and new examples},
  journal={Selecta Math. (N.S.)},
  issn={1022-1824},
  volume={27},
  date={2021},
  pages={92},
  doi={10.1007/s00029-021-00699-2},
}

\bib{Paterson1999}{book}{
  author={Paterson, Alan L.~T.},
  title={Groupoids, Inverse Semigroups, and their Operator Algebras},
  publisher={Springer Verlag},
  year={1999},
}

\bib{Preston1954-1}{article}{
author={Preston, Gordon Bamford},
title={Inverse semi-groups},
volume={29},
journal={J. London Math. Soc.},
date={1954},
pages={396--403},
}

\bib{Preston1954-2}{article}{
author={Preston, Gordon Bamford},
title={Inverse semi-groups with minimal right ideals},
volume={29},
journal={J. London Math. Soc.},
date={1954},
pages={404--411},
}

\bib{Preston1954-3}{article}{
author={Preston, Gordon Bamford},
title={Representations of inverse semi-groups},
volume={29},
journal={J. London Math. Soc.},
date={1954},
pages={411--419},
}

\bib{SiebenFellbundles}{article}{
  author={Sieben, Nandor},
  title={Fell bundles over r-discrete groupoids and inverse semigroups},
  status={unpublished},
  date={1998},
  eprint={http://jan.ucc.nau.edu/~ns46/bundle.ps.gz},
}

\bib{SimsNotes2020}{article}{
  Author={Sims, Aidan},
  Booktitle={Operator algebras and dynamics: groupoids, crossed products, and Rokhlin dimension},
  Editor={Perera, Francesc},
  doi={10.1007/978-3-030-39713-5},
  Organization={CRM Barcelona},
  Publisher={Birkh\"auser/Springer},
  Series={Advanced Courses in Mathematics},
  Title={Hausdorff \'etale groupoids and their {C*-algebras}},
  Year={2020},
 }

\bib{Wagner1952}{article}{
author={Wagner, Viktor},
title={Generalised groups},
volume={84},
journal={Proc. USSR Ac. Scn.},
date={1952},
pages={1119--1122},
}

\bib{Wagner1953}{article}{
author={Wagner, Viktor},
title={The theory of generalized heaps and generalized groups},
volume={3},
journal={Sb. Mat.},
date={1953},
pages={545--632},
}

  \end{biblist}
\end{bibdiv}
\end{document}